\documentclass[preprint,10pt]{elsarticle}

\usepackage{amssymb}
\usepackage{amsmath,mathrsfs}
\usepackage[all,cmtip]{xy}
\usepackage{latexsym}
\usepackage{mathtools}
\usepackage{amsthm}
\usepackage{a4wide}
\usepackage{tikz}
\usetikzlibrary{arrows}
\usepackage{geometry}
\usepackage{array}
\usepackage{verbatim}
\DeclareMathAlphabet{\EuRoman}{U}{eur}{sb}{n}       
\SetMathAlphabet{\EuRoman}{bold}{U}{eur}{sb}{n}   
\usepackage{yfonts}         	
\usepackage{mathabx}		
\usepackage{euscript}
\usepackage[T1]{fontenc}
\usepackage{graphicx}
\usepackage{xcolor}

\journal{}

\theoremstyle{plain}
  \newtheorem{thm}{Theorem}[section]
  \newtheorem{lem}[thm]{Lemma}
  \newtheorem{prop}[thm]{Proposition}
  \newtheorem{cor}[thm]{Corollary}
  \newtheorem{exa}[thm]{Example}
\theoremstyle{definition}
  \newtheorem{defn}[thm]{Definition}

  \newtheorem{rem}[thm]{Remark}

\makeatletter
\def\ps@pprintTitle{%
 \let\@oddhead\@empty
 \let\@evenhead\@empty
 \def\@oddfoot{\centerline{\thepage}}%
 \let\@evenfoot\@oddfoot}
\makeatother

\usepackage[colorlinks=true,linkcolor=black, citecolor=blue, urlcolor=blue,bookmarks=true]{hyperref}
\newcommand{\from}{\colon}
\newcommand{\DefEq}{\coloneq}
\newcommand{\Defn}[1]{{\em #1}}
\DeclareMathOperator*{\colim}{colim}
\newcommand{\CoLim}{\mathrm{colim}}
\newcommand{\CoLimOfOver}[2]{\mathrm{colim}^{#2}(#1)}
\newcommand{\LimOf}[1]{\mathrm{lim}(#1)}
\newcommand{\OneMapOn}[1]{\mathrm{1}_{#1}}
\newcommand{\XRA}[1]{\xrightarrow{\ #1\ }}
\newcommand{\Ra}{\Rightarrow}
\newcommand{\lra}{\longrightarrow}
\newcommand{\Lra}{\Longrightarrow}
\newcommand{\Lla}{\Longleftarrow}
\newcommand{\al}{\alpha}
\newcommand{\be}{\beta}
\newcommand{\ep}{\epsilon}
\newcommand{\ga}{\gamma}
\newcommand{\lam}{\lambda}
\newcommand{\si}{\sigma}

\newcommand{\CB}{\mathcal{B}}
\newcommand{\CC}{\mathcal{C}}
\newcommand{\CD}{\mathcal{D}}
\newcommand{\CE}{\mathcal{E}}
\newcommand{\CF}{\mathcal{F}}
\newcommand{\CG}{\mathcal{G}}

\newcommand{\CI}{\mathcal{I}}
\newcommand{\CJ}{\mathcal{J}}
\newcommand{\CK}{\mathcal{K}}
\newcommand{\CL}{\mathcal{L}}

\newcommand{\CO}{\mathcal{O}}

\newcommand{\CS}{\mathcal{S}}

\newcommand{\CX}{\mathcal{X}}
\newcommand{\CY}{\mathcal{Y}}
\newcommand{\Cat}{{\sf Cat}}
\newcommand{\CAT}{{\sf CAT}}
\newcommand{\co}{{\mathrm co}}
\newcommand{\op}{\mathrm{op}}
\newcommand{\CatSymbA}[1]{\mathcal{#1}}
\newcommand{\CatSymbB}[1]{#1}
\newcommand{\SmallCats}{\EuRoman{C{\kern-0.17ex}a{\kern-0.12ex}t}}
\newcommand{\LocSmallCats}{\EuRoman{C{\kern-0.08ex}A{\kern-0.25ex}T}}
\newcommand{\Sets}{\EuRoman{S{\kern-0.12ex}e{\kern-0.07ex}t}}
\newcommand{\CleavageOf}[2]{\theta^{#1}_{#2}}

\newcommand{\GrothendieckCFibOf}[1]{\Pi^{#1}}
\begin{document}

\begin{frontmatter}



\title{Diagrams, Fibrations, and the Decomposition of Colimits}


\author{George Peschke}
\ead{gepe@ualberta.ca}
\address{Dept. of Mathematical and Statistical Sciences, University of Alberta, Edmonton, Alberta, Canada, T6G 2G1}
\author{Walter Tholen\fnref{A}}

\ead{tholen@yorku.ca}

\address{Department of Mathematics and Statistics, York University, Toronto, Ontario, Canada, M3J 1T6}
\address{\em }

\fntext[A]{The second author's research is supported by a Discovery Grant (no. 501260) of the Natural Sciences and Engineering Research Council of Canada (NSERC).}

\begin{abstract}
\indent The contributions of this paper are twofold. Within the framework of Grothendieck's fibrational category theory, we present a web of fundamental 2-adjunctions surrounding the formation of the category of all small diagrams in a given category and the formation of the Grothendieck category a functor into the category of small categories. We demonstrate the utility of these adjunctions, in part by deriving three formulae for (co-)limits: a `twisted' generalization of the well-known Fubini formula, as first established by Chach\'{o}lski and Scherer; a new `general colimit decomposition formula'; and a special case of the general formula, which actually initiated this work, and which was proved independently by Batanin and Berger. We give three proofs for this colimit decomposition formula, using methods that provide quite distinct insights.

The `base' of our web of 2-adjunctions extends earlier work of the Ehresmann school and Guitart and promises to be of independent interest. It involves forming  the diagram category of an arbitrary functor, seen as an object of the arrow category of the category of locally small categories, rather than that of a mere category. The left adjoint of the emerging generalized Guitart 2-adjunction factors through the 2-equivalence of split Grothendieck (co-)fibrations and strictly (co-)indexed categories, which we present here most generally by allowing 2-dimensional variation in the base categories. 

Two appendices to this work spell out in detail all needed tools of fibrational category theory, along with an easily applicable presentation of the formation of quotient categories and a novel criterion for the confinality of their projection functors.
\end{abstract}


\end{frontmatter}

\newpage
{\footnotesize
\tableofcontents
}


\section{Introduction} 

\subsection{Motivation} An initial goal of this work was the development of techniques that may facilitate the computation of ``complicated'' objects, say in algebra or topology, from smaller and more easily computed ``pieces''. In fact, much of this work emanated from a colimit formula established early on, saying that a colimit of a diagram $X:\CK\to\CX$ in a cocomplete category $\CX$ may be computed ``piecewise'', whenever the small category $\CK$ is itself expressed as $\CK=\CoLim \Phi$ in the category $\SmallCats$ of small categories, for some functor $\Phi\from \CD\to\SmallCats$. That is, if $K_d:\Phi d\to \CK\,(d\in\CD)$ denotes the colimit cocone of $\Phi$, one has the following {\em Colimit Decomposition Formula}:
\begin{equation}\label{eqn:CDF}\tag{CDF}
\CoLimOfOver{X}{\CK} \cong \CoLimOfOver{\CoLimOfOver{XK_d}{\Phi d}}{d\in \CD}.
\end{equation}
\begin{center}
$\xymatrix{\Phi d\ar[r]^{K_d\qquad}\ar@{..}[d] & {\rm colim}\,\Phi=\CK\ar[r]^{\qquad X} & \CX\\
\Phi e\ar[ru]_{K_e} & & \\
}$
\hfil$\xymatrix{X(K_dx)\ar[r]\ar@{..}[d] & {\rm colim}\,XK_d\ar[r]\ar@{..}[d] & {\rm colim}\,X\\
X(K_dy)\ar[ru] & {\rm colim}\,XK_e\ar[ru] & \\
}$	
\end{center}
(Note that this visualization of the formula must not be misinterpreted as a kind of preservation of a colimit by $X$: the arrows on the left live in $\LocSmallCats$, while those on the right are in $\CX$.) 

The CDF (\ref{decomp formula})    and its limit analogue (\ref{decomposition dual}) were presented by the first author in May 2017 at the {\em Workshop on Categorical Methods in Non-Abelian Algebra} in Louvain-la-Neuve. The argument consisted of brute-force constructions in the category ${\sf Diag}^{\circ}(\CX)$ of small diagrams in $\CX$, see (\ref{diagramcats}), with a strong suspicion that fundamental properties of the evident functor ${\sf Diag}^{\circ}(\CX)\to \SmallCats$ would give a transparent and conceptual explanation. From this suspicion the present collaboration evolved.
After a talk presented by the second author at CT2019, it turned out that the CDF had actually appeared slightly earlier than the first author's 2017 presentation, as Lemma 7.13 in the paper \cite{BataninBerger} by Michael Batanin and Clemens Berger, who give credit to Steve Lack for the short proof they present. 

We found that Grothendieck's {\em fibrational category theory} provides a good framework for our purposes, and that it yields insights on categories of diagrams which are of interest in their own right. As hoped for, we obtain transparent and conceptual proofs for the CDF and its variations, see Sections \ref{sec:Twisted-Fubini} and  \ref{sec:CDF-Proofs}. Here is an outline.


\subsection{The Grothendieck construction and the Guitart adjunction}
Consider a Grothendieck fibration $P:\CE\to\CB$. By assigning to every object $b$ of the base category $\CB$  its fibre $\CE_b$ in $\CE$ one obtains a pseudo-functor $\CB^{\op}\to\LocSmallCats$ into the huge 2-category of all categories (see the end of this Introduction for notational conventions with respect to $\LocSmallCats$ and $\SmallCats$). Conversely, the Grothendieck construction produces for every pseudo-functor $\Phi:\CB^{\op}\to\LocSmallCats$ the total category of $\Phi$, here denoted by $\int^{\circ}\Phi$, which is fibred over $\CB$. In this way, fibred categories over $\CB$ are equivalently presented as contravariantly indexed categories, that is: as contravariant $\LocSmallCats$-valued pseudo-functors defined on $\CB$. Dually, {\em Grothendieck cofibrations} (nowadays more frequently called {\em opfibrations}\footnote{In this paper we adhere to Grothendieck's original terminology, but note that, in order to prevent confusion with the terminology used in topology, in recent years Grothendieck cofibrations have more commonly been referred to as opfibrations. However, this departure from the standard categorical procedure for dualizing terms causes further deviations from established standards, such as ``opcartesian" versus ``cocartesian". See also \cite{LoregianRiehl2019} for a brief discussion of this terminological difficulty .} 
) $\CE\to\CB$ correspond equivalently to covariantly indexed categories $\Phi:\CB\to\LocSmallCats$, via the dual Grothendieck construction, here denoted by $\int_{\circ}\Phi$. Under this equivalence, so-called split cofibrations correspond, by definition, to those pseudo-functors that are actually functors, and one may further restrict the equivalence to small-fibred split cofibrations and $\SmallCats$-valued (as opposed to $\LocSmallCats$-valued) functors.

These facts are well known and amply documented in the literature, albeit predominantly in a context that leaves the base $\CB$ fixed; in historical order, references include \cite{Grothendieck1961}, \cite{Giraud}, \cite{Gray1966}, \cite{Benabou1975} \cite{JohnstonePare1978}, \cite{Street1980},    \cite{Benabou1985}, \cite{Borceux1994}, \cite{JanelidzeTholen1994}, \cite{Fusco}, \cite{Jacobs1999}, \cite{Johnstone2002}, \cite{Streicher}, \cite{RosickyTholen2007}, \cite{Vistoli}, \cite{JohnsonYau2020}. Less known is the fact that the Grothendieck construction was studied very early on by Ehresmann \cite{Ehresmann1965} and his school, under the name {\em produit crois\'{e}}. In particular, Guitart (see \cite{Guitart1973}, \cite{Guitart1974}, \cite{GuitartVdB1977}) showed that the assignment $\Phi\longmapsto\int_{\circ}\Phi$ leads to a left-adjoint functor
$$\textstyle{\int}_{\circ}:\LocSmallCats/\SmallCats\lra\LocSmallCats$$
whose right adjoint, here denoted by ${\sf Diag}^{\circ}$, deserves independent interest. It assigns to a category $\CX$ the category ${\sf Diag}^{\circ}(\CX)$ of all small diagrams in $\CX$ which, curiously enough, may be thought of as arising via the Grothendieck construction. Indeed, with the functor $[\Box,\CX]=\CX^{\Box}:\SmallCats^{\op}\to\LocSmallCats$ one has
$${\sf Diag}^{\circ}(\CX)=\textstyle{\int}^{\circ}\CX^{\Box},$$
to be considered as a (fibred) category over $\SmallCats$; explicitly, a morphism $(F,\varphi):(\CI,X)\to (\CJ,Y)$ with small categories $\CI,\CJ$ is given by a lax-commutative diagram
\begin{center}
$\xymatrix{\CI\ar[rr]^F\ar[dr]_X^{\;\;\varphi:\Lra} & & \CJ\ar[dl]^{Y}\\
& \CX}$ 
\end{center}
in $\LocSmallCats$.

Since one actually has an adjunction of 2-functors (\ref{Guitart})  , the first nice, but immediate, consequence of the Guitart adjunction is the known characterization of the Grothendieck category $\int_{\circ}\Phi$ as a lax colimit of $\Phi:\CB\to\LocSmallCats$ in $\LocSmallCats$ (\ref{universalchar dualGr}).
Much less known, or expected, is the second consequence that we draw from the Guitart adjunction, as it gives us the connection with the CDF. Considering the functor $\Phi$ again as a diagram in $\Cat$ (and, thus, renaming the category $\CB$ as $\CD$ from this viewpoint), a diagram $T:\int_{\circ}\Phi\to\CX$ in a category $\CX$ corresponds to a $\CD$-shaped diagram in ${\sf Diag}^{\circ}(\CX)$, given by a family $T(d,-):\Phi d\to X\,(d\in\CD)$ of diagrams in $\CX$. If all these have colimits in $\CX$, then one has the
{\em Twisted Fubini Colimit Formula} (\ref{newFubini}):
\begin{equation}\tag{TFCF}
{\rm colim}^{(d,x)\in{\int}_{\circ}\!\Phi}\;T(d,x)\;\cong\;{\rm colim}^{d\in\CD}({\rm colim}^{x\in\Phi d}\;T(d,x)),
\end{equation}
with the colimit on either side existing if the one on the other side exists. This formula appears in the Appendix of the {\em Memoir} \cite{WChacholskiJScherer2002} by Wojciech Chach\'{o}lski and J\'{e}r$\hat{\rm o}$me Scherer. In Section \ref{sec:CDF-Proofs} we explain how the TFCF implies the CDF, using a general confinality criterion (\ref{quotientscofinal}) for quotient functors (= regular epimorphisms in $\LocSmallCats$).

\subsection{A network of global 2-adjunctions}
It seems peculiar that, in the Guitart adjunction, the left adjoint $\int_{\circ}$ does not keep track of the fact that, for the $\SmallCats$-valued functor $\Phi$, the total category $\int_{\circ}\Phi$ actually lives over $\SmallCats$. However, not forgetting this important fact, and still maintaining an adjunction, means that we should extend the functor ${\sf Diag}^{\circ}$, so that it operates not just on categories, but also on functors, generally to be considered as objects of the arrow category $\LocSmallCats^{\EuRoman{2}}$. Accordingly, a major undertaking in this paper is the replacement of $\LocSmallCats$ by $\LocSmallCats^{\EuRoman{2}}$ as the domain of the Guitart adjunction, as depicted on the right of the diagram
\begin{center}
$\xymatrix{\LocSmallCats\ar@/^0.5pc/[rr]^{!_{\Box}} & \top & \LocSmallCats^{\EuRoman 2}\ar@/^0.5pc/[rr]^{\sf Diag^{\circ}\quad}\ar@/^0.5pc/[ll]^{\rm Dom} & \top & \LocSmallCats/\SmallCats\;,\ar@/^0.5pc/[ll]^{\int_{\circ}\quad}
}$
\end{center}
If we compose this extended 2-adjunction on the right (\ref{fundamental}) with the rather trivial 2-adjunction on the left, one obtains back the original Guitart adjunction. The significance of the extended Guitart adjunction on the right is that it communicates well with the Grothendieck equivalence (\ref{GrothendieckEquivalence})  between split cofibrations with small fibres and $\SmallCats$-valued functors. In fact, the $2$-functor $\int_{\circ}$ above factors as
\begin{center}
$\xymatrix{\LocSmallCats^{\EuRoman{2}} & \EuRoman{SCoFIB}_{\rm sf}\ar[l]_{\rm incl} & \LocSmallCats//\SmallCats\ar[l]_{\int_{\circ}}^{\simeq} & \LocSmallCats/\SmallCats\ar[l]_{\rm incl}\\
},$
	\end{center}
	with the (non-full) subcategory $\EuRoman{SCoFIB}_{\rm sf}$ of small-fibred split cofibrations and their cleavage-preserving morphisms in $\LocSmallCats^{\EuRoman{2}}$, and with the left-adjoint (non-full) inclusion functor (\ref{comma})  of the comma category $\LocSmallCats/\SmallCats$ into the lax comma category $\LocSmallCats//\SmallCats$, all considered as 2-categories. (Objects and morphisms of $\LocSmallCats//\SmallCats$ are defined as in ${\sf Diag}^{\circ}(\SmallCats)$; just specialize $\CX$ to $\SmallCats$ above, except that there is no smallness requirement for $\CI,\CJ$.)
	
	Of course, dropping the requirement of small-fibredness, we may factor the 2-functor $\int_{\circ}$ of the extended Guitart adjunction also as
\begin{center}
$\xymatrix{\LocSmallCats^{\EuRoman{2}} & \EuRoman{SCoFIB}\ar[l]_{\rm incl} & \LocSmallCats//\LocSmallCats\ar[l]_{\int_{\circ}}^{\simeq} & \LocSmallCats/\SmallCats\ar[l]_{\rm incl}\\
}.$
	\end{center}
	Now the full inclusion on the left has become a right adjoint, with its left adjoint producing the {\em free split cofibration} generated by an arbitrary functor (\ref{Reducing dimension}). Either way, the extended Guitart adjunction factors through the classical Grothendieck equivalence between (certain) cofibrations and (certain) indexed categories.
	
\subsection{Organization of the paper and terminological conventions}
 To make the paper self-contained and to introduce notation, we collect background material on fibrational category theory and the Grothendieck construction in Appendix 1 (Section 9). In doing so, we clarify some essential details which don't seem to be documented explicitly in the literature. In Appendix 2 (Section 10) we present a characterization of regular epimorphisms in $\LocSmallCats$ (likely to be known in this or some similar form) and a confinality criterion for such functors (likely to be new), of which we make essential use in one of the proofs for the CDF.
 
 Section 2 introduces diagram categories, shows how to compute colimits and limits in them, identifies a colimit in $\SmallCats$ that is generic for the computation of colimits in terms of coproducts and coequalizers, and finally presents the Guitart adjunction, with the characterization of the Grothendieck construction as a lax colimit in $\LocSmallCats$ following from it. The Twisted Fubini Colimit Formula appears in Section 3, followed by three independent proofs for the Colimit Decomposition Formula in Section 4. The first one follows \cite{BataninBerger}, the second is based on the new \emph{General Colimit Decomposition Formula} \ref{general decomp formula}, and the third one uses the TFCF.
 
 In Sections 5--7 we present successively the extended Guitart adjunction, the Grothendieck equivalence, and finally their global interactions. In doing so, we restrict ourselves to the consideration of split (co)fibrations and functorially indexed categories, but take full account of the 2-categorical nature of the global correspondences. However, in the supplementary Section 8 we briefly show how diagram categories may be considered as 2-(co)fibered categories over $\SmallCats$. 
 The expectation is that there are richer or higher-dimensional contexts (such as those recently considered in \cite{Vasilakopoulou2018}, \cite{MoellerVasilakopoulou}, \cite{MoellerVasilakopoulou}, \cite{Lurie2009}, \cite{LoregianRiehl2019}, \cite{RiehlVerity2020}), in which this network of adjunctions, as well as the the decomposition formulae, may be established.

 Throughout the paper, the term {\em category} refers to an ordinary category, also called 1-category; when a category carries a higher-dimensional structure and is considered with it, we will say so explicitly. Categories may be large (so that their objects may form a proper class), but they are always assumed to be {\em locally small}, so that their hom-functors take values in $\Sets$. Categories whose object class is a set are called {\em small}. 
 $\SmallCats$ denotes the category of small categories, and $\LocSmallCats$ is the {\em huge category} of all (1-)categories, which contains $\Sets$ and $\SmallCats$ as particular objects. The huge category of classes and their maps is denoted by $\EuRoman{SET}$. These casual conventions may be made more precise and justified through the provision of Grothendieck universes.
 
{\em Acknowledgements.} We thank the organizers of the {\em Workshop on Categorical Methods in Non-Abelian Algebra}, Louvain la Neuve,
June 1-3 2017 for their hospitality and for providing a fruitful environment from which this work evolved. The second author thanks also Paolo Perrone for very helpful comments on parts of this work.

\section{Diagram categories}
\label{sec:DiagCats-GuitartAdjunction}

\subsection{Two types of diagram categories of a given category}
Let $\CX$ be a category. Since every small category $\CI$ is, via the formation of the functor category $\CX^{\CI}=[\CI,\CX]=\CAT(\CI,\CX)$, exponentiable in $\CAT$, one has the (internal hom-)functor
\begin{equation*}
\CX^{\Box}=[\Box,\CX]:\Cat^{\rm op}\to \CAT,\; (F:\CI\to\CJ)\longmapsto(F^*:\CX^{\CJ}\to\CX^{\CI},\;Y\mapsto YF).
\end{equation*}
Applying both, the Grothendieck construction and the dual Grothendieck construction (see Section \ref{GrothendieckConstruction}) to $\CX^{\Box}$ yields
 the {\em diagram categories} %
 \footnote{In the literature one finds the notation $\SmallCats//\CX$ for either of these categories. We use it later on in special cases.}%
\begin{equation*}
{\sf Diag}^{\circ}(\CX)=\textstyle{\int}^{\circ}\CX^{\Box}\qquad{\text{ and }\qquad {\sf Diag}_{\circ}(\CX)}=\textstyle{\int}_{\circ}\CX^{\Box}\ .
\end{equation*}
For the reader's convenience, here is a detailed description:

\begin{defn}\label{diagramcats}
The objects of the category ${\sf Diag}^{\circ}(\CX)$ are pairs $(\CI,X)$ with $X\from \CI\to\CX$ a functor of a small category $\CI$; a morphism $(F,\varphi):(\CI,X)\to(\CJ,Y)$ is given by a functor $F:\CI\to\CJ$ of small categories and a natural transformation $\varphi:X\to YF$, as depicted on the left of the diagram 
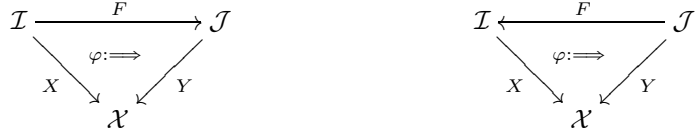
\begin{center}
$\xymatrix{\CI\ar[rr]^F\ar[dr]_X^{\;\;\varphi:\Lra} & & \CJ\ar[dl]^{Y}\\
& \CX}$ \hfil $\xymatrix{\CI\ar[dr]_X^{\;\;\varphi:\Lra} & & \CJ\ar[ll]_F\ar[dl]^{Y}\\
& \CX}$\end{center}
The category ${\sf Diag}_{\circ}(\CX)$ has the same objects as ${\sf Diag}^{\circ}(\CX)$, but a morphism $(F,\varphi):(\CI,X)\to(\CJ,Y)$ in ${\sf Diag}_{\circ}(\CX)$
is now given by a functor $F:\CJ\to\CI$ and a natural transformation $\varphi:XF\to Y$, as depicted on the right of the above diagram.

The composite of $(F,\varphi)$ followed by $(G,\psi): (\CJ,Y)\to(\CK,Z)$ in ${\sf Diag}^{\circ}(\CX){\text{ and }{\sf Diag}_{\circ}(\CX)}$ is respectively given by
$$(GF,\psi F\cdot\varphi)\qquad\text{ and }\qquad(FG,\psi\cdot\varphi G)\,.$$
\end{defn}
One has the obvious forgetful functors
$$D^{\CX}:{\sf Diag}^{\circ}(\CX)\to\SmallCats\qquad\text{ and }\qquad D_{\CX}:{\sf Diag}_{\circ}(\CX)\to\SmallCats^{\rm op}$$
which remember just the top rows of the above triangles; in the notation of  (\ref{GrothendieckConstruction}), they are precisely the functors $\Pi^{\CX^{\Box}}$ and $\Pi_{\CX^{\Box}}$, respectively. Consequently one has the following Proposition, which is also easily established ``directly''.

\begin{prop}\label{Diagasfibration} 
$D^{\CX}$ is a split fibration, with cleavage	$${\theta}^F_{(\CJ,Y)}=(F,1_{YF}):F^{*}(\CJ,Y)=(\CI,YF)\to (\CJ,Y)\;\text{ for }\; F\from \CI\to\CJ\text{ and }Y\in\CX^{\CJ}.$$
Dually, $D_{\CX}$ is a split cofibration, with cleavage $$\delta^F_{(\CI,X)}=(F,1_{XF}): (\CI,X)\to F_!(\CI,X)=(\CJ,XF)\;\text{ for }\; F\from \CJ\to\CI \text{ and } X\in\CX^{\CI}.$$
\end{prop}

Since the two types of Grothendieck constructions are dual to each other (as described in \ref{GrothendieckConstruction}), so are the two types of diagram categories. This fact one may easily see in a direct manner, by an application of the bijective 2-functor
$$\Box^{\rm op}:\LocSmallCats^{\co}\lra\LocSmallCats,\quad [\alpha:S\Lra T:\CC\to\CD]\longmapsto[\alpha^{\rm op}:S^{\rm op}\Lla T^{\rm op}:\CC^{\rm op}\to\CD^{\rm op}]$$ 
(which maps morphisms covariantly but 2-cells contravariantly) to the inscribed triangle on the right of the diagram of Definition \ref{diagramcats}. In this way one establishes an isomorphism between the dual of the (ordinary) category ${\sf Diag}_{\circ}(\CX)$ and the category
${\sf Diag}^{\circ}(\CX^{\rm op})$, coherently so with respect to the forgetful functors, as shown in the commutative diagram
\begin{center}
$\xymatrix{({\sf Diag}_{\circ}(\CX))^{\rm op}\ar[r]^{\Box^{\rm op}}_{\cong}\ar[d]_{(D_{\CX})^{\rm op}} & {\sf Diag}^{\circ}(\CX^{\rm op})\ar[d]^{D^{(\CX^{\rm op})}}\\
\SmallCats\ar[r]^{\Box^{\rm op}}_{\cong} & \SmallCats\,.
}$
\end{center}

The  objects and morphisms of $\CX$ may be considered as living in both, ${\sf Diag}^{\circ}(\CX)$ and ${\sf Diag}_{\circ}(\CX)$. Indeed, there are full embeddings
$$E^{\CX}:\CX\to {\sf Diag}^{\circ}(\CX)\qquad\text { and }\qquad E_{\CX}:\CX\to {\sf Diag}_{\circ}(\CX)$$
which interpret every object $X$ of $\CX$ as a functor ${\EuRoman 1}\to\CX$ of the terminal category $\EuRoman 1$ and every morphism $f:X\to Y$ in $\CX$ as a natural transformation, giving respectively the (hardly distinguishable) morphisms 
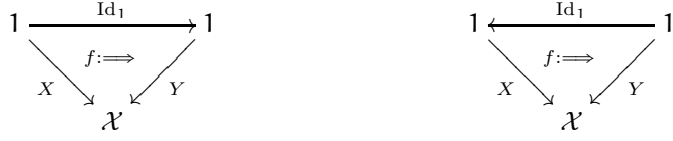
\begin{center}
$\xymatrix{{\EuRoman 1}\ar[rr]^{{\rm Id}_{\EuRoman 1}}\ar[dr]_X^{\;\;f:\Lra} & & {\EuRoman 1}\ar[dl]^{Y}\\
& \CX}$ \hfil $\xymatrix{{\EuRoman 1}\ar[dr]_X^{\;\;f:\Lra} & & {\EuRoman 1}\ar[ll]_{{\rm Id}_{\EuRoman 1}}\ar[dl]^{Y}\\
& \CX}$\end{center}
in ${\sf Diag}^{\circ}(\CX)$ and ${\sf Diag}_{\circ}(\CX)$. $E^{\CX}$ and $E_{\CX}$ cooperate with the dualization isomorphism for the diagram categories, as shown in the commutative diagram

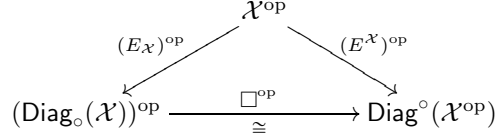
\begin{center}
$\xymatrix{
& \CX^{\rm op}\ar[dl]_{(E_{\CX})^{\rm op}}\ar[dr]^{(E^{\CX})^{\rm op}} & \\
({\sf Diag}_{\circ}(\CX))^{\rm op}\ar[rr]^{\Box^{\rm op}}_{\cong} & & {\sf Diag}^{\circ}(\CX^{\rm op})
}$
\end{center}

With (co)completeness to be understood to mean that every small diagram comes with a choice of (co)limit, one has the following ``folklore'' proposition, whose routine proof we may skip.
\begin{prop}\label{functor colim}
\begin{enumerate}[(1)]
\item The category $\CX$ is functorially\footnote{This means that colimits in $\CX$ are assumed to be chosen for all small diagrams in $\CX$.} cocomplete if, and only if, $E^{\CX}$ is a reflective embedding (with left adjoint ${\rm colim}:{\sf Diag}^{\circ}(\CX)\to \CX$).
\item The category $\CX$ is functorially complete if, and only if, $E_{\CX}$
 is a coreflective embedding (with right adjoint ${\rm lim}:{\sf Diag}_{\circ}(\CX)\to\CX$).
 \end{enumerate}
\end{prop}

\begin{rem}\label{Ignoring2cells}
(1) We emphasize that, here, we are considering ${\sf Diag}^{\circ}(\CX)$ and ${\sf Diag}_{\circ}(\CX)$, just like $\CX$, as 1-categories and, thus, ignore their obvious 2-categorical structures which make
$D^{\CX}$ and $D_{\CX}$ 2-functors. In the case of ${\sf Diag}^{\circ}(\CX)$, a 2-cell $\alpha:(F,\varphi)\Longrightarrow(F',\varphi')$
is simply a natural transformation $\alpha:F\to F'$ with $Y\al\cdot\varphi=\varphi'$: 
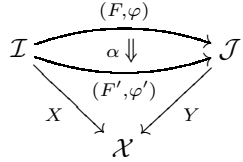
\begin{center}
$\xymatrix{\CI\ar@/^0.7pc/[rr]^{(F,\varphi)}\ar@/_0.7pc/[rr]_{(F',\varphi')}\ar[dr]_X & {\scriptstyle \al}\!\Downarrow & \CJ\ar[dl]^{Y}\\
& \CX}$
\end{center}
As $D^{\CX}$ should preserve the horizontal and vertical compositions of 2-cells, there is no choice of how to define them in 
${\sf Diag}^{\circ}(\CX)$. All verifications proceed routinely.

(2) Likewise, at this point we ignore the obvious fact that the functor $\CX^{\Box}=[-,\CX]:\SmallCats^{\rm op}\to \LocSmallCats$, as considered at the beginning of this section, is actually a 2-functor: it maps every natural transformation $\alpha:F\Longrightarrow F'$ (covariantly) to the transformation $\alpha^{*}:F^{*}\Longrightarrow {F'}^*$ with $\alpha^{*}_Y=Y\al:YF\lra YF'$, preserving both the vertical and horizontal composition of natural transformations.

(3) With ${\EuRoman 1}$ denoting the terminal category we trivially have
 $${\sf Diag}^{\circ}(\emptyset)\cong{\sf Diag}_{\circ}(\emptyset)\cong\sf 1 \text{ and } {\sf Diag}^{\circ}(\EuRoman 1)\cong\SmallCats,\, {\sf Diag}_{\circ}(\EuRoman 1)\cong\SmallCats^{\rm op}.$$
Much less obvious is the fact that the (ordinary) category ${\sf Diag}^{\circ}(\SmallCats )$ is equivalent to a suitably defined category which has the split cofibrations of small categories as its objects, and that ${\sf Diag}_{\circ}(\SmallCats^{\rm op})$ is equivalent to the dual of the category of split fibrations of small categories, as one may conclude from the 2-categorical equivalence formulated in Corollary \ref{GrothendieckEquivalenceDual}.
\end{rem}	

\subsection{Limits and colimits in diagram categories}\label{colimitsinDiag}

We continue to work with a fixed category $\CX$ and now consider the question of how to form limits or colimits in its diagram categories. Using fibrational methods we first confirm the following assertions:
\begin{prop}\label{limsandcolims}
\begin{enumerate}[(1)]
\item ${\sf Diag}^{\circ}(\CX)$ has coproducts, and $D^{\CX}$ preserves them. For any small category $\CD$, the category ${\sf Diag}^{\circ}(\CX)$ has all $\CD$-shaped limits with $D^{\CX}$ preserving them if, and only if, $\CX$ has $\CD$-limits.
If ${\sf Diag}^{\circ}(\CX)$ has coequalizers, then so does $\CX$, and $E^{\CX}$ preserves them.
\item ${\sf Diag}_{\circ}(\CX)$ has products, and $D_{\CX}$ preserves them. For small $\CD$, the category  ${\sf Diag}^{\circ}(\CX)$ has all $\CD$-colimits with $D_{\CX}$ preserving them if, and only if, $\CX$ has $\CD$-colimits. If ${\sf Diag}_{\circ}(\CX)$ has equalizers, then so does $\CX$, and $E_{\CX}$ preserves them.
\end{enumerate}
\end{prop}

\begin{proof}
(2) follows from (1) by dualization	(see Proposition \ref{Diagasfibration}), and the first and last assertion of (1) are quite obvious. 

The assertion about the existence $\CD$-limits for a small category $\CD$ when these exist in $\CX$ follows from Theorem \ref{LiftingColimits}, as soon as we have demonstrated that the comma insertions 
$$I_{\CJ}:[\CJ,\CX]\longrightarrow \CJ\!\downarrow\! D^{\CX}, \;X\longmapsto ({\rm Id}_{\CI}\to D^{\CX}(\CI,X))$$     
preserve $\CD$-limits, for any small category $\CJ $. But this preservation condition just means that, given a $\CD$-shaped limit cone $\lambda_d:Y\to Y_d$ in $[\CJ,\CX]$ and any $\CD $-shaped cone $(F,\varphi_d):(\CI,X)\to (\CJ,Y_d)$ in ${\sf Diag}^{\circ}(\CX)$, the cone $\varphi_d:X\to Y_dF$ in $[\CI,\CX] $ factors through the cone $\lambda_dF:YF\to Y_dF$, with a uniquely determined transfomation $\varphi :X\to YF$. But this is an immediate consequence of the pointwise formation of limits in functor categories, which makes the cone $\lambda_dF:YF\to Y_dF$ a limit cone in $[\CI,\CX]$. 

We leave to the reader the proof that the existence of $\CD$-limits in ${\sf Diag}^{\circ}(\CX)$ forces the existence of such limits in $\CX$.
\end{proof}

The question of how to form coequalizers in ${\sf Diag}^{\circ}(\CX)$ (or, equivalently, equalizers in ${\sf Diag}_{\circ}(\CX)$) may also be addressed with fibrational methods, as follows. For every functor $F:\CI\to\CJ$ of small categories, the functor $F^{*}:\CX^{\CJ}\to\CX^{\CI}$ has, by definition, a left adjoint $F_!$ if, and only if, for every $X\in\CX^{\CI}$, a (chosen) {\em left Kan extension} $F_!X={\rm Lan}_FX$ of $X$ along $F$ exists. The existence of this extension is certainly guaranteed when $\CX$ is cocomplete; conversely, a colimit of $X:\CI\to\CX$ can be obtained as the left Kan extension of $X$ along the functor $\CI\to{\EuRoman 1}$ to the terminal category. Consequently, with Proposition \ref{Diagasfibration} and Theorem \ref{bifibrations} we obtain:

\begin{prop}\label{Diagbifibred}
The split fibration $D^{\CX}$ is a bifibration if, and only if, the category $\CX$ is (small-) cocomplete. Dually, the split cofibration $D_{\CX}$ is a bifibration if, and only if, $\CX$ is (small-)complete.
\end{prop}

With the help of Theorem \ref{LiftingLimits}, Corollary \ref{thm:BiFibration-LiftingColimits}, and Proposition \ref{Diagbifibred} we can now state:	

\begin{thm}\label{Diag bicomplete} {\em (1)} 
If $\CX$ is a cocomplete category, then
${\sf Diag}^{\circ}(\CX)$ is also cocomplete, and $D^{\CX}$ preserves all colimits. 
 
 {\em (2)} If $\CX$ is complete, then ${\sf Diag}_{\circ}(\CX)$ is also complete, and $D_{\CX}$ preserves all limits (so that $D_{\CX}$ transforms limits in ${\sf Diag}_{\circ}(\CX)$ into colimits in $\SmallCats$). 
 \end{thm}
\begin{proof}
(1) Cocompleteness of $\CX$ is needed to make $D^{\CX}$ a bifibration. Its fibres, {\em i.e.}, the functor categories of $\CX$, have the types of (co)limits that $\CX$ has. Since the  base category, $\Cat$, is bicomplete, the assertion follows with the statements referred to above.

(2) With the dualization principle as stated before Proposition \ref{functor colim}, item (2) follows from an application of (1) to $\CX^{\rm op}$, rather than to $\CX$. 
\end{proof}

\begin{rem}
Here is a brief illustration	 of the construction of standard (co)limits in ${\sf Diag}(\CX)$ as it follows from the proofs of Propositions \ref{limsandcolims} and \ref{Diagbifibred} and Theorem \ref{Diag bicomplete}:

{\em Products}: The product $(\CI, X)$ of the family $(\CI_{\nu},X_{\nu}),\; \nu\in N,$ is given by
the products $\CI=\prod_{\nu\in N}\CI_{\nu}$ in $\SmallCats$ and $X((i_{\nu})_{\nu})=\prod_{\nu\in N}X_{\nu}(i_{\nu})$ in $\CX$, with the obvious projections. 

{\em Equalizers}: The equalizer $(E,\eta): (\CE,Z)\to (\CI,X)$ of $(F,\varphi),(G,\psi):(\CI,X)\to(\CJ,Y)$ is given by the equalizer $E:\CE\hookrightarrow \CI$ of $F,G$ in $\SmallCats$ and the equalizers $\eta_i: Zi\to XI$ of $\varphi_i,\psi_i:Xi\to YFi=YGi$ in $\CX$, for all $i\in \CE$.

{\em Coproducts}: The coproduct $(\CI, X)$ of the family $(\CI_{\nu},X_{\nu}),\; \nu\in N,$ is given by
the coproduct $\CI=\coprod_{\nu\in N}\CI_{\nu}$ in $\SmallCats$ and $X$ determined by the functors $X_{\nu}$.

{\em Coequalizers}:  To construct the coequalizer $(H,\gamma H\cdot\kappa): (\CJ,Y)\to (\CK,Z)$ of $(F,\varphi),(G,\psi):(\CI,X)\to(\CJ,Y)$, one first forms the coequalizer $H:\CJ\to\CK$ of $F,G$ in $\SmallCats$ and then the left Kan extensions of $X$ along $HF=HG$ and of $Y$ along $H$, with associated universal transformations $\lambda$ and $\kappa$; now, with
$\alpha,\beta$ determined by $\varphi,\psi$ as in the commutative diagrams
\begin{center}
$\xymatrix{X\ar[d]_{\varphi}\ar[rr]^{\lambda\qquad} && ({\rm Lan}_{HF}X)HF\ar@{-->}[d]^{\alpha HF}\\
YF\ar[rr]^{\kappa F\qquad} && ({\rm Lan}_H Y)HF
}$
\hfil
$\xymatrix{X\ar[d]_{\psi}\ar[rr]^{\lambda\qquad} && ({\rm Lan}_{HG}X)HG\ar@{-->}[d]^{\beta HG}\\
YG\ar[rr]^{\kappa G\qquad} && ({\rm Lan}_H Y)HG
}$	
\end{center}
one forms the coequalizer
\begin{center}
$\xymatrix{{\rm Lan}_{HF}X\ar@/^3pt/[rr]^{\alpha}\ar@/_3pt/[rr]_{\beta} && {\rm Lan}_HY\ar[rr]^{\gamma} && Z\\
}$	
\end{center}
in $[\CK,\CX]$. This completes the construction of the coequalizer
$(H, \gamma H\cdot\kappa)$.
\end{rem}

For completeness, we record here without proof the following result that is due to Foltz \cite{Foltz1970}; it will not be used in this paper.
\begin{thm}
	Let $\CX$ has products. Then ${\sf Diag}^{\circ}(\CX)$ is cartesian closed with exponentials preserved by $D^{\CX}$ if, and only if, $\CX$ is cartesian closed, and then $E^{\CX}$ preserves exponentials.
	 \end{thm}

\subsection{Left Kan extensions are cocartesian over $\SmallCats$, invariance under the formation of colimits}
The argumentation given for Proposition \ref{Diagbifibred} suggests the following characterization of left Kan extensions with target category $\CX$, when viewed as morphisms in ${\sf Diag}^{\circ}(\CX)$:
\begin{prop}\label{Kaniscocart}
A morphism $(F,\varphi):(\CI,X)\to(\CJ,Y)$ in ${\sf Diag}^{\circ}(\CX)$ is $D^{\CX}$-cocartesian if, and only if, $\varphi:X\to YF$ exhibits $Y$ as a left Kan extension of $X$	 along $F$.
\end{prop}

\begin{proof}
By definition, $D^{\CX}$-cocartesianness of $(F,\varphi)$ means that, for all functors $G:\CJ\to\CK,\; Z:\CK\to\CX$, one has a natural bijective correspondence
\begin{center}
$\xymatrix{X\ar@{=>}[r]^{\psi} & ZGF & \leftrightsquigarrow &
Y\ar@{=>}[r]^{\chi} & ZG\\
}$	
\end{center}
that is facilitated by the condition $\chi F\cdot \varphi=\psi$. Such correspondence is certainly in place if $Y$ is a left Kan extension of $X$ along $F$, with universal transformation $\varphi$. But this sufficient condition is also necessary, since one may simply consider $G={\rm Id}_{\CJ}$.
\end{proof}

Proposition \ref{Kaniscocart}
lets us conclude that, \emph{for $\CX$ cocomplete, the reflector ${\rm colim}:{\sf Diag}^{\circ}(\CX)\to \CX$ of Proposition} \ref{functor colim} \emph{maps $D^{\CX}$-cocartesian morphisms to isomorphisms.} In fact, one has the following more refined statement that appeared in  essence in \cite{PerroneTholen2020} (see the remark after the proof of Theorem 5.6 in \cite{PerroneTholen2020}). Here we formulate it without the imposition of smallness or cocompleteness conditions on the participating categories:

\begin{prop}\label{colimofKan}
 If $Y\cong {\rm Lan}_FX$, then the canonical morphism ${\rm \colim} X\to{\rm colim Y}$ is an isomorphism, with either colimit existing when the other exists.	
\end{prop}
\begin{proof}
For every object $z$ in $\CX$, we need to establish a natural bijective correspondence
\begin{center}
$\xymatrix{X\ar@{=>}[r]^{\alpha} & \Delta z & \leftrightsquigarrow &
Y\ar@{=>}[r]^{\beta} & \Delta z\;.\\
}$	
\end{center}
But this correspondence is readily provided by the correspondence of the proof of Proposition \ref{Kaniscocart} where one has to consider $Z=\Delta z,\; G={\rm Id}_{\CJ}$ and observe the trivial identity $\Delta z=(\Delta z) F$.	
\end{proof}
Proposition \ref{colimofKan} suggests the viewpoint that {\em Kan extensions are partial colimits}, a claim that is pursued in greater depth in \cite{PerroneTholen2020}.

\subsection{Strictification of morphisms in diagram categories}\label{strictification1}

Other than containing $\CX$ as a full subcategory, ${\sf Diag}^{\circ}({\CX})$ contains, of course, also the ordinary comma category $\SmallCats  /\CX$ as a {\em non-full} subcategory.  In fact, for {\em small} (!) $\CX$ one has the following proposition, embedded in the proof of Lemma 7.13 of \cite{BataninBerger}, with credit given to S. Lack. For arbitrary $\CX$, see Remark \ref{2functorStrict}(2); we can still draw the essential conclusion of preservation of colimits (see Corollary \ref{colimit preservation}) that is needed in the first proof of the CDF (see Section 4).

\begin{prop}\label{LackLemma}
For a small category $\CX$, the inclusion functor $\SmallCats /\CX\to{\sf Diag}^{\circ}({\CX})$ has a right adjoint.
\end{prop}

\begin{proof}
One defines a functor $${\rm Strict}:{\sf Diag}^{\circ}({\CX})\to\SmallCats /\CX,$$ which transforms lax commutative triangles into strictly commutative triangles. It assigns to an object $X$ in ${\sf Diag}^{\circ}(\CX)$ the comma category $\CX\downarrow X$, equipped with its domain functor which takes an object $(u:a\to Xi, \;i)$ in $\CX\downarrow X$ with $i\in\CI$ to the object $a\in\CX$. On morphisms 
it is defined by
\begin{center}
$\xymatrix{\CI\ar[rr]^F\ar[dr]_X^{\;\;\varphi:\Lra} & & \CJ\ar[dl]^{Y} & \longmapsto & \CX\downarrow X\ar[rr]^{{\rm Strict}(F,\varphi)}\ar[rd]_{{\rm dom}_X} & & \CX\downarrow Y\ar[dl]^{{\rm dom}_Y} \\
& \CX & & {\rm Strict} & & \CX & \\}$
\end{center}
where the functor ${\rm Strict}(F,\varphi)$ takes an object $(u:a\to Xi, \;i)$ in $\CX\downarrow X$ with $i\in\CI$ to the object $(\varphi_i\cdot u:a\to Y(Fi),\; Fi)$ in $\CX\downarrow Y$. It is now a routine exercise to establish a natural isomorphism
$$(\SmallCats /\CX)(X,{\rm dom}_Y)\cong{\sf Diag}^{\circ}({\CX})(X,Y)$$
in $\Sets$.
\end{proof}

\begin{rem}\label{2functorStrict}
(1) With 2-cells $\alpha:F\Longrightarrow F'$ given by natural transformations $\alpha$ satisfying $Y\alpha=1_X$,\begin{center}
$
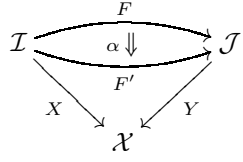
\xymatrix{\CI\ar@/^0.7pc/[rr]^{F}\ar@/_0.7pc/[rr]_{F'}\ar[dr]_X & {\scriptstyle \al}\!\Downarrow & \CJ\ar[dl]^{Y}\\
& \CX}$
\end{center}
$\SmallCats  /\CX$ becomes a 2-category. In fact, it is the 2-subcategory of ${\sf Diag}^{\circ}({\CX})$ whose 2-cells are described in Remark \ref{Ignoring2cells}(1). One easily confirms that the isomorphism at the end of the proof of Proposition \ref{LackLemma} actually lives in $\SmallCats$. Consequently, one has in fact a 2-adjunction
$${\rm Inclusion}\dashv {\rm Strict}: {\sf Diag}^{\circ}({\CX})\to\SmallCats /\CX.$$

(2)
When $\CX$ is large, {\em i.e.}, a $\LocSmallCats $-object, we still have a right adjoint, ${\rm Strict}$, to the inclusion 2-functor 
$$\LocSmallCats/\CX\to{\sf DIAG}^{\circ}(\CX),$$
where ${\sf DIAG}^{\circ}(\CX)$ is defined like ${\sf Diag}^{\circ}(\CX)$, except that the domain $\CI$ of an object $X:\CI\to\CX$ is not constrained to be small. We temporarily step into this higher universe to prove the following corollary, which will be used in the proof of Theorem \ref{decomp formula}.
\end{rem}

\begin{cor}\label{colimit preservation}
For any category $\CX$, the inclusion functor $\SmallCats/\CX\to{\sf Diag}^{\circ}(\CX)$ preserves all colimits.
\end{cor}

\begin{proof}
Trivially, the full embedding $\SmallCats \to\LocSmallCats$ 
preserves all colimits and, hence, so does
$\SmallCats/ \CX\to\LocSmallCats/\CX$. 
By (a large version of) 
Proposition \ref{LackLemma}, a colimit taken in $\SmallCats  /\CX$ is therefore also a colimit in ${\sf DIAG}^{\circ}(\CX)$, and it trivially maintains that role in its ``home'' category ${\sf Diag}^{\circ}({\CX})$.	
\end{proof}

\begin{rem}\label{Diagcirc}
We note that, in the notation of \ref{Grothendieck fibrations}, every small-fibred split fibration $P:\CE\to\CB$ comes with a {\em mate} 
$$P^{\leftarrow}:\CB\to {\sf Diag}_{\circ}(\CE),\;(u:a\to b)\mapsto ((u^*,\theta^u): J_a\to J_b),$$
which, when composed with the split fibration $D_{\CE}:{\sf Diag}_{\circ}(\CE)\to \LocSmallCats^{\rm op}  $, reproduces the functor $(\Phi^P)^{\rm op}:\CB\to\LocSmallCats^{\rm op}  $ of \ref{cofib and bifib}. Although $\Phi^P$ maintains sufficient information about $P$ to reproduce $P$ (as in Proposition \ref{equivalent}), the mate $P^{\leftarrow}$ may well be regarded as doing so more comprehensively.
\end{rem}

\subsection{The generic colimit for the standard construction of colimits}
Let $\CX$ have coproducts. As is well known, the \emph{standard construction} of the colimit of a small diagram $X:\CI\to\CX$ proceeds by forming the diagram 
\begin{center}
$\xymatrix{\coprod_{(u:i\to j)}Xi\ar@/^3pt/[rr]^f\ar@/_3pt/[rr]_g && \coprod_iXi\ar[rr]^q && C\\
Xi\ar[u]^{s_u}\ar[rr]_{Xu}\ar[rru]_{t_i} && Xj\ar[u]^{t_j}\ar[rru]_{\kappa_j}\\
}$
\end{center}	
where the left coproduct, with injections $s_u$, runs through all morphism $u$ of $\CI$ and the second, with injections $t_i$, through all objects $i$ of $\CI$; the morphisms $f$ and $g$ satisfy $f\cdot s_u=t_i$ and $g\cdot s_u=t_j\cdot Xu$, for all $u:i\to j$ in $\CI$. Then the colimit $C$ of $X$ with cocone $\kappa: X\to \Delta C$ exists in $\CX$ if, and only if, the coequalizer $q:\coprod_iX_i\to C$ of $f,g$ exists in $\CX$, with $\kappa$ and $q$ determining each other by $q\cdot t_i=\kappa_i$ for all $i\in\CI$.

\begin{exa}\label{generic} \emph{(Generic colimit)}
	For every small category $\CI$, considering the fibres of the codomain functor $\CI^{\EuRoman 2}\to\CI$ (see {\em\ref{Standard} (3)}), one has a coequalizer diagram
	\begin{center}
	$\xymatrix{\coprod_{(u:i\to j)}\CI/i\ar@/^3pt/[rr]^F\ar@/_3pt/[rr]_G && \coprod_i\CI/i\ar[rr]^Q && \CI	&\quad\quad(*)}$	
	\end{center}
in $\SmallCats$, where $Q$ restricts to the domain functor $Q_i$, for every summand $\CI/i$. \emph{Indeed, by direct inspection one sees that the cocone given by the domain functors exhibits $\CI$ as a colimit of the diagram $\Phi:\CI\longrightarrow\SmallCats,\; (u:i\to j)\longmapsto (u_!:\CI/i\to\CI/j)$, as follows: Given any cocone $(H_i:\CI/i\to\CC)_i$, a functor $H:\CI\to\CC$ with $HQ_i =H_i$ for all $i\in\CI$ must necessarily map every object $i$ to $H_i(1_i)$ and every morphism $u:i\to j$ to the morphism $H_j(u)$ in $\CC$, since $u$ in $\CI$ may be seen as the $Q_j$-image of the morphism $u:u\to 1_j$ in $\CI/j$, the domain $u$ of which may be written as the object $u=u_!(1_i)$ in $\CI/j$. Conversely, putting $Hu=H_ju$ one similarly confirms the functoriality of $H$.}
\end{exa}
Justifying its denomination, we now show in which sense Example \ref{generic} is generic:
\begin{thm}\label{colimpresentation}
The standard construction of the colimit of any small diagram $X:\CI\to \CX$ in a cocomplete category	 $\CX$ is the image under the functor ${\rm colim}:{\sf Diag}^{\circ}(\CX)\to\CX$ of a coequalizer diagram in ${\sf Diag}^{\circ}(\CX)$ whose underlying coequalizer diagram in $\SmallCats$ is given by diagram $(*)$ that depends only on the shape $\CI$ of the diagram $X$.
\end{thm}

\begin{proof}
Given $X$, the coequalizer diagram $(*)$	 may be lifted  trivially to become a coequalizer diagram in $\SmallCats/\CX$:
\begin{center}
$\xymatrix{\coprod_{(u:i\to j)}\CI/i\ar@/^3pt/[rr]^F\ar@/_3pt/[rr]_G\ar[rrd]_{XQF} && \coprod_i\CI/i\ar[rr]^Q\ar[d]^{XQ} && \CI\ar[lld]^X\\
&& \CX &&
}$	
\end{center}
By Corollary \ref{colimit preservation}, we actually have a coequalizer diagram in ${\sf Diag}^{\circ}(\CX)$ which, by Proposition \ref{functor colim}, is preserved by ${\rm colim}:{\sf Diag}^{\circ}(\CX)\to\CX$. We claim that the emerging coequalizer diagram in $\CX$,
\begin{center}
$\xymatrix{{\rm colim}XQF\ar@/^3pt/[rr]^f\ar@/_3pt/[rr]_g && {\rm colim}XQ \ar[rr]^q && {\rm colim}X\;,
}$
\end{center}	
coincides (up to isomorphism) with the standard construction of the colimit of $X$ in $\CX$. Indeed, in the notation of Example \ref{generic}, $XQ$ is the coproduct of the diagrams $XQ_i$ in ${\sf Diag}(\CX)$, so that we may compute
\begin{center}
	$\xymatrix{{\rm colim}(\coprod_i\CI/i\ar[r]^{XQ\qquad} & \CX)\quad\cong \quad\coprod_i{\rm colim}(\CI/i\ar[r]^{\quad XQ_i} & \CX)\quad\cong\quad \coprod_iXi\;,
	}$
\end{center}
where, for every object $i\in\CI$,  we were able to evaluate the colimit of $XQ_i$ trivially as $Xi$, since
$1_i$ is terminal in $\CI/i$. With the colimit of $XQF$ computed analogously, one now routinely checks that $\rm colim$ maps the functors $F,G, Q$ (seen as morphisms in ${\sf Diag}^{\circ}(\CX)$, and obtained by the standard construction of a colimit in $\SmallCats$) to the morphisms $f,g,q$ in $\CX$, as given by the standard construction of the colimit of $X$ in $\CX$.
\end{proof}	
\begin{rem}
Other than $(*)$ of Example \ref{generic}, there are alternative and quite natural presentations of a small category $\CI$ as a coequalizer in $\SmallCats$, to which one may then apply the same procedure for the computation of the colimit of a diagram $X:\CI\to\CX$ as the one used in the proof of Theorem \ref{colimpresentation}. Consider, for example the coequalizer diagrams
\begin{center}
	$\xymatrix{\CI^{\EuRoman 2}\times_{\CI}\CI^{\EuRoman 2}\ar@/^3pt/[rr]^{P_1}\ar@/_3pt/[rr]_{P_2} && \CI^{\EuRoman 2}\ar[rr]^P && \CI\;,	}$	
	\end{center}
\begin{center}
	$\xymatrix{\quad(\CI^{\EuRoman 2})^{\EuRoman 2}\ar@/^3pt/[rr]^M\ar@/_3pt/[rr]_{Q^{\EuRoman 2}} && \CI^{\EuRoman 2}\ar[rr]^P && \CI\;, }$	
	\end{center}
	where $P$ is the the codomain functor, and where $M$	 assigns to a commutative square in $\CI$, seen as an object of $(\CI^{\EuRoman 2})^{\EuRoman 2}$, its diagonal morphism, seen as an object in $\CI^{\EuRoman 2}$. Unlike $(*)$, both coequalizer diagrams are even contractible and, hence, absolute. However, the sections of $P, P_1, M$ making the coequalizers contractible are all easily seen to be confinal and, hence, under the procedure of the proof of Theorem \ref{colimpresentation}, produce just coequalizer diagrams consisting of isomorphisms, 
	for any diagram $X:\CI\to\CX$ that has a colimit in $\CX$.
	\end{rem}

\section{The Guitart  adjunction}

It is hardly surprising that ${\sf Diag}^{\circ}(\CX)$, constructed as a Grothendieck category over $\Cat $, behaves 2-functorially in the variable $\CX$. But it is a nice twist that the assignment $\CX\mapsto {\sf Diag}^{\circ}(\CX)$ (considered as a category over $\SmallCats   $) has a left adjoint, given again by the Grothendieck construction. This fact was stated by Guitart \cite{Guitart1974} (see also \cite{GuitartVdB1977}) in 1-categorical terms. In what follows, we give some details in 2-categorical terms. A generalization is formulated, and proved, as Theorem \ref{fundamental}.

\subsection{${\sf Diag}^{\circ}$ is right adjoint to the Grothendieck construction}

Considering $\SmallCats$ as a 1-category and $\LocSmallCats $ as a (huge) 2-category, containing $\SmallCats$ as one of its objects, we form the 2-category $\LocSmallCats/\SmallCats$ as in Remark \ref{2functorStrict}(1). Then 
the 2-functor
$${\sf Diag}^{\circ}:\LocSmallCats   \longrightarrow \LocSmallCats/\SmallCats $$
assigns to a category $\CX$ the fibration $D^{\CX}:{\sf Diag}^{\circ}(\CX)\to\Cat$; it extends a functor $T:\CX\to\CY$ from ordinary to  ``variable'' objects of $\CX$ by post-composition with $T$, that is: ${\sf Diag}^{\circ}$ assigns to $T$ the $\LocSmallCats/\SmallCats$-morphism $D^{\CX}\to D^{\CY}$, given by the functor 
$$T(-):{\sf Diag}^{\circ}(\CX)\longrightarrow{\sf Diag}^{\circ}(\CY),\quad[(F,\varphi):(\CI,X)\to (\CJ, Y)]\longmapsto [(F,T\varphi):(\CI,TX)\to (\CJ, TY)]\;;$$
and it assigns to a natural transformation $\tau:T\to T'$ the natural transformation
$\tau(-):T(-)\to T'(-)$, given by $({\rm Id}_{\CI},\tau X): (\CI, TX)\to(\CI, T'X)$ for all
objects $(\CI,X)$ in ${\sf Diag}^{\circ}(\CX)$.

The 2-functor
$$\LocSmallCats\longleftarrow\LocSmallCats/\SmallCats:\textstyle\int_{\circ}$$
assigns to a $\LocSmallCats/\SmallCats$-object $\Phi:\CB\to\SmallCats$ its dual Grothendieck category $\int_{\circ}\Phi$, and to a $\LocSmallCats/\SmallCats$-morphism $\Sigma:\Phi\to\Psi$ with $\Psi:\CC\to\SmallCats$ the functor
$$(\Sigma-,=):\textstyle{\int}_{\circ}\Phi\longrightarrow\textstyle{\int}_{\circ}\Psi,\quad[(u,f):(a,x)\to(b,y)]\longmapsto [(\Sigma u,f):(\Sigma a,x)\to(\Sigma b,y)].$$
A natural transformation $\si:\Sigma\to \Sigma'$ with 
$\Psi \sigma=1_{\Phi}$ is sent by $\int_{\circ}$ to the natural transformation 
$(\Sigma-,=)\to (\Sigma'-,=)$ whose component at an object $(a,x)$ in $\int_{\circ}\Phi$ is the morphism $(\si_a,1_{(\Phi a)x}):(\Sigma a,x)\to(\Sigma'a,x)$ in $\int_{\circ}\Psi$.

\begin{thm}\label{Guitart}
The 2-functor $\int_{\circ}$ is left adjoint to the 2-functor ${\sf Diag}^{\circ}$.
\end{thm}

\begin{proof}
It suffices to show that, for every category $\CX$ and every functor $\Phi:\CB\to\SmallCats$, one has bijective functors
$$\widehat{\Box}:\;\LocSmallCats  (\textstyle{\int}_{\circ}\Phi,\CX)\rightleftarrows(\LocSmallCats/\SmallCats)(\Phi,D^{\CX})\;:\widecheck{\Box}$$
that are natural in $\CX$ and $\Phi$. To this end, for a functor $T:\int_{\circ}\Phi\to \CX$, one lets the functor
 $\hat{T}:\CB\to{\sf Diag}^{\circ}(\CX)$ map an object $a\in\CB$ to the functor 
 $$T_a:\Phi a\to\CX,\quad(f:x\to x')\longmapsto [T(1_a,f):(T(a,x)\to T(a,x')].$$
 $\hat{T}$ maps a morphism $u:a\to b$ in $\CB$ to the ${\sf Diag}^{\circ}(\CX)$-morphism
 \begin{center}
$\xymatrix{{\Phi}a\ar[rr]^{\Phi u}\ar[dr]_{T_a}^{\;\;\varphi^u:\Lra} & & {\Phi}b \ar[dl]^{T_b}\\
& \CX\;, & }$
\end{center}
where the natural transformation $\varphi^u$ is defined by $\varphi^u_x	=T(u,1_{\Phi u(x)})$, for all $x\in \Phi a$; clearly then, $D^{\CX}\hat{T} =\Phi$. Also, for a natural transformation $\tau: T\to T'$ one has the 2-cell $\hat{\tau}:\hat{T}\to\hat{T'}$, the components of which are the ${\sf Diag}^{\circ}(\CX)$-morphisms
\begin{center}
$\xymatrix{\Phi a\ar[rr]^{{\rm Id}_{\Phi a}}\ar[dr]_{T_a}^{\;\;\hat{\tau}_a:\Lra} & & \Phi a\ar[dl]^{T'_a}\\
& \CX & ,}$
\end{center}
defined by $(\hat{\tau}_a)_x=\tau_{(a,x)}$, for all $a\in\CB, x\in\Phi a$.

Conversely, for a functor $\Sigma:\CB\to {\sf Diag}(\CX)$ with $D^{\CX}\Sigma=\Phi$, one defines the functor $\check{\Sigma}:\int_{\circ}\Phi\to\CX$, as follows. For $u:a\to b$ in $\CB$, 
writing the ${\sf Diag}(\CX)$-morphism $\Sigma u$ in the form $\Sigma u=(\Phi u,\varphi^u)$ (as in the triangle on the left of the following diagram), one lets $\check{\Sigma}$ map a morphism $(u,f):(a,x)\to(b,y)$ in $\int_{\circ}\Phi$ to the composite morphism of the triangle on the right:
\begin{center}
$\xymatrix{\Phi a \ar[rr]^{\Phi u}\ar[dr]_{\Sigma a}^{\;\;\varphi^u:\Lra} & & \Phi b \ar[dl]^{\Sigma b}\\
& \CX & }$
\hfil
$\xymatrix{& (\Sigma b)(\Phi u)x\ar[rd]^{(\Sigma b)f} &\\
(\Sigma a)x=\check{\Sigma}(a,x)\ar[ru]^{\varphi^u_x}\ar[rr]^{\check{\Sigma}(u,f)} & & \check{\Sigma}(b,y)=(\Sigma b)y\;. 
}$
\end{center}
Not surprisingly now, a natural transformation $\si:\Sigma\to \Sigma'$ with $D^{\CX}\si=1_{\Phi}$ gives the natural transformation $\check{\sigma}:\check{\Sigma}\to\check{\Sigma}'$, defined by $\check{\sigma}_{(a,x)}=(\si_a)_x$, for all $(a,x)\in \int_{\circ}\Phi$.

All verifications proceed in a standard manner.
\end{proof}

Let us make explicit how, for a functor $\Phi:\CB\to\SmallCats$, Theorem \ref{Guitart} provides an effective characterization of the category $\int_{\circ}\Phi$ in the $2$-category $\LocSmallCats$. 
A {\em lax cocone} over $\Phi$ with vertex $\CX$ is given by a family of functors $\Sigma_a:\Phi a\to\CX\; (a\in \CB)$ and a family of natural transformations $\varphi^u:\Sigma_a\to \Sigma_b(\Phi u)\;(u:a\to b\text{ in }\CB)$, satisfying the conditions
$$\varphi^{1_a}=1_{\Sigma_a}\;\text{ and }\;\varphi^{v\cdot u}=\varphi^v(\Phi u)\cdot\varphi^u,$$
for all $u:a\to b,\,v:b\to c$ in $\CB$. We recall that the category $\int_{\circ}\Phi$ is the vertex of the lax cocone over $\Phi$, given by the functors 
$$J_a:\Phi a\to \textstyle{\int}_{\circ}\Phi,\quad(h:x\to x')\mapsto(1_a,h):(a,x)\to(a,x'),$$ and the natural transformations
$$\delta^u:J_a\to J_b(\Phi u),\;\text{ with } \delta^u_x=(u,1_{\Phi u(x)}):(a,x)\to(b,(\Phi u)x),$$
for all $u:a\to b$ in $\CB$ and $x\in \Phi a$ (see \ref{cofib and bifib}). This lax cocone is initial amongst all lax cocones over $\Phi$, in the following sense:

\begin{cor}[Lax Colimit Characterization of $\int_{\circ}\Phi$]\label{universalchar dualGr} For every lax cocone over $\Phi$, given by $(\Sigma_a:\Phi a\to\CX)_{a\in \CB},\,(\varphi^u:\Sigma_a\to \Sigma_b(\Phi u))_{u:a\to b}$, there is a uniquely determined functor $T:\int_{\circ}\Phi\to \CX$ with $T J_a=\Sigma_a$ and $T \delta^u=\varphi^u$, for all $u:a\to b$ in $\CB$.
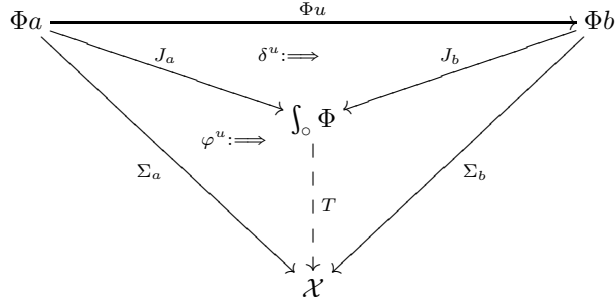
\begin{equation*}
\xymatrix@R=5ex@C=4em{
\Phi a\ar[rrrr]^{\Phi u}\ar[drr]^{\quad\quad J_a\qquad\quad\delta^u:\Lra}\ar[dddrr]_{\Sigma_a}^{\quad\varphi^u:\Lra} & & & & \Phi b\ar[dll]_{J_b}\ar[dddll]^{\Sigma_b}\\
& & \int_{\circ}\Phi\ar@{-->}[dd]^{T} & &\\
& & & & \\
& & \CX & & 
}
\end{equation*}
\end{cor}

\begin{proof}
The lax cocone $(J_a),(\delta^u)$ describes the adjunction unit $J:\CB\to{\sf Diag}^{\circ}(\int_{\circ}\Phi)$, and the corollary just paraphrases its universal property, as indicated by the diagram
\begin{center}
$\xymatrix{\CB\ar[r]^{J\quad\quad\;}\ar[rd]_{\Sigma} & {\sf Diag}^{\circ}(\int_{\circ}\Phi)\ar[d]^{T(-)} & \int_{\circ}\Phi\ar@{-->}[d]_{T}^{(=\widecheck{\Sigma})}\\
& {\sf Diag}^{\circ}(\CX) & \CX\,.\\ }$
\end{center}
\end{proof}

The dualization of Corollary \ref{universalchar dualGr} for a functor $\Phi:\CB^{\rm op}\to \Cat$ reads as follows:
\begin{cor}[Oplax Colimit Characterization of $\int^{\circ}\Phi$]\label{universalchar Gr} For every oplax cocone over $\Phi$, given by $(\Sigma_a:\Phi a\to\CX)_{a\in \CB},\,(\varphi^u:\Sigma_a(\Phi u)\to \Sigma_b)_{u:a\to b}$, there is a uniquely determined functor $T:\int^{\circ}\Phi\to \CX$ with $T J_a=\Sigma_a$ and $T \theta^u=\varphi^u$, for all $u:a\to b$ in $\CB$.	\end{cor}

\begin{rem}(1) A ``direct'' proof of Corollary \ref{universalchar Gr} makes essential use of the (vertical, $\Pi^{\Phi}$-cartesian)--factorization $(u,f)=(u,1_{\Phi u(y)})\cdot (1_a,f)=\theta^u_y\cdot J_af$ of every morphism $(u,f):(a,x)\to(b,y)$ in $\int^{\circ}\Phi$.
Likewise for Corollary \ref{universalchar dualGr}.

(2) Of course, Corollaries \ref{universalchar dualGr} and \ref{universalchar Gr} remain valid {\em verbatim} if the functor $\Phi$ is $\CAT$-valued (rather than $\Cat$-valued).
\end{rem}

As another important consequence of Theorem \ref{Guitart} we note:
\begin{cor}
${\sf Diag}^{\circ}:\CAT\to \CAT/\Cat$ preserves all (weighted) limits, and 	its left adjoint $\int_{\circ}$ preserves all (weighted) colimits.
\end{cor}

\subsection{${\sf Diag}^{\circ}$ has the structure of a normal pseudomonad on ${\LocSmallCats}$}\label{pseudomonad}
That ${\sf Diag}^{\circ}$ belongs to a (pseudo-)monad on $\LocSmallCats$ was already observed in \cite{Guitart1973}. But since a detailed exposition of this claim, even in a 1-categorical form, does not seem to be readily accessible, we outline the construction of the monad; a more detailed exposition and proof of the claim as given by this section's header appears in \cite{PerroneTholen2020}. 

Our  2-functor
$${\sf Diag}^{\circ}:\LocSmallCats\longrightarrow\LocSmallCats,\quad \CX\longmapsto{\sf Diag}^{\circ}(\CX),$$
arises by post-composing the right-adjoint of Theorem \ref{Guitart} with the forgetful 2-functor $\LocSmallCats/\SmallCats\to\LocSmallCats$. The full embedding $E^{\CX}:\CX\to{\sf Diag}^{\circ}(\CX)$ of \ref{colimitsinDiag} may then be considered as the $\CX$-component of a (strictly) 2-natural transformation
$$E: {\rm Id}_{\LocSmallCats}\longrightarrow{\sf Diag}^{\circ}$$
since, as one routinely confirms, every natural transformation $\alpha:F\Rightarrow F':\CX\to\CY$ satisfies $${\sf Diag}^{\circ}(F)E^{\CX}=E^{\CY}F\quad{\text{ and }}\quad{\sf Diag}^{\circ}(\alpha)E^{\CX}=E^{\CY}\alpha.$$
 In order to establish ${\sf Diag}^{\circ}$ as the carrier of a pseudo-monad, we now define for every category $\CX$ a functor
$$M^{\CX}:{\sf Diag}^{\circ}({\sf Diag}^{\circ}(\CX))\longrightarrow{\sf Diag}^{\circ}(\CX).$$
 An object in ${\sf Diag}^{\circ}({\sf Diag}^{\circ}(\CX))$ is a functor $\Sigma:\CB\to{\sf Diag}^{\circ}(\CX)$ with $\CB$ small so that, with $\Phi:=D^{\CX}\Sigma:\CB\to{\SmallCats}$, for every object $a$ in $\CB$ one has a functor $\Sigma a:\Phi a\to\CX$, and for every morphism $u:a\to b$ in $\CB$ a morphism $(\Phi u,\sigma^u):\Sigma a\to\Sigma b$ in ${\sf Diag}^{\circ}(\CX)$.
Considering $\Sigma$ as a lax cocone with vertex $\CX$, by Corollary \ref{universalchar dualGr} we may represent $\Sigma$ equivalently as a functor
$$M^{\CX}(\Sigma):=\check{\Sigma}:\;\textstyle{\int}_{\circ}(\Phi)\longrightarrow\CX,\quad[(u,f):(a,x)\to(b,y)]\longmapsto[\Sigma b(f)\cdot\sigma^u_x: \Sigma a(x)\to\Sigma b(y)].$$
A ${\sf Diag}^{\circ}({\sf Diag}^{\circ}(\CX))$-morphism $(S,\tau):\Sigma \to \Xi$ with codomain  $\Xi:\CC\to{\sf Diag}^{\circ}(\CX)$ 
is given by a functor $S:\CB\to\CC $ of small categories and a natural transformation $\tau:\Sigma\to\Xi S$ whose component at $a\in\CB$ is, in turn, given by a ${\sf Diag}^{\circ}(\CX)$-morphism $(R_a,\rho^a):\Sigma a\to\Xi a$, as in
\begin{center}
$\xymatrix{{\Phi a}\ar[rr]^{R_a}\ar[rd]_{\Sigma a}^{{\quad\rho^a}:\Longrightarrow} & & \Psi(Sa)\ar[ld]^{\Xi a}\\
& \CX & \\
}$	
\end{center}
where $\Psi:=D^{\CX}\Xi$. Now one lets $M^{\CX}$ assign to $(S,\tau)$ the ${\sf Diag}^{\circ}(\CX)$-morphism $(\check{S},\check{\tau})$, as shown by
\begin{center}
$\xymatrix{\CB\ar[rr]^S\ar[dr]_{\Sigma}^{\qquad\tau:\Lra} & & {\CC  }\ar[dl]^{\Xi} & \longmapsto & {\int}_{\circ}\Phi\ar[rr]^{\check{S}=\int_{\circ}(S,D^{\CX}\tau)}\ar[rd]_{\check{\Sigma}}^{\quad\check{\tau}:\Longrightarrow } & & {\int}_{\circ}\Psi\ar[dl]^{\check{\Xi}} \\
& {\sf Diag}^{\circ}(\CX) & & & & \CX & \\}$
\end{center}
where the functor $\check{S}$ and the natural transformation $\check{\tau}$ are defined by 
$$\check{S}(u,f)=(Su,R_bf) \quad{\text{ and }}\quad\check{\tau}_{(a,x)}=\rho^a_x,$$ for all morphisms 
$(u,f):(a,x)\to(b,y)$ in ${\int}_{\circ}\Phi$.

\section{The twisted Fubini formula and cofibred colimits}
\label{sec:Twisted-Fubini}
\subsection{The twisted Fubini formulae for colimits and limits}

We now exploit the adjunction of Theorem \ref{Guitart} for the computation of colimits of those diagrams in a category $\CX$ whose shape is the dual Grothendieck category of a functor $\Phi:\CD\to\SmallCats  $. Such a diagram $T:\int_{\circ}\Phi\to\CX$ in $\CX$ corresponds equivalently to a diagram $\hat{T}:\CD\to {\sf Diag}^{\circ}(\CX)$ in ${\sf Diag}^{\circ}(\CX)$ with $D^{\CX}\hat{T}=\Phi$. As we will show in Theorem \ref{newFubini}, the colimit of $\hat{T}$ facilitates the computation of the colimit of $T$. The essence of its proof lies in the next lemma, for which we use the following notation.
Given $\Phi:\CD\to \SmallCats$ and $\CX$, every functor $F:\CD\to\CX$ gives us trivially the functor 
$$\tilde{F}:\CD\to{\sf Diag}^{\circ}(\CX),\quad(u:d\to e)\longmapsto (\xymatrix{\Phi d \ar[rr]^{\Phi u}\ar[dr]_{\Delta Fd}^{\;\Delta Fu:\Lra} & & \Phi e \ar[dl]^{\Delta Fe}\\
& \CX & })$$
with $D^{\CX}\tilde{F}=\Phi$. (As usual, we use $\Delta$ for constant-value functors and transformations.) For a natural transformation $\al:F\to F'$, one defines a natural transformation $\tilde{\al}:\tilde{F}\to\tilde{F'}$ with $D^{\CX}\tilde{\al}=1_{\Phi}$ whose components are $\tilde{\al}_d=({\rm Id}_{\Phi d},\Delta\al_d)$. This defines the functor
$$\widetilde{\Box}:\LocSmallCats(\CD,\CX)\to(\LocSmallCats/\SmallCats)(\Phi,D^{\CX}).$$
\begin{lem}\label{constant} 
For every functor $\Phi:\CD\to\SmallCats$ and every category $\CX$, the functor $\widetilde{\Box}$ makes the diagram
\begin{center}
$\xymatrix{\LocSmallCats(\int_{\circ}\Phi,\CX)\ar[r]^{\widehat{\Box}\quad}_{\cong\quad} & (\LocSmallCats  /\SmallCats)(\Phi,D^{\CX})\\
\CX\ar[u]^{\Delta}\ar[r]^{\Delta\quad\quad} & \LocSmallCats(\CD,\CX)\ar[u]_{\widetilde{\Box}}  }$	
\end{center}
commute. If, for all $d\in\CD$, the category $\CX$ is $\Phi d$-cocomplete, then $\widetilde{\Box}$ has a left adjoint.
\end{lem}
\begin{proof}
Checking the commutativity of the diagram is a routine matter. In order to construct a left adjoint $\widebar{\Box }\dashv\widetilde{\Box }$, assuming $\CX$ to be $\Phi d$-cocomplete and using the same notation as in the proof of Theorem \ref{Guitart},
for a functor $\Sigma:\CD\to{\sf Diag}^{\circ}(\CX)$ with $D^{\CX}\Sigma=\Phi$, we define $\bar{\Sigma}:\CD\to\CX$ on objects by
$$\bar{\Sigma }d={\rm colim}(\Sigma_d:\Phi d	\to\CX);$$
this definition extends canonically to morphisms. (Of course, for $\CX$ cocomplete, $\bar{\Sigma}$ is the composite functor ${\rm colim} \circ\Sigma$, with 
${\colim}\dashv E_{\CX}:\CX\to {\sf Diag}^{\circ}(\CX)$, as in Proposition \ref{functor colim}.) For every functor $F:\CD\to\CX$ one now obtains a natural bijection
$$\LocSmallCats(\CD,\CX)(\bar{\Sigma},F)\to (\CAT/\Cat)(\Phi,D^{\CX})(\Sigma,\tilde{F}),$$
which associates with a natural transformation $\al:\bar{\Sigma}\to F$ its mate $\al^{\sharp}:\Sigma\to\tilde{F}$, as follows: for every $d\in\CD$, the natural transformation $\al_d^{\sharp}:S_d\to \Delta Fd:\Phi d\to \CX$ is simply the composite transformation
\begin{center}
$\xymatrix{\Sigma_d\ar[r] & \Delta({\rm colim}(\Sigma_d))\ar[r]^{\quad\quad\Delta\al_d} &  \Delta Fd \,.}$	
\end{center}
This confirms the adjunction.
\end{proof}
With the notation used in the proofs of Theorem \ref{Guitart} and of Lemma \ref{constant} one now obtains a general Fubini-type colimit formula that seems to have been proved first by  Chach\'{o}lski and Scherer  \cite[40.2]{WChacholskiJScherer2002}, as follows:
\begin{thm}[Twisted Fubini Colimit Formula]\label{newFubini}
 For a functor $\Phi:\CD\to \Cat$, let the category $\CX$ be $\Phi d$-cocomplete, for all $d\in \CD$. Then the colimit of any diagram $T:\int_{\circ}\Phi\to \CX$ exists in $\CX$ if, and only if, the colimit of the diagram $\CD\to\CX,\; d\mapsto{\rm colim}(\hat{T}d),$ exists in $\CX$, and in that case the two colimits coincide:
$${\rm colim}^{(d,x)\in{\int}_{\circ}\!\Phi}\;T(d,x)\;\cong\;{\rm colim}^{d\in\CD}({\rm colim}^{x\in\Phi d}\;T(d,x)).$$
\end{thm}
\begin{proof}
By the commutative diagram and the adjunction established in Lemma \ref{constant}, cocones $T\Lra \Delta X:\int_{\circ}\Phi\to \CX$ correspond bijectively to cocones 
	$\bar{\hat{T}}\Lra\Delta X:\CD\to\CX$, and naturally so in $X\in\CX$. Consequently, the universal representation of either type of cocone exists if the other does, and they then coincide, up to a canonical isomorphism.
\end{proof}
\begin{rem}
(1) Note that, since all $\Phi d\;(d\in\CD)$ are small, also $\int_{\circ}\Phi$ is small when $\CD$ is small.

(2) It is not hard to prove Theorem \ref{newFubini} ``directly'', without recourse to Theorem \ref{Guitart} and Lemma \ref{constant}: one may simply prove that the composite cocone
$$T(d',-)\longrightarrow {\rm colim}_{x\in\Phi d'}\,T(d',x)\longrightarrow {\rm colim}_{d\in\CD}({\rm colim}_{x\in\Phi d}\,T(d,x))$$
(as given by the right-hand side of the formula) is well defined and serves as a colimit cocone for $T$, and conversely.

(3) As shown in \cite{PerroneTholen2020}, the Twisted Fubini Colimit Formula for a cocomplete category $\CX$ is equivalently expressed by the fact that, for the pseudo-algebra $(\CX,\,{\rm colim}:{\sf Diag}^{\circ}(\CX)\to\CX)$
with respect to the pseudo-monad ${\sf Diag}^{\circ}$ (see \ref{pseudomonad}), the diagram 
\begin{center}
$\xymatrix{{\sf Diag}^{\circ}({\sf Diag}^{\circ}(\CX))\ar[rr]^{\quad{\sf Diag}^{\circ}({\rm colim})}\ar[d]_{M^{\CX}} & & {\sf Diag}^{\circ}(\CX)\ar[d]^{\rm colim}\\
{\sf Diag}^{\circ}(\CX)\ar[rr]^{\qquad\rm colim} && \CX\\
}$
\end{center}
commutes up to isomorphism.

(4) Although every colimit may be described as a coend (and conversely), we do not see a convenient way of deriving Theorem \ref{newFubini} from the interchange formula for coends (``Fubini'') for coends, as recorded in \cite{SMacLane1998},  Section IX.8, or \cite{Loregian2020}, Section 1.3.
\end{rem}

The dualization of Theorem \ref{newFubini} reads as follows:
\begin{cor}[Twisted Fubini Limit Formula] For a functor $\Phi:\CD^{\rm op}\to \Cat$, let the category $\CX$ be $\Phi d$-complete, for all $d\in \CD$. Then the limit of any diagram $T:\int^{\circ}\Phi\to \CX$ exists in $\CX$ if, and only if, the limit of the diagram $\CD\to\CX,\; d\mapsto{\rm lim}\,T(d,-),$ exists in $\CX$, and in that case the two limits coincide:
$${\rm lim}^{(d,x)\in\int^{\circ}\!\Phi}\;T(d,x)\;\cong\;{\rm lim}^{d\in\CD}({\rm lim}^{x\in\Phi d}\;T(d,x)).$$
\end{cor}

 Theorem \ref{newFubini} implies the ``untwisted'' Fubini formula that is recorded in Mac Lane's book \cite{SMacLane1998}:

\begin{cor}[Fubini (Co)Limit Formula]
\label{cor:FubiniForProduct-CoLims}
For every functor $T:\CD\times \CE\to \CX$ into an $\CE$-cocomplete category $\mathcal{X}$, the colimit of $T$ exists in $\CX$ if, and only if, the colimit of the $\CD$-indexed diagram $d\mapsto {\rm colim}_{\CE}\,T(d,-)$ exists in $\CX$, and then the two colimits coincide:
\begin{equation*}
{\rm colim}^{(d,e)\in\CD\times\CE}\,T(d,e)\;\cong\;{\rm colim}^{d\in\CD}({\rm colim}^{e\in\CE}\,T(d,e)).
\end{equation*}
Likewise for limits.
\end{cor}
\begin{proof}
Let $\Phi:\CD\to \Cat$ be the functor which has constant value $\CE$ (formally assumed to be a small category). Then $\int_{\circ}\Phi=\CD\times\CE$, and the assertion of the corollary follows from Theorem \ref{newFubini}.
\end{proof}

\subsection{Cofibred colimits}
The shape of the diagram $T$ considered in Theorem \ref{newFubini} is the dual Grothendieck category of a $\SmallCats$-valued functor $\Phi$. But by the (well-known) dual of Theorem \ref{FixedBaseGrothEquivalence}(i), we may replace $\int_{\circ}\Phi$ isomorphically by the domain of any (split) cofibred category over the domain of $\Phi$, which we now call $\CB$ (rather than $\CD$ as in Theorem \ref{newFubini}). Hence, instead of $\Phi$ we consider a split cofibration $P:\CE\to\CB$ with (in the notation of Section \ref{cofib and bifib}) cocleavages $\delta^u:J_a\to J_bu_!$, for all $u:a\to b$ in $\CB$. The Twisted Fubini Colimit Formula then tells us that we may evaluate the colimit of an $\CE$-shaped diagram in any (sufficiently cocomplete) category by first computing the colimits ``fibre by fibre'' and then the colimit of the emerging $\CB$-shaped diagram in $\CX$. This means:

\begin{thm}[Cofibred Colimit Theorem]\label{cofibredcolimit} 
Let $P:\CE\to\CB$ be a (split) cofibration and $T:\CE \to \CX$ be a diagram	such that the colimits of all restricted diagrams $TJ_b:\CE_b\to\CX\;(b\in\CB)$ exist in $\CX$. Then the
left Kan extension ${\rm Lan}_PT:\CB\to\CX$ exists, given by
$({\rm Lan}_PT)b={\rm colim}^{x\in\CE_b}Tx$ for all $b\in\CB$, and one has the formula
$${\rm colim}^{x\in \CE}Tx\cong {\rm colim}^{b\in\CB}({\rm colim}^{x\in\CE_b}Tx)\;,$$
with the colimit on the left existing in $\CX$ if, and only if, the one on the right exists in $\CX$.
\end{thm}
\begin{proof}
By the dual of Theorem \ref{FixedBaseGrothEquivalence} (i), with $\Phi_P:b\mapsto\CE_b$ we have the bijective functor
\begin{equation*}
\xymatrix@R=0.2ex@C=4em{
\int_{\circ}\Phi_P \ar[r]^{K_P} \ar[ddd]_{\Pi_{\Phi_P}} &
	\CatSymbA{E} \ar[ddd]^{P} &
	(b,y) \ar@{|->}[r] & y \\
	&& [(u,f)\from (a,x)\to (b,y)] \ar@{|->}[r] &
			[ f \cdot \delta^u_x\from x\to y] \\ \\
\CatSymbA{B} \ar@{=}[r] &
	\CatSymbA{B} 
}
\end{equation*}
which ``commutes over $\CB$''. An application of Theorem \ref{newFubini} to $\Pi_{\Phi_P}$ and $TK_P$ establishes precisely the claimed colimit formula, up to the negligible isomorphism $K_P$. But we   must still confirm the additional claim that the functor
$$L:\CB\to \CX,\; b\mapsto {\rm colim}TJ_b,$$
constitutes a left Kan extension of $T$ along $P$ which, then, provides another proof of the. stated colimit formula, by Proposition \ref{colimofKan}.

  Indeed, in the notation of the proof of Theorem \ref{Guitart}, we have $L={\colim}\widehat{TK_P}$; explicitly, for $u:a\to b$ in $\CB$, the morphism $Lu$ is defined by the commutative diagram
\begin{equation*}
\xymatrix{TJ_a\ar[r]^{\lambda^a}\ar[d]_{T\delta^u} & \Delta La\ar@{-->}[d]^{\Delta Lu}\\
TJ_bu_!\ar[r]^{\lambda^bu_!} & \Delta Lb\\
}	
\end{equation*}
where $\lambda^a$ denotes the colimit cocone of $TJ_a$. Let us verify that putting $\lambda_x:=\lambda^{Px}_x$ for all objects $x\in\CE$ defines a natural transformation $\lambda: T\to LP$. Indeed, for any morphism $f:x\to y$ in $\CE$, with $u=Tf$ and the (cocartesian,vertical) factorization of $f=\nu_f\cdot \delta^u_x$, we have the commutative diagram
\begin{equation*}
\xymatrix{Tx\ar[r]^{\lambda^{Px}_x}\ar@/_30pt/[dd]_{Tf}\ar[d]_{T\delta^u_x} &  LPx\ar[d]^{ LPf}\\
Tu_!x\ar[r]^{\lambda^{Py}_{u_!x}}\ar[d]_{T\nu_f} & LPy\\
Ty\ar[ur]_{\lambda_y^{Py}} & \\
}	
\end{equation*}	
where the triangle commutes by naturality of the cocone $\lambda^{Py}$. This confirms the naturality of $\lambda$. Given any functor $M:\CE\to\CX$ and a natural transformation $\psi:T\to MP$, then for every object $a\in \CB$ we have the cocone $\psi J_a: TJ_a\to MPJ_a=\Delta Ma$, which determines the morphism $\chi_a:La\to Ma$ with $\Delta\chi_a\cdot\lambda^a=\psi J_a$. Its naturality in $a$ is easily checked, and it is clearly uniquely determined by $\psi$ under the condition $\chi P\cdot \lambda=\psi$.
\end{proof}
Using the standard construction of colimits in a cocomplete category, one concludes:

\begin{cor}
Given a cofibration $P:\CE\to\CB$ of small categories, one has a presentation of the colimit of any diagram $T:\CE\to\CX$ in a cocomplete category $\CX$ in terms of a coegalisator diagram
\begin{center}
	$\xymatrix{\coprod_{(u:a\to b)}{\rm colim}TJ_a\ar@/^3pt/[rr]^f\ar@/_3pt/[rr]_g && \coprod_a{\rm colim}TJ_a\ar[rr]^q && {\rm colim}T\;,}$	
	\end{center}
where $TJ_a:\CE_a\to\CX$ is the restriction	of $T$ to the fibre of $P$ at $a\in\CB$.

\end{cor}

\begin{exa}
\emph{Returning to the diagram of the generic colimit of Example \ref{generic}	, $\Phi:\CI\to\SmallCats, i\to\CI/i$, with a small category $\CI$, then $\int_{\circ}\Phi$ is the arrow category $\CI^{\EuRoman 2}$ of $\CI$, considered as a cofibred category over $\CI$ (see \ref{Standard} (3)). Keeping in mind the trivial existence and evaluation of $(\CI/i)$-shaped colimits in any category,} the Twisted Fubini Colimit Formula implies that the colimit of any diagram $T:\CI^{\EuRoman 2}\to\CX$ may be computed by restricting $T$ to the ``diagonal'' of $\CI^{\EuRoman 2}$, as in
$${\rm colim}^{(i,x)\in \CI^{\EuRoman 2}}T(i,x)\cong{\rm colim^{i\in\CI}T(i,1_i)},$$
with either colimit existing when the other exists. 

\emph{Alternatively, this isomorphism of colimits follows also from the fact that the ``diagonal'' embedding $\CI\to\CI^{\EuRoman 2}, i\mapsto(i,1_i),$ is confinal.}
\end{exa}

\section{The Colimit Decomposition Formula: three proofs}
\label{sec:CDF-Proofs}

\subsection{The basic formula and its short first proof}\label{basicCDF}

Given a (small) diagram $\Phi:\CD\to\SmallCats   $, we let $$K_d:\Phi d\lra\CK={\rm colim}\,\Phi\quad (d\in\CD)$$ denote its colimit cocone in the 1-category $\SmallCats$. We re-state the formula given in the Introduction and first provide its short proof as given by Batanin and Berger in Lemma 7.13 of \cite{BataninBerger} (modulo a small correction):

\begin{thm}[Colimit Decomposition Formula]
\label{decomp formula}
For every diagram $X:\CK\to\CX$ in a cocomplete category $\CX$, the $\CK$-shaped colimit of $X$ may be computed as the $\CD$-shaped colimit of the diagram given by the $\Phi d$-shaped colimits of $XK_d$, for every $d\in\CD$:
$${\rm colim}^{\CK}X\cong{\rm colim}^{d\in\CD}({\rm colim}^{\Phi d}XK_d).$$
\end{thm}

\begin{proof}[Proof 1 of the CDF] 
Since, the domain functor $\SmallCats  /\CX\to\SmallCats$ reflects colimits, the given colimit in $\SmallCats$ gives us the colimit cocone 
\begin{center}
$\xymatrix{\Phi d\ar[rr]^{K_d}\ar[rd]_{XK_d} & & \CK\ar[ld]^{X}\\
& \CX & }$
\end{center}
in $\SmallCats/\CX$. By Corollary \ref{colimit preservation}, the inclusion functor $\SmallCats/\CX\to\sf{Diag}^{\circ}(\CX)$ preserves this colimit, and then, by Proposition \ref{functor colim}, it is again preserved by the left-adjoint functor ${\rm colim}:{\sf Diag}(\CX)\to \CX$. But this is precisely the claim of the theorem.
\end{proof}

The dualization of the theorem reads as follows:

\begin{cor}[Limit Recomposition Formula]\label{decomposition dual}
As above, let $\CK$ be the colimit of $\Phi$ in $\SmallCats$, with colimit injections $K_d$. Then, for a diagram $X:\CK\to\CX$ in a complete category $\CX$, the limit of $X$ in $\CX$ may be computed stepwise, according to the formula
$${\rm lim}^{\CK}X\,\cong\,{\rm lim}^{d\in\CD}({\rm lim}^{\Phi d}XK_d).$$
\end{cor}
\begin{proof}
Apply Theorem \ref{decomp formula} to $X^{\rm op}:\CK^{\rm op}\cong{\rm colim}^{d\in\CD}(\Phi d)^{\rm op}\lra\CX^{\rm op}$.
\end{proof}

\subsection{A generalized colimit decomposition formula and the second proof of the CDF}
Our second proof of Theorem \ref{decomp formula} is based on (what turns out to be) a generalization of the decomposition formula. This generalization follows from the lifting of colimits along a bifibration with cocomplete fibres, as given in Corollary \ref{LiftingColimits}. By Proposition \ref{Diagbifibred}, for $\CX$ cocomplete, we may apply this corollary to the bifibration $D^{\CX}:{\sf Diag}^{\circ}(\CX)\to\SmallCats  $, keeping in mind that cocartesian liftings are given by left Kan extensions in this case. Hence, in order to obtain the colimit of a diagram $T:\CD\to{\sf Diag}^{\circ}(\CX)$, we follow the dualization of the construction given in the proof of Theorem \ref{LiftingLimits} and, with
$$\Phi=D^{\CX}T:\CD\to \SmallCats  ,$$
form the colimit $\CK$ of $\Phi$ in $\SmallCats  $, as in \ref{basicCDF}. Then, for every $u:d\to e$ in $\CD$, writing the ${\sf Diag}^{\circ}(\CX)$-object $Td$ as $(\Phi d,\,X_d:\Phi d\to\CX)$ and the morphism $Tu$ as $(\Phi u,\varphi^u)$, as in the triangle on the left,
\begin{center}
$\xymatrix{\Phi d\ar[rr]^{\Phi u}\ar[dr]_{X_d}^{\;\;\varphi^{u}:\Lra} & & \Phi e\ar[dl]^{X_e}\\
& \CX\,,}$
\hfil
$\xymatrix{\Phi d\ar[rr]^{K_d}\ar[dr]_{X_d}^{\;\;\kappa_d:\Lra} & & \CK\ar[dl]^{L_d}\\
& \CX\,,}$
\end{center}
we form the left Kan extensions $L_d:={\rm Lan}_{K_d}X_d$. These extensions come with diagram morphisms as in the triangle on the right and form a $\CD$-shaped diagram in $\CX^{\CK}$ (the fibre of $D^{\CX}$ at $\CK$). Its
colimit $X:={\rm colim}^{d\in\CD}L_d$ has colimit injections $\lambda_d:L_d\to X$. Finally then, the composite ${\sf Diag}^{\circ}(\CX)$-morphisms 
$$(K_d,\lambda_dK_d\cdot\kappa_d):(\Phi d,X_d)\to (\CK,X)$$
 present $(\CK,X)$ as a colimit of $T$ in ${\sf Diag}^{\circ}(\CX)$.

 \begin{rem}\label{rem:joint Kan}
 In the construction above, we may think of $X$ as the {\em joint left Kan extension} of the functors $X_d$ along $K_d$, characterized by the universal property that, for every functor $Y:\CK\to\CX$ and any family $(\mu_d)_{d\in\CD}$ of natural transformations $\mu_d:X_d\to YK_d$ with $\mu_e(\Phi u)\cdot\varphi^u=\mu_d$ for all $u:d\to e$ in $\CD$, there is a unique natural transformation $\beta:X\to Y$ with
$\beta K_d\cdot\lambda K_d\cdot\kappa_d=\mu_d$, for all $d\in\CD$.
\end{rem}
Since the left adjoint functor ${\rm colim}$ of Proposition \ref{functor colim}(1) preserves colimits, we obtain:

\begin{thm}[General Colimit Decomposition Formula]
\label{general decomp formula}
For a cocomplete category $\CX$ and any diagram $T:\CD\to{\sf Diag}^{\circ}(\CX)$ with $D^{\CX}T=\Phi$, writing $Td$ as $(\Phi d,\,X_d:\Phi d\to\CX)$ for all $d\in\CD$, one has
$${\rm colim}^{\CK}X\,\cong\,{\rm colim}^{d\in\CD}({\rm colim}^{\Phi d}\,X_d)$$
in $\CX$, where $\CK={\rm colim}^{d\in\CD}\Phi d$ with colimit injections $K_d$ in $\SmallCats  $ and $X={\rm colim}^{d\in\CD}({\rm Lan}_{K_d}X_d)$ is a colimit of left Kan extensions in the functor category $\CX^{\CK}$.
\end{thm}

\begin{rem}
Although the formula given in Theorem \ref{general decomp formula} may formally look similar to the CDF of Theorem \ref{decomp formula}, there is a crucial difference between the two statements: whereas in Theorem \ref{general decomp formula} $X$ is formed with the help of the given diagrams $X_d$, in Theorem \ref{decomp formula} one proceeds the other way around and defines $X_d$ with the help of $X$ as $XK_d$. 	
\end{rem}

Here is the dualization of Theorem \ref{general decomp formula}, obtainable with the dualization procedure given after Proposition \ref{DiagasGr}.

\begin{cor}[General Limit Recomposition Formula]\label{general formula dual}
For a complete category $\CX$ and any diagram $T:\CD\to{\sf Diag}_{\circ}(\CX),\;d\mapsto (\Phi d,\,X_d:\Phi d\to\CX),$ with $D_{\CX}T=\Phi^{\op}$ and $\Phi:\CD^{\op}\to \SmallCats$, one has
$${\rm lim}^{\CK}X\,\cong\,{\rm lim}^{d\in\CD}({\rm lim}^{\Phi d}\,X_d)$$
in $\CX$, where $\CK={\rm colim}^{d\in\CD}\,\Phi d$ with colimit injections $K_d$ in $\SmallCats   $, and where $X={\rm lim}_{d\in\CD}({\rm Ran}_{K_d}\,X_d)$ is a limit of right Kan extensions in $\CX^{\CK}$.
\end{cor}

Let us now show how the General Colimit Decomposition Formula may be used to derive the CDF of Theorem \ref{decomp formula}:

\begin{proof}[Proof 2 of the CDF]   We are given the diagrams $\Phi:\CD\to\SmallCats$ and $X:\CK\to\CX$, where 
$\CK$ is the colimit of $\Phi$, with colimit injections $K_d:\Phi d\to\CK$.
They allow us to form the diagram $T_X:\CD\to {\sf Diag}^{\circ}(\CX)$, sending $u:d\to e$ in $\CD$ to the ${\sf Diag}^{\circ}(\CX)$-morphism
\begin{center}
$\xymatrix{\Phi d\ar[rr]^{\Phi u}\ar[dr]_{X_d=XK_d}^{\;\;1:\Lra} & & \Phi e\ar[dl]^{XK_e=X_e}\\
& \CX}$
\end{center}
which actually lives in $\SmallCats  /\CX$; so, $T_X$ factors through $\SmallCats   /\CX$. The assertion of Theorem \ref{decomp formula} will follow from an application of Theorem \ref{general decomp formula} to $T_X$, once we have shown the following lemma, formulated in the terminology introduced in Remark \ref{rem:joint Kan}.
\end{proof}

\begin{lem}\label{thm:joint Kan}
$X$ is the joint left Kan extension of the functors $X_d$ along $K_d,\,d\in\CD$.
\end{lem} 

\begin{proof}
We have to check the relevant universal property, as described in Remark \ref{rem:joint Kan}. To this end we consider a functor $Y:\CK\to\CX$ and a family of natural transformations $(\mu_d:XK_d\to YK_d)_{d\in\CD}$ with $\mu_e (\Phi u)=\mu_d$ for all $u:d\to e$ in $\CD$ and must present $\mu_d$ as $\mu_d=\beta K_d\,(d\in\CD)$, for a unique natural transformation $\beta:X\to Y$. But this follows immediately from the fact that the functor 
$\CX^{(-)}:\SmallCats  ^{\rm op}\to\LocSmallCats  $ transforms the colimit cocone $(K_d:\Phi d\to \CK)$ in $\SmallCats  $ into a limit cone $(K_d^*:\CX^{\CK}\to\CX^{\Phi d})$ in $\LocSmallCats   $. Indeed, for $\CX$ small, as a consequence of the cartesian closedness of $\Cat$, this fact follows from the self-ajointness of $\CX^{(-)}:\SmallCats   ^{\rm op}\to\SmallCats  $ (see, for example, Proposition 27.7 in \cite{AHS1990}); for $\CX$ large, {\em pro forma} one has to step temporarily into the colossal category $\mathbb{CAT}$ to generate the needed natural bijective correspondence between families of natural transformations $\mu_d$ and natural transformations $\beta$, in the same way as in the small case.
\end{proof}

\subsection{Obtaining the Colimit Decomposition Formula via Fubini: the third proof of the CDF}
Our third proof of Theorem \ref{decomp formula} takes advantage of the twisted Fubini formula of Theorem \ref{newFubini} and the confinality criterion of Theorem \ref{quotientscofinal} for regular epimorphisms in $\LocSmallCats$.
\begin{proof}[Proof 3 of the CDF] 
Once again, we are given the diagrams $\Phi:\CD\to\SmallCats   $ and $X:\CK\to\CX$, where $\CK$ is the colimit of $\Phi$, with colimit injections $K_d:\Phi d\to\CK$. Via Theorem \ref{thm:Grothendieck-To-StrictColimit}, which is of independent interest, we have the confinal functor $Q\colon \textstyle{\int}_{\circ}\Phi\lra\CK,\;(d,x)\mapsto K_dx$. With Theorem \ref{newFubini}, one concludes
$${\rm colim}^{\CK}X\cong{\rm colim}^{{\int}_{\circ}\Phi}\,XQ\cong {\rm colim}^{d\in\CD}({\rm colim}^{\Phi d}\,XK_d),$$
for every diagram $X:\CK\to \CX$ in a cocomplete category $\CX$.
\end{proof}

\begin{thm}
\label{thm:Grothendieck-To-StrictColimit}
For a functor $\Phi\from \CD\to \SmallCats$, with $\CD$ small, the comparison functor
\begin{equation*}
Q\colon \textstyle{\int}_{\circ}\Phi\lra\CK=\CoLim\ \Phi,\;((u,f):(d,x)\to(e,y))\longmapsto (K_ef:K_dx=K_e(\Phi u)x\to K_ey).
\end{equation*}
from the lax to the strict colimit of $\Phi$ (see Corollary {\em \ref{universalchar dualGr}})     is a confinal regular epimorphism in $\SmallCats$.
\end{thm}

\begin{proof}
Let $\coprod\Phi$ denote the coproduct of the categories $\Phi d$, $d\in \CD$, in $\SmallCats$, with injections $I_d:\Phi d\to \coprod\Phi$. As in any cocomplete category, the colimit $\CK$ of $\Phi$ may be constructed as the joint coequalizer $C:\coprod\Phi\to \CK$ of the family of pairs $(I_d,\;I_e\Phi u)$, indexed by all the morphisms $u:d\to e$ in $\CD$. In accordance with the notation of Proposition \ref{quotient}, the equivalence relation induced by $C$ on the set of objects of $\coprod\Phi$ is the least equivalence relation $\approx_C$ with $I_d(x)\approx_C I_e\Phi u(x)$, for all $u:d\to e$ and objects $x\in\Phi d$. Hence, one has $I_dx\approx_C I_ey$ in $\coprod \Phi$ precisely when there are morphisms $u_i:d_i\to e_i,\; v_i:d_{i+1}\to e_i$ in $\CD$ and objects $x_i\in\Phi d_i$ with
\begin{center}
$\xymatrix{d=d_0\ar[r]^{u_0} & e_0 & d_1\ar[l]_{v_0}\ar[r]^{u_1} & e_1 &...\;d_n\ar[l]_{v_1}\ar[r]^{u_n}& e_n=e\\
}$	
\end{center}
and $x=x_0,\;\Phi u_i(x_i)=\Phi v_i(x_{i+1})\;\;(i=0, ..., n-1)
,\;\Phi u_n(x_n)=y$.

Since $CI_d=K_d$ for all objects $d\in\CD$, the equivalence relation $\approx_Q$ induced by $Q$ on the set of objects of $\textstyle{\int}_{\circ}\Phi$ is described by
$$ (d,x)\approx_Q(e,y)\iff I_d(x)\approx_C I_e(y).$$
Hence, any two objects $(d,x),(e,y)$ in the same fibre of $Q$ are linked by a string of morphisms
\begin{center}
$\xymatrix{(d,x)=(d_0,x_0)\ar[r]^{\delta_{x_0}^{u_0}} &  (e_0,\Phi u_0(x_0)) & (d_1,x_1)\ar[l]_{\quad\delta_{x_1}^{v_0}}\ar[r]^{\delta_{x_1}^{u_1}} & ...\ar[r]^{\delta_{x_n}^{u_n}\qquad\qquad} & (e_n,\Phi u_n(x_n))=(e,y)
}$	
\end{center}
in $\textstyle{\int}_{\circ}\Phi$. These morphisms actually live in the same fibre of $Q$ since $Q\delta_{x_i}^{u_i}=K_{d_i}(1_{\Phi u_i(x_i)})$ and $Q\delta_{x_{i+1}}^{v_i}=K_{d_{i+1}}(1_{\Phi v_i(x_{i+1})})$ are identity morphisms in $\CK$. Consequently, the fibres of $Q$ are connected.

In order for us to conclude with Theorem \ref{quotientscofinal} that $Q$ is confinal, we must also confirm that $Q$ is actually a regular epimorphism. But this is a generally valid consequence of the fact that $Q$ is a second factor of the regular epimorphism $C$, with the first factor $E$ being easily recognized as an epimorphism in $\SmallCats$:
\begin{center}
$\xymatrix{& \textstyle{\int}_{\circ}\Phi\ar[rd]^Q & \\
\coprod\Phi\ar[ru]^E\ar[rr]^C && \CK\;.\\
}$	
\end{center}
 Explicitly, $E$ is the functor with $EI_d=J_d:\Phi d\to\textstyle{\int}_{\circ}\Phi,\;x\mapsto (d,x),$ for all $d\in\CD$ and $x\in\Phi d$.
\end{proof}

\begin{rem}
We note that, unlike $Q$, the regular epimorphism $C$ as above fails to be confinal, as soon as there exists a morphism with distinct domain and codomain in $\CD$.
\end{rem}

\section{Extending the Guitart adjunction, making Grothendieck a left adjoint}
  We return to the Guitart adjunction
  \begin{center}
  $\xymatrix{\LocSmallCats\ar@/^0.5pc/[rr]^{\sf Diag^{\circ}\quad} & \top & \LocSmallCats/\SmallCats\;.\ar@/^0.5pc/[ll]^{\int_{\circ}\quad}
}$
\end{center}
of Theorem \ref{Guitart}. Realizing that $\int_{\circ}\Phi$ is cofibred over the domain of any functor $\Phi:\CB\to\SmallCats$, so that the 2-functor $\int_{\circ}$ actually takes values in the morphism category $\LocSmallCats^{\EuRoman{2}}$ of $\LocSmallCats$, in this section we indicate how to extend the 2-functor ${\sf Diag}^{\circ}$ and, in fact, the entire 2-adjunction, from $\LocSmallCats$ to $\LocSmallCats^{\EuRoman{2}}$. We also make precise that the 2-functor $\int_{\circ}$ may actually be defined on the ``lax slice'' $\LocSmallCats//\SmallCats$, rather than on its subcategory $\LocSmallCats/\SmallCats$.

\subsection{The diagram category ${\sf Diag}^{\circ}(P)$ of a functor $P:\CE\to\CB$.}\label{Diag(P)}

We start by giving a fibered version of the formation of the category ${\sf Diag}^{\circ}(\CX)$, replacing $\CX$ by a functor $P$, in such a way that the original construction is described as ${\sf Diag}^{\circ}(\CX\to {\EuRoman 1})$, with $\EuRoman 1$ the terminal category. Hence, for any functor $P:\CE\to\CB$ we define the category
${\sf Diag}^{\circ}(P),$
as follows:
\begin{itemize}
\item 
{\em objects} are triples $(a,\CI,X)$, with $a$ an object in $\CB$ and $X:\CI\to\CE_a=P^{-1}(a)$ a functor of a small category $\CI$;
\item a {\em morphism} $(u,F,\varphi):(a,\CI,X)\to(b,\CJ,Y)$ is given by a morphism $u:a\to b$ in $\CB$, a functor $F:\CI\to\CJ$, and a natural transformation $\varphi:J_aX\lra J_bYF$ with $P\varphi=\Delta u$ (the constant transformation with value $u$); $J_a, J_b$ are inclusion functors, as in
 \begin{center}
 $
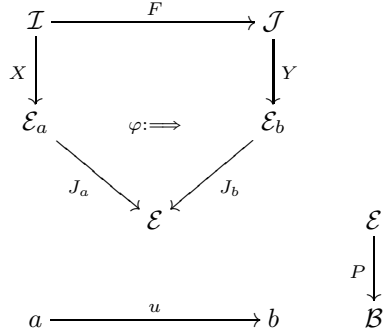
\xymatrix{\CI\ar[rr]^{F}\ar[d]_{X} & & \CJ\ar[d]^{Y} &\\
\CE_a\ar[dr]_{J_a} & \scriptstyle{\varphi:\Lra} & \CE_b\ar[dl]^{J_b} &\\
& \CE & &\CE\ar[d]_{P}\\
a\ar[rr]^{u} & & b & \CB\\
}$
\end{center}
\item {\em composition}: $(v,G,\psi)\cdot(u,F,\varphi) = (v\cdot u,GF, \psi F\cdot\varphi)$.
\end{itemize}
The category ${\sf Diag}^{\circ}(P)$ comes equipped with the obvious functors

\medskip

\begin{tabular}{l}
$B^P: {\sf Diag}^{\circ}(P)\to\CB,\quad (u,F,\varphi)\mapsto u,$\\
$D^P:{\sf Diag}^{\circ}(P)\to \SmallCats  ,\quad (u,F,\varphi)\mapsto F,$\\
$E^P:\CE\to {\sf Diag}^{\circ}(P), \quad x\mapsto (Px,\EuRoman 1,\Delta x\!:\!{\EuRoman 1}\to\CE_{Px}),\; (f:x\to y)\mapsto (Pf,{\rm Id}_{\EuRoman 1},\Delta f).$\\
\end{tabular}

\medskip
\noindent $E^P$ is a full embedding which makes the diagram
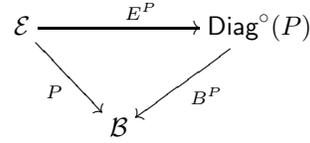
\begin{center}
$\xymatrix{\CE\ar[rr]^{E^P}\ar[rd]_{P} & & {\sf Diag}^{\circ}(P)\ar[ld]^{B^P}\\
& \CB & \\
}$	
\end{center}
commute. For $D^P$ and $B^P$ one easily proves:

\begin{prop}
 {\rm (1)} $D^P$ is a split fibration.
 
 {\rm (2)} $B^P$ is a (split) (co)fibration if, and only if, $P$ has the corresponding property.
\end{prop}
 
 \begin{proof}
 (1) The $D^P$-cartesian lifting of $F:\CI\to\CJ$ at $(b,\CJ,Y)$ may be taken to be $(1_b,F,1_{J_bYF}):(b,\CI,YF)\to(b,\CJ,Y).$
 
 (2) For a (split) fibration $P$, with $P$-cartesian liftings denoted by $\theta$, we claim that the $B^P$-cartesian lift of $u:a\to b$ at $(b,\CJ,Y)$	
may be taken to be $(u, {\rm Id}_{\CJ},\theta^{u}Y)$. Indeed, the required universal property, as depicted by the left side of the diagram below, follows from a pointwise application of the corresponding $P$-cartesian property, as depicted on the right side:
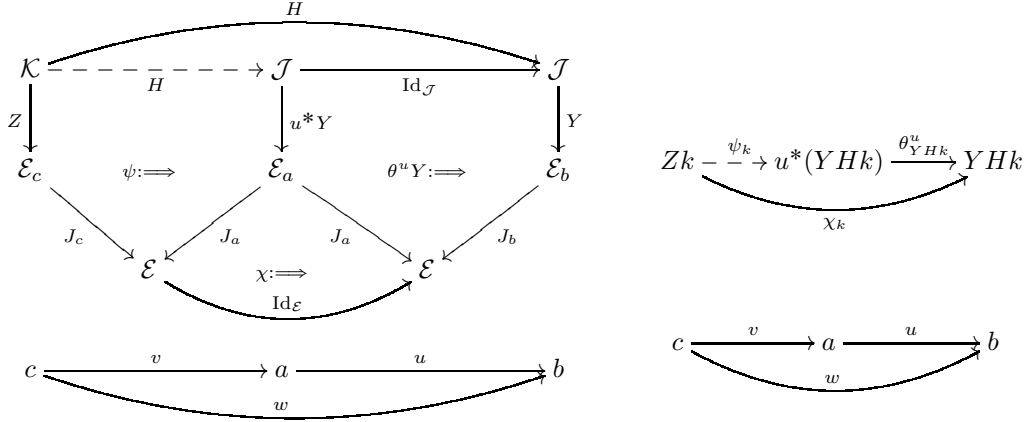
\begin{center}
$\xymatrix{\CK\ar@{-->}[rr]_{H}\ar@/^1.5pc/[rrrr]^{H}\ar[d]_Z & & \CJ\ar[rr]_{{\rm Id}_{\CJ}}\ar[d]^{u^*Y} & & \CJ\ar[d]^{Y} \\
\CE_c\ar[dr]_{J_c} & \scriptstyle{\psi:\Lra} & \CE_a\ar[dl]^{J_a}\ar[dr]_{J_a} & \scriptstyle{\theta^uY:\Lra }& \CE_b\ar[dl]^{J_b}  \\
& \CE\ar@/_1.5pc/[rr]^{{\rm Id}_{\CE}} & \scriptstyle{\chi:\Lra} & \CE & \\
c\ar[rr]^{v}\ar@/_1.5pc/[rrrr]^{w\quad} & & a\ar[rr]^{u} & & b }$
\hfil	
$\xymatrix{& & \\
Zk\ar@{-->}[r]^{\psi_k\quad}\ar@/_1.5pc/[rr]_{\chi_k} & u^*(YHk)\ar[r]^{\quad\theta^u_{YHk}} & YHk \\
& & \\
c\ar[r]^{v}\ar@/_1.5pc/[rr]^{w\;} & a\ar[r]^{u} & b}$
\end{center}
Conversely, if $B^P$ is a (split) fibration, a $P$-cartesian lifting of $u:a\to b$ at $y$ may be realized as a $B^P$-cartesian lifting of $u:a\to b$ at $(b,\EuRoman 1,\Delta y:{\EuRoman 1}\to\CE_B)$.

When $P$ or $B^P$ is a (split) opfibration, the proof proceeds analogously.
\end{proof}

We should clarify further the interdependency of the diagram constructions for categories and for functors. Trivially, for a category $\CX$, one has ${\sf Diag}^{\circ}(\CX)\cong {\sf Diag}^{\circ}(\CX\to{\EuRoman 1})$. Less trivially, when $P:\CE\to\CB$ is a split cofibration, with the help of the Grothendieck construction we may build ${\sf Diag}^{\circ}(P)$ from the categories ${\sf Diag}^{\circ}(\CE_a)\;(a\in\CB)$, as follows. Consider the functor
$$\Theta_P:\CB\to \LocSmallCats,\;(u:a\to b)\mapsto [u_!(-):{\sf Diag}_{\circ}(\CE_a)\to{\sf Diag}_{\circ}(\CE_b)],$$  
 where the functor $u_!(-)$ maps $(F,\psi):(\CI,X)\to (\CJ,X')$ in ${\sf Diag}^{\circ}(\CE_a)$ to $(F,u_!\psi)$ in ${\sf Diag}^{\circ}(\CE_b)$:
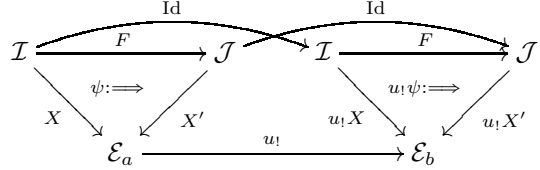
 \begin{center}
 $\xymatrix{\CI\ar[rr]^{F}\ar[rd]_{X}^{\;\;\psi:\Lra}\ar@/^1.0pc/[rrr]^{\rm Id} & & \CJ\ar[ld]^{X'}\ar@/^1.0pc/[rrr]^{\rm Id} & \CI\ar[rr]^{F}\ar[rd]_{u_!X}^{\;u_!\psi:\Lra} & & \CJ\ar[ld]^{u_!X'}\\
 & \CE_a\ar[rrr]^{u_!} & & & \CE_b & 
 }$	
 \end{center}
 \begin{prop}\label{Diag as Gr}
If $P:\CE\to\CB$ is a split cofibration, then, as a cofibred category over $\CB$, the category ${\sf Diag}^{\circ}(P)$ is isomorphic to the dual Grothendieck category $\int_{\circ}\Theta_P$, by an isomorphism that maps objects identically.
 \end{prop}
 
 \begin{proof} For every morphism $u:a\to b$ in $\CB$, where $a=Px$ and $b=Py$ with $x,y\in\CE$, one has the natural bijection
 $$\CE_{b}(u_!x,y)\to \CE_u(x,y),\;(f:u_!x\to y)\mapsto f\cdot \delta^u_x,$$	
 where $\CE_u(x,y)=\CE(x,y)\cap P^{-1}(u)$ (see \ref{Grothendieck fibrations}). Given functors $X:\CI\to\CE_a,\;Y:\CJ\to\CE_b,\;F:\CI\to\CJ$, exploiting the above bijection for $x=Xi,\;y=YFi\;(i\in\CI)$, one obtains the natural bijection
 $$\{\psi\;|\;\psi :u_!X\to YF \text{ nat.tr.}\}\to\{\varphi\;|\;\varphi: J_aX\to J_bYF\text{ nat.tr.}, P\varphi=\Delta u\},\;\psi\mapsto J_b\psi\cdot\delta^uX.$$
 Equivalently, writing $(a,X)$ instead of $(a,(\CI,X))$, we have the natural bijection
 $$\{\psi\;|\;(u,F,\psi)\in(\textstyle{\int}_{\circ}\Theta_P)((a,X),(b,Y))\}\to \{\varphi\;|\;(u,F,\varphi)\in{\sf Diag}^{\circ}(P)((a,X),(b,Y))\},\;\psi\mapsto J_b\psi\cdot\delta^uX.$$
 With objects kept fixed, this defines a bijective functor $\int_{\circ}\Theta_P\to {\sf Diag}^{\circ}(P)$ which obviously commutes with the $\CB$-valued split cofibrations:
 \begin{center}
$\xymatrix{\int_{\circ}\Theta_P\ar[rd]_{\Pi_{\Theta_P}}\ar[rr]^{\cong} &  &{\sf Diag}^{\circ}(P)\ar[ld]^{B^P}\\
& \CB & \\
}$

 \end{center}
 \end{proof}
 In the same way as one arrives at the definition of morphisms of ${\sf Diag}_{\circ}(\CX)$ once those of ${\sf Diag^{\circ}}(\CX)$ have been defined, one may also define the morphisms of the category ${\sf Diag}_{\circ}(P)$; that is: keeping the same objects, but inverting the direction of the functor $F$ while keeping the direction of the natural transformation $\varphi$ in the definition of a morphism $(u,F,\varphi)$ in ${\sf Diag}^{\circ}(P)$, one defines the morphisms of the category ${\sf Diag}_{\circ}(P)$. The dualization of Proposition \ref{Diag as Gr} then says that, when $P$ is a split fibration, ${\sf Diag}_{\circ}(P)$ is isomorphic to $\int^{\circ}\Theta^P$ as a fibred category over $\CB$, with
 $$\Theta^P:\CB^{\op}\to \LocSmallCats,\;(u:a\to b \text{ in }\CB)\mapsto [u^*(-):{\sf Diag}_{\circ}(\CE_b)\to{\sf Diag}_{\circ}(\CE_a)]\;.$$
 
\subsection{Review of the 2-categories $\LocSmallCats^{\EuRoman 2},\; \LocSmallCats//\SmallCats$ and $\LocSmallCats/\SmallCats$.}\label{2categories} 
 In order to extend the transitions
 $$(P:\CE\to\CB)\mapsto (D^P:{\sf Diag}^{\circ}(P)\mapsto\SmallCats),\qquad (\Phi:\CB\to\SmallCats)\mapsto(\Pi_{\Phi}:\textstyle{\int}_{\circ}\Phi\to \CB),$$
 2-functorially, we form the 2-categories $\LocSmallCats^{\EuRoman 2},\; \LocSmallCats//\SmallCats$ and $\LocSmallCats/\SmallCats$ in a standard manner:
 \begin{itemize}
 	\item The objects of $\LocSmallCats^{\EuRoman 2}$ are functors $P:\CE\to\CB$ of 1-categories (= $\LocSmallCats$-objects); a morphism $(S,T):P\to Q$ is given by functors  that make the square on the left of the diagram
 	\begin{center}
 	$\xymatrix{\CE\ar[d]_{P}\ar[r]^{T} & \CF\ar[d]^{Q}\\
 	\CB\ar[r]^{S} & \CC}
 	\hfil\xymatrix{\CE\ar[d]_{P}\ar@/^0.7pc/[rr]^{T}\ar@/_0.7pc/[rr]_{T'} & {\scriptstyle{\beta}}\Downarrow\quad &\CF\ar[d]^{Q}\\
 	\CB\ar@/^0.7pc/[rr]^{S}\ar@/_0.7pc/[rr]_{S'} & {\scriptstyle{\alpha}}\Downarrow\quad&\CC
 	}$
 	\end{center}
 commutative; and a 2-cell $(\alpha,\beta):(S,T)\Lra(S',T')$ is a pair of natural transformations $\alpha: S\to S',\;\beta:T\to T'$ with $Q\beta=\alpha P$; their horizontal and vertical compositions are inherited from the 2-category $\LocSmallCats$ in each of the two components.	
    \item The objects of $\LocSmallCats//\SmallCats$ are functors $\Phi:\CB\to\SmallCats$ of 1-categories; for $\Psi:\CC\to\SmallCats$, a
    morphism $(\Sigma,\tau):\Phi\to\Psi	$ is given by a functor $\Sigma:\CB\to\CC$ and a natural transformation $\tau:\Phi\to\Psi\Sigma$; a 2-cell $(\sigma,\mu):(\Sigma,\tau)\Lra(\Sigma',\tau')$ is a natural transformation $\sigma:\Sigma\to\Sigma'$ together with a modification\footnote{For this term to make sense here, we consider the ordinary category $\CB$ as a discrete 2-category ({\em i.e.}, as having identical 2-cells, so that $\Phi, \Psi \Sigma^{(')}$ become 2-functors and $\tau,\tau'$ 2-natural transformations, for free.)}\;$\mu:\Psi\sigma\cdot\tau\to\tau'$; this means that, for every object $a\in\CB$, we have a natural transformation $\mu_a:(\Psi\sigma_a)\tau_a\to\tau'_a$, such that, for every morphism $u:a\to b$ in $\CB$, the following two natural transformations coincide:
    $$(\Psi\Sigma'u)\mu_a: (\Psi\Sigma'u)(\Psi\sigma_a)\tau_a\to(\Psi\Sigma'u)\tau'_a\quad\text{and} \quad \mu_b(\Phi u):(\Psi \sigma_b)\tau_b(\Phi u)\to\tau'_b(\Phi u).$$
    (These two transformations have the same domain and codomain, by the naturality of $\sigma,\tau$.)
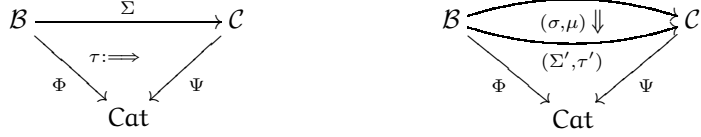
 \begin{center}   
$\xymatrix{\CB\ar[rr]^{\Sigma}\ar[dr]_{\Phi}^{\;\;\tau:\Lra} & & \CC\ar[dl]^{\Psi}\\
& \SmallCats}
\hfil \xymatrix{\CB\ar@/^0.7pc/[rr]^{(\Sigma,\tau)}\ar@/_0.7pc/[rr]_{(\Sigma',\tau')}\ar[dr]_{\Phi} & {\scriptstyle (\sigma,\mu)}\!\Downarrow & \CC\ar[dl]^{\Psi}\\
& \SmallCats}$
\end{center}
The horizontal and vertical compositions are such that the $\LocSmallCats$-valued assignment $(\Sigma,\tau)\mapsto\Sigma$ becomes a 2-functor.

{\em Important Note:}
  $\LocSmallCats//\SmallCats$ has a richer 2-categorical structure than ${\sf DIAG}^{\circ}(\CX)$ (as defined in Remark \ref{2functorStrict}(2)),
which is due to the fact that, when considering 2-cells in $\LocSmallCats//\SmallCats$, we are envoking the 2-categorical structure of $\SmallCats$, in order to form modifications. These are all identities when $\SmallCats$ is replaced by a 1-category $\CX$, considered as a 2-category with identity 2-cells.
\item $\LocSmallCats/\SmallCats$, as already defined in Remark \ref{2functorStrict}(1),  
is the sub-2-category of 
    $\LocSmallCats//\SmallCats$ whose morphisms $(\Sigma,\tau):\Phi\to\Psi$ satisfy $\tau=1_{\Phi}$, so that  $\Phi=\Psi\Sigma$; consequently, a 2-cell $\sigma:\Sigma\Lra\Sigma'$ in $\LocSmallCats/\SmallCats$ is just a natural transformation satisfying $\Psi\sigma=1_{\Phi}$. 
    \end{itemize}
    
    \begin{prop}\label{2functorsDiagGr}
  The transitions $P\mapsto D^P$  and $\Phi\mapsto \Pi_{\Phi}$ are the object assignments of 2-functors
  $${\sf Diag}^{\circ}:\LocSmallCats^{\EuRoman 2}\to\LocSmallCats/\SmallCats\quad\text{ and }\quad \textstyle{\int}_{\circ}:\LocSmallCats//\SmallCats\to \LocSmallCats^{\EuRoman 2}.$$   	
    \end{prop}
    
    \begin{proof} We just describe the assignments for morphisms and 2-cells and leave all routine verifications to the reader. ${\sf Diag}^{\circ}$ assigns to a morphism $(S,T):P\to Q$ the functor
    $$\Sigma={\sf Diag}^{\circ}(S,T):{\sf Diag}^{\circ}(P)\to{\sf Diag}^{\circ}(Q)$$ which, in turn, is given by the morphism assignment
    $$((u,F,\varphi):(a,\CI,X)\to(b,\CJ,Y))\quad\mapsto\quad ((Su,F,T\varphi):(Sa,\CI,T_aX)\to(Sb,\CJ,T_bY));$$
    here $T_a$ is the restriction of $T$ that makes the square of the diagram below commute.
    \begin{center}
    $
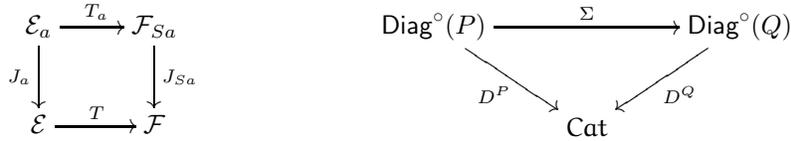
\xymatrix{\CE_a\ar[r]^{T_a}\ar[d]_{J_a} & \CF_{Sa}\ar[d]^{J_{Sa}}\\
    \CE\ar[r]^{T} & \CF\\
    }\hfil
    \xymatrix{{\sf Diag}^{\circ}(P)\ar[rr]^{\Sigma}\ar[rd]_{D^P} & & {\sf Diag}^{\circ}(Q)\ar[ld]^{D^Q}\\
    & \SmallCats &\\}$	
    \end{center}
    Trivially, the triangle on the right commutes as well, so that $\Sigma$ is indeed a morphism in $\LocSmallCats/\SmallCats$. For a 2-cell $(\alpha,\beta):(S,T)\Lra(S',T')$, one defines the natural transformation
    $$\sigma={\sf Diag}^{\circ}(\alpha,\beta):\Sigma={\sf Diag}^{\circ}(S,T)\lra \Sigma'={\sf Diag}^{\circ}(S',T')\quad\text{by}$$
    $$\sigma_{(a,\CI,X)}=(\alpha_a,{\rm Id}_{\CI},\beta J_aX):\Sigma(a,\CI,X)=(Sa,\CI,T_aX)\lra \Sigma'(a,\CI,X)=(Sa',\CI,T_a'X),$$
    for all objects $(a,\CI,X)$ in ${\sf Diag}^{\circ}(P)$. Note that $\sigma_{(a,\CI,X)}$ is well defined since $Q(\beta J_aX)=\alpha PJ_aX=\Delta\alpha_a$.
    
    \begin{center}
 $\xymatrix{\CI\ar[rr]^{{\rm Id}_{\CI}}\ar[d]_{X} & & \CI\ar[d]^{X} &\\
\CE_a\ar[r]^{J_a}\ar[d]_{T_a} & \CE\ar@/_1.0pc/[d]_{T}^{\beta:\Lra}\ar@/^1.0pc/[d]^{T'}& \CE_a\ar[l]_{J_a}\ar[d]^{T'_a} &\\
 \CF_{Sa}\ar[r]_{J_{Sa}} & \CF & \CF_{S'a}\ar[l]^{J_{S'a}} &\CF\ar[d]_{Q}\\
Sa\ar[rr]^{\alpha_a} & & S'a & \CC \\
}$
\end{center}       	 
     $\int_{\circ}$ assigns to the ($\LocSmallCats//\SmallCats$)-
     morphism $(\Sigma,\tau):\Phi\to\Psi$ 
     the   $\LocSmallCats^{\EuRoman 2}$-morphism given by the square
    \begin{center}
    $\xymatrix{\int_{\circ}\Phi\ar[r]^{T}\ar[d]_{\Pi_{\Phi}} & \int_{\circ}\Psi\ar[d]^{\Pi_{\Psi}}\\
    \CB\ar[r]^{\Sigma} & \CC\;,\\
    }$	
    \end{center}
    where the functor $T$ maps a morphism $(u,f):(a,x)\to (b,y)$ to 
    $$T(u,f)=(\Sigma u,\tau_b(f)): T(a,x)= (\Sigma a, \tau_a(x))\lra T(b,y)=(\Sigma b,\tau_b(y));$$
    note that the naturality of $\tau$ makes $\tau_b(f)$ have the correct domain, namely  $\tau_b(\Phi u(x))=\Psi(\Sigma u)(\tau_a(x))$.
    
     Given a 2-cell $(\sigma,\mu):(\Sigma, \tau)\Lra (\Sigma',\tau')$ in $\LocSmallCats//\SmallCats$, we need to define a natural transformation $\beta:T\to T'$, where $(\Sigma,T)=\int_{\circ}(\Sigma,\tau),\; (\Sigma',T')=\int_{\circ}(\Sigma',\tau')$, that satisfies $\Pi_{\Psi}\beta=\sigma\Pi_{\Phi}$. To this end, for $(a,x)\in\int_{\circ}\Phi$, we put $\beta_{(a,x)}=(\sigma_a,(\mu_a)_x)$, which is a well-defined  morphism $T(a,x)\to T'(a,x)$ in $\int_{\circ}\Psi$ since $(\mu_a)_x$ is a morphism $\Psi\sigma_a(\tau_a(x))\to\tau'_a(x)$ in $\Psi(\Sigma'a)$.
\end{proof}

\subsection{The extended Guitart adjunction}

We are now ready to prove that the restriction of the 2-functor $\int_{\circ}:\LocSmallCats//\SmallCats\to\LocSmallCats^{\EuRoman{2}}$ to $\LocSmallCats/\SmallCats$ is left adjoint to ${\sf Diag}^{\circ}$ of Proposition \ref{2functorsDiagGr}:

\begin{thm}\label{fundamental}\; $\int_{\circ}\dashv {\sf Diag}^{\circ}:\LocSmallCats^{\EuRoman 2}\to \LocSmallCats/\SmallCats$ is an adjunction of 2-functors.
\end{thm}

\begin{proof}
In generalization of the adjunction established in the proof of Theorem \ref{Guitart}, for all functors $\Phi: \CB\to\SmallCats,\;Q:\CF\to\CC$ we must, naturally in  $\Phi$ and $Q$,  establish functors
\begin{center}
$\xymatrix{\LocSmallCats^{\EuRoman 2}({\Pi_{\Phi},Q)}\ar@/^0.2pc/[r]^{\widehat{\Box}\quad} & (\LocSmallCats/\SmallCats)(\Phi,D^Q)\ar@/^0.2pc/[l]^{\widecheck{\Box}\quad }\\
}$ 	
\end{center}
that are inverse to each other. In doing so, we follow the notation used in the proof of Proposition \ref{2functorsDiagGr}, with slight adjustments. In particular, we write $(c,Z)$ instead of $(c,\CK,Z)$ for objects of ${\sf Diag}^{\circ}(Q)$.

``$\to$'': First, given the commutative square on the left, we must define the functor $\Sigma=\widehat{(S,T)}$ of the commutative triangle on the right:
\begin{center}
$\xymatrix{\int_{\circ}\Phi\ar[d]_{\Pi_{\Phi}}\ar[r]^{\; T} & \CF\ar[d]^{Q} & & \CB\ar[rr]^{\Sigma\quad}\ar[rd]_{\Phi} & & {\sf Diag}^{\circ}(Q)\ar[dl]^{D^Q}\\
\CB\ar[r]^{S} & \CC & & & \SmallCats & \\
}$	
\end{center}
$\Sigma$ sends an object $a\in\CB$ to the ${\sf Diag}^{\circ}(Q)$-object $(Sa,T_a)$, with the functor 
$$T_a:\Phi a\to \CF_{Sa}\;,\;\;(f:x\to x')\mapsto (T(1_a,f):T(a,x)\to T(a,x')),$$
and a morphism $u:a\to b$ in $\CB$ is sent to the ${\sf Diag}^{\circ}(Q)$-morphism $$\Sigma u=(Su, \Phi u,T\delta^u): \Sigma a=(Sa,T_a)\longrightarrow\Sigma b=(Sb,T_b),$$
where $\delta^u_x=(u,1_{\Phi u(x)}):(a,x)\to(b,\Phi u(x))$ is the $\Pi_{\Phi}$-cocartesian lift of $u$ at $x\in\Phi a$. The commutativity of the square above guarantees $Q(T\delta^u)=S\Pi_{\Phi}\delta^u=\Delta Su$, as required.
\begin{center}
$\xymatrix{\Phi a\ar[rr]^{\Phi u}\ar[d]_{T_a}\ar[rd]_{J_a}^{\qquad\delta^u:\Lra} & & \Phi b\ar[d]^{T_b}\ar[ld]^{J_b}\\
\CF_{Sa}\ar[rd]_{J_{Sa}} & {\int}_{\circ}\Phi\ar[d]^{T} & \CF_{Sb}\ar[ld]^{J_{Sb}}\\
& \CF & \\
}$	
\end{center}
We note that the emerging functor $\Sigma$ satisfies $D^Q\Sigma=\Phi$, as required. To establish the functoriality of $\widehat{\Box}$,
for a 2-cell $(\alpha,\beta):(S,T)\Lra(S',T')$ one defines the natural transformation 
$$\sigma=\widehat{(\alpha,\beta)}: \Sigma=\widehat{(S,T)}\lra\Sigma'=\widehat{(S',T')}$$ by
$\sigma_a=(\alpha_a, {\rm Id}_{\Phi a},\beta_{(a,-)}):(Sa,T_a)\lra(S'a,T'_a)$ in ${\sf Diag}^{\circ}(Q)$, with $\beta_{(a,x)}:T_ax\to T'_ax\;(x\in \Phi a)$. Note that one has $D^Q\sigma=1_{\Phi}$, as required. 
\medskip

``$\leftarrow$'': Conversely now, given the functor $\Sigma$ of the commutative triangle on the right of the above diagram, we must define the pair of functors $(S,T)=\widecheck{\Sigma}$, making the square on the left commute. With $B^Q: {\sf Diag}^{\circ}(Q)\to \CC$ as in \ref{Diag(P)}, we put $S=B^Q\Sigma$. For objects $a\in \CB$ and $x\in\Phi a$, having the functor $\Sigma a: \Phi a\to\CF_{Sa}$, we put $T(a,x)=\Sigma a(x)$. For a morphism $u:a\to b$ in $\CB$, we may write the ${\sf Diag}^{\circ}(Q)$-morphism $\Sigma u:\Sigma a\to \Sigma b$ as
$$\Sigma u= (Su,\Phi u,\varphi^u): \Sigma a\to\Sigma b, \;\text{with}\; \varphi^u:J_{Sa}\,\Sigma a\to J_{Sb}\,\Sigma b\,\Phi u\;\text{and}\; Q\varphi^u=\Delta Su.$$
For $(u,f):(a,x)\to(b,y)$ in $\int_{\circ}\Phi$, we can now define $T(u,f):T(a,x)\to T(b,y)$ as the composite arrow
\begin{center}
$\xymatrix{\Sigma a(x)\ar[r]^{\varphi^u_x\quad} & \Sigma b(\Phi u(x))\ar[r]^{\quad\Sigma b(f)} & \Sigma b(y)\;.
}$	
\end{center}
Its second morphism is $Q$-vertical since the functor $\Sigma b$ takes its values in $\CF_{Sb}$. Consequently,
$$QT(u,f)=Q(\varphi^u_x)=Su=S\Pi_{\Phi}(u,f),$$
so that we indeed have $QT=S\Pi_{\Phi}$. Clearly, $\widecheck{\Box}$ becomes a functor since, for a 2-cell $\sigma:\Sigma\Lra\Sigma'$, we may define $(\alpha,\beta)=\widecheck{\sigma}:(S,T)\!=\!\widecheck{\Sigma}\Lra (S',T')\!=\!\widecheck{\Sigma'}$, by putting $\alpha=B^Q\sigma$, then writing $\sigma_a$ as $(\alpha_a,\,{\rm Id}_{\Phi a},\,\beta^a\!:\!J_{Sa}\Sigma a\to J_{S'a}\Sigma' a)$ and finally setting $\beta_{(a,x)}=\beta^a_x$, for all objects $(a,x)\in {\sf Gr}_{\circ}(\Phi)$.

\medskip

Finally, we must confirm that the functors $\widehat{\Box},\;\widecheck{\Box}$ are inverse to each other. First, given $(S,T):\Pi_{\Phi}\to Q$ in $\LocSmallCats^{\EuRoman 2}$, let $\Sigma =\widehat{(S,T)}$ and $(\tilde{S},\tilde{T})=\widecheck{\Sigma}$. Then, trivially, $\tilde{S}=B_Q\Sigma=S$, and for all $(a,x)\in\int_{\circ}\Phi$ one has $\tilde{T}(a,x)=\Sigma a(x)=T_a(x)=T(a,x)$; likewise, $\tilde{T}(1_a,h)=T(1_a,h)$ for every morphism $h$ in $\Phi a$.
For an arbitrary morphism $(u,f):(a,x)\to (b,y)$ in $\int_{\circ}(\Phi)$, with its (cocartesian,vertical)-factorization $(u,f)=(1_b,f)\cdot(u,1_{\Phi u(x)})=J_b(f)\cdot\delta^u_x$, and with $\varphi^u$ as in ``$\leftarrow$'', one obtains
$$\tilde{T}(u,f)=\Sigma b(f)\cdot \varphi^u_x=T_b(f)\cdot T\delta^u_x=T(J_b(f)\cdot\delta^u_x)=T(u,f).$$
This shows $\tilde{T}=T$.
 Conversely, given $\Sigma:\Phi\to D^Q$ in $\LocSmallCats/\SmallCats$, one argues very similarly that the transitions $\Sigma\mapsto \widecheck{\Sigma}=(S,T)\mapsto\tilde{\Sigma}=\widehat{(S,T)}$ actually return $\Sigma$. Indeed, having $u: a\to b$ and writing $\Sigma u=(S,\Phi u, \varphi^u)$ as above, we deduce
 $$T\delta^u_x=T(u,1_{\Phi u(x)})=\Sigma b(1_{\Phi u(x)})\cdot \varphi^u_x=\varphi^u_x,$$ 
 for all $x\in\Phi a$ and, hence, $\tilde{\Sigma} u=(Su, \Phi u,T\delta^u)=(Su,\Phi u, \varphi^u)=\Sigma u$.
 
  This then shows that $\widehat{\Box},\, \widecheck{\Box}$ are inverse to each other on the objects of their (co)domains. Showing that the same happens for the morphisms ({\em i.e.}, the 2-cells in $\LocSmallCats^{\EuRoman 2}$ and $\LocSmallCats/\SmallCats$) involves only easy routine checks.
\end{proof}
\begin{rem}
The Guitart adjunction of Theorem \ref{Guitart}
follows from the extended Guitart adjunction of Theorem \ref{fundamental}, with the help of (the quite trivial) adjunction
$${\rm Dom}\dashv \;!_{\Box}:\LocSmallCats\to\LocSmallCats^{\EuRoman 2},$$
where the right adjoint to the domain functor ${\rm Dom}$ (which exhibits $\LocSmallCats^{\EuRoman 2}$ as fibered over $\LocSmallCats)$ assigns to a category $\CX$ the functor $!_{\CX}:\CX\to{\EuRoman 1}$ (which happens to be a bifibration), considered as an object of $\LocSmallCats^{\EuRoman 2}$. Post-composing this adjunction of 2-functors with the extended Guitart adjunction produces the Guitart adjunction, as the composite adjunction 
\begin{center}
$\xymatrix{\LocSmallCats\ar@/^0.5pc/[rr]^{!_{\Box}} & \top & \LocSmallCats^{\EuRoman 2}\ar@/^0.5pc/[rr]^{\sf Diag^{\circ}\quad}\ar@/^0.5pc/[ll]^{\rm Dom} & \top & \LocSmallCats/\SmallCats\;.\ar@/^0.5pc/[ll]^{\int_{\circ}\quad}
}$
\end{center}	
\end{rem}

\section{The Grothendieck equivalence via the extended Guitart adjunction}
In this section we show how the 2-equivalence of split cofibrations $P:\CE\to\CB$ and $\LocSmallCats$-valued functors $\Phi:\CB\to\LocSmallCats$, with functorial and natural changes of the base category $\CB$ permitted, may be obtained from the fundamental adjunction of Theorem \ref{fundamental}. Initially we will restrict ourselves to the consideration of split cofibrations with small fibres. We also formulate the dualized statement for split fibrations.

\subsection{Strictification of lax-commutative diagrams}\label{laxintostrict}
As the Grothendieck equivalence for strict cofibrations with small fibres involves the 2-category $\LocSmallCats//\SmallCats$, rather than its subcategory $\LocSmallCats/\SmallCats$, our first goal is to map the latter 2-category into the former, with a right adjoint to the inclusion functor. To explain the importance of this step we start with the observation, that the left-adjoint 2-functor $\int_{\circ}:\LocSmallCats/\SmallCats\to \LocSmallCats^{\EuRoman 2}$ of Theorem \ref{fundamental}, assigning to the functor $\Phi:\CB\to\SmallCats$ the split cofibration $\Pi_{\Phi}:{\sf Gr}_{\circ}(\Phi)\to\CB$, actually takes values in the (non-full) sub-2-category
$$\EuRoman{SCoFIB}_{\rm sf}$$
of $\LocSmallCats^{\EuRoman 2}$. Its objects are split cofibrations with small fibres, and its morphisms $(S,T):P\to Q$ are morphisms of split cofibrations $P:\CE\to\CB,\;Q:\CF\to\CC$, {\em i.e.,} $\LocSmallCats^{\EuRoman 2}$-morphisms that respect the cocleavages: 
$$T_bu_!=(Su)_!T_a\quad\text{and}\quad T\delta^u=\delta^{Su}T_a,$$ for all $u:a\to b$ in $\CB$, where $T_a:\CE_a\to\CF_{Sa}$ is a restriction of $T$; 2-cells are as in $\LocSmallCats^{\EuRoman 2}$ (see \ref{2categories}).
As a consequence, the functor $T:\CE\to\CF$ must transform the designated ($P$-cocartesion,\;$P$-vertical)-factorization of a morphism $f:x\to y$ in $\CE$ into the designated ($Q$-cocartesian,\;$Q$-vertical)-factorization of $Tf:Tx\to Ty$ in $\CF$:
\begin{center}
$\xymatrix{T[\,x\ar[r]^{\delta^{Pf}_x\quad} & (Pf)_!(x)\ar[r]^{\nu_f\quad} & y\,]\quad=\quad[\;Tx\ar[r]^{\delta^{SPf}_{Tx}} & (SPf)_!(Tx)\ar[r]^{\quad\nu_{Tf}} & Ty\;].
}$	
\end{center}
Even when we consider the extension of $\int_{\circ}$ to $\LocSmallCats//\SmallCats$ as in Proposition \ref{2functorsDiagGr}, the values still lie in $\EuRoman{SCoFIB}_{\rm sf}$, as one confirms easily. So, we have the commutative diagram
\begin{center}
$\xymatrix{ & \LocSmallCats//\SmallCats\ar[ld]_{\int_{\circ}} & \\
\EuRoman{SCoFIB}_{\rm sf} & & \LocSmallCats/\SmallCats\ar[ll]^{\int_{\circ}}\ar[ul]_{\rm Inclusion}
}$	
\end{center}
in which the bottom 2-functor has a right adjoint, $\sf{Diag}^{\circ}$, as a consequence of Theorem \ref{fundamental}. We want to show that the extended Guitart adjunction factors through $\LocSmallCats//\SmallCats$, leading to a non-trivial factorization of $\sf{Diag}^{\circ}$ as a composite of right-adjoint 2-functors. To this end, we now prove:

\begin{prop}\label{Reducing dimension}
The inclusion 2-functor $\LocSmallCats/\SmallCats\to \LocSmallCats//\SmallCats$ has a right adjoint, $\mathbb S{\rm trict}$, given by the strictification of lax-commutative diagrams.	
\end{prop}

\begin{proof}
As an ordinary functor, $\mathbb S{\rm trict}$ may be described by a slight adjustment of the 2-functor ${\rm Strict}$ established in Proposition \ref{LackLemma}, where the ordinary category $\CX$ is now taken to be the 2-category $\SmallCats$. As a result, the strictification needs to account for the greater supply of 2-cells in $\LocSmallCats//\SmallCats$ than that in ${\sf DIAG}^{\circ}(\SmallCats)$. The action of $\mathbb S{\rm trict}$ on objects and morphisms is now visualized by
\begin{center}
$\xymatrix{\CB\ar[rr]^S\ar[dr]_{\Phi}^{\;\;\tau:\Lra} & & \CC\ar[dl]^{\Psi} & \longmapsto & \SmallCats\! \Downarrow\! \Phi\ar[rr]^{\mathbb S{\rm trict}(S,\tau)}\ar[rd]_{{\rm dom}_{\Phi}} & & \SmallCats\! \Downarrow\! \Psi\ar[dl]^{{\rm dom}_{\Psi}} \\
& \SmallCats & & \mathbb S{\rm trict} & & \SmallCats & \\}$
\end{center}
with the {\em lax comma category} $\SmallCats\!\Downarrow\!\Phi$ replacing the ordinary comma category $\SmallCats\!\downarrow\! \Phi$ that (with $\CX$ instead of $\SmallCats $) was considered in Proposition \ref{LackLemma}. We write the objects of $\SmallCats\!\Downarrow\!\Phi$ in the form $(a,\CI,X)$ with $a\in\CB$ and $X:\CI\to\Phi a$ in $\SmallCats$ and let the functor $\mathbb S{\rm trict}(S,\tau)$ map a morphism  $(u,F,\varphi):(a,\CI,X)\to (b,\CJ,Y)$ as is indicated by 
\begin{center}
$\xymatrix{\CI\ar[r]^F\ar[d]_X^{\quad \varphi:\Longrightarrow}& \CJ\ar[d]^Y & \longmapsto & \CI\ar[r]^F\ar[d]_{\tau_aX}^{\quad\tau_b\varphi:\Longrightarrow} & \CJ\ar[d]^{\tau_bY}\\
\Phi a\ar[r]^{\Phi u} & \Phi b & \mathbb S{\rm trict}(S,\tau) & \Psi(Sa)\ar[r]^{\Psi(Su)} & \Psi(Sb)\\
}$
\end{center}
$\mathbb S{\rm trict}$ maps a 2-cell $(\sigma,\mu):(S,\tau)\to(S',\tau')$ to the natural transformation $\mathbb S{\rm trict}(\sigma,\mu):{\mathbb S\rm trict}(S,\tau)\to\mathbb S{\rm trict}(S',\tau')$, defined at the object $(a,\CI,X)\in \SmallCats\!\Downarrow\!\Phi$ as the morphism 
$$(\mathbb S{\rm trict}(\sigma,\mu))_{(a,\CI,X)}=(Sa,{\rm Id}_{\CI},\mu_aX):(Sa,\CI,\tau_aX)\to(S'a,\CI,\tau'_aX)$$ in $\SmallCats\!\Downarrow\!\Psi$.
For any functors $\Phi, \Psi$ as above,
the needed adjunction isomorphism
$$(\LocSmallCats//\SmallCats)(\Phi,\Psi)\cong(\LocSmallCats/\SmallCats)(\Phi,{\rm dom}_{\Psi})$$
of categories associates with $(S,\tau):\Phi\to\Psi$ the functor 
$$(S,\tau)^{\sharp}:\CB\to\SmallCats\!\Downarrow\!\Psi,\quad (u:a\to b)\mapsto ((Su,\Phi u,1_{\tau_b\Phi u}):(Sa,\Phi a,\tau_a)\to(Sb,\Phi b,\tau_b))\;;$$
a 2-cell $(\sigma,\mu):(S,\tau)\to(S',\tau')$ corresponds to the natural transformation $(\sigma,\mu)^{\sharp}$, defined at every $a\in\CB$ by
$$(\sigma,\mu)^{\sharp}_a=(\sigma_a, {\rm Id}_{\Phi a},\mu_a):(Sa,\Phi a,\tau_a)\to(S'a,\Phi a,\tau'_a).$$
Note that $(\sigma,\mu)^{\sharp}$ is indeed a 2-cell in $\LocSmallCats/\SmallCats$ since ${\rm dom}_{\Psi}(\sigma,\mu)^{\sharp}=1_{\Phi}$. We omit the details of all the lengthy, but routine verifications.
\end{proof}

\subsection{Replacing diagrams by fibres}

When $P:\CE\to\CB$ is a split cofibration with small fibres, considered as an object of ${\EuRoman{SCoFIB}}_{\rm sf}$, rather than mapping it with the 2-functor ${\sf Diag}^{\circ}$, we may now map $P$ to its (covariant) {\em fibre decomposition functor}
$${\sf Fib}_{\circ}(P)=\Phi_P: \CB\to\SmallCats, \quad u_!:\CE_a\to\CE_b=P^{-1}(b),$$
considered as an object of $\LocSmallCats//\SmallCats$. 

\begin{prop}\label{2functorFib}
The assignment $P\longmapsto {\sf Fib}_{\circ}(P)$ extends to a
	 2-functor
${\sf Fib}_{\circ}:\EuRoman{SCoFIB}_{\rm sf}\lra\LocSmallCats//\SmallCats.$
\end{prop}
\begin{proof}
Keeping the notation of \ref{laxintostrict}, we map a morphism $(S,T):P\to Q$ in ${\EuRoman{SCoFIB}}_{\rm sf}$  to the $\LocSmallCats//\SmallCats$-morphism
$${\sf Fib}_{\circ}(S,T)=(S,\tau):\Phi_P\to\Phi_Q,$$
where we define the natural transformation $\tau:\Phi_P\to\Phi_QS$ by restricting the functor $T$, via $\tau_a=T_a:\CE_a\to \CF_{Sa}$; the naturality of $\tau$ follows from $(S,T)$ being a morphism of split cofibrations. For a 2-cell $(\alpha,\beta):(S,T)\to(S',T')$ in $\EuRoman{SOFIB}_{\rm sf}$, we define the 2-cell $${\sf Fib}_{\circ}(\alpha,\beta)=(\alpha,\mu):(S,\tau)\to (S',\tau')={\sf Fib}_{\circ}(S',T')$$ in
$\LocSmallCats//\SmallCats$ by specifying the modification $\mu:\tau\to\tau'$, as follows: for every $a\in \CB$, we define the natural transformation $\mu_a: (\alpha_a)_!\tau_a\to\tau'_a$, by letting $(\mu_a)_x$ be the $Q$-vertical factor of the canonical ($Q$-cocartesian, $Q$-vertical)-factorization of $\beta_x$, for every $x\in\CE_a$:
\begin{center}
$\xymatrix{ & (\alpha_a)_!(Tx)\ar[d]^{(\mu_a)_x} &\\
Tx\ar[r]_{\beta_x}\ar[ru]^{\delta_{Tx}^{\alpha_a}} &T'x & \CF\ar[d]^Q \\
Sa\ar[r]^{\alpha_a} &S'a & \CC  
}$		
\end{center}
Naturality of every $\mu_a$ follows easily from the naturality of $\beta$ and $\delta^{\alpha_a}$; indeed, for every $f:x\to x'$ in $\CE_a$ one has 
\begin{align*}
	T'f\cdot(\mu_a)_x\cdot\delta_{Tx}^{\alpha_a} &= T'f\cdot\beta_x & (\text{definition of }(\mu_a)_x)\\
	& =\beta_{x'}\cdot Tf & (\text{naturality of }\beta)\\
	& = (\mu_a)_{x'}\cdot \delta_{Tx'}^{\alpha_a}\cdot Tf & (\text{definition of }(\mu_a)_{x'})\\
	& =(\mu_a)_{x'}\cdot (\alpha_a)_!(Tf)\cdot\delta_{Tx}^{\alpha_a} & (\text{naturality of }\delta^{\alpha_a}),\\
\end{align*}
which implies the desired equality $T_a'f\cdot(\mu_a)_x=(\mu_a)_{x'}\cdot(\alpha_a)_!(T_af)$ in $\CF_{Sa}$. 

For $\mu$ to qualify as a modification, we must verify that the natural transformations $(S'u)_!\mu_a$ and $\mu_bu_!$ coincide, for all $u:a\to b$ in $\CB$. Indeed, by the naturality of $\alpha$ and the preservation of cocartesian liftings by $T$ and $T'$, they have the common domain $(S'u)_!(\alpha_a)_!T_a=(\alpha_b)_!(Su)_!T_a=(\alpha_b)_!T_bu_!$ and the common codomain $(S'u)_!T'_a=T'_bu_!$. 
Hence,
it remains to be shown that, for all $x\in \CE_a$, we have the equality 
$(S'u)_!((\mu_a)_x)=(\mu_b)_{u_!x}$ in $\CF_{Sa}$, which follows from the following sequence of equalities that may be traced by chasing around this diagram:
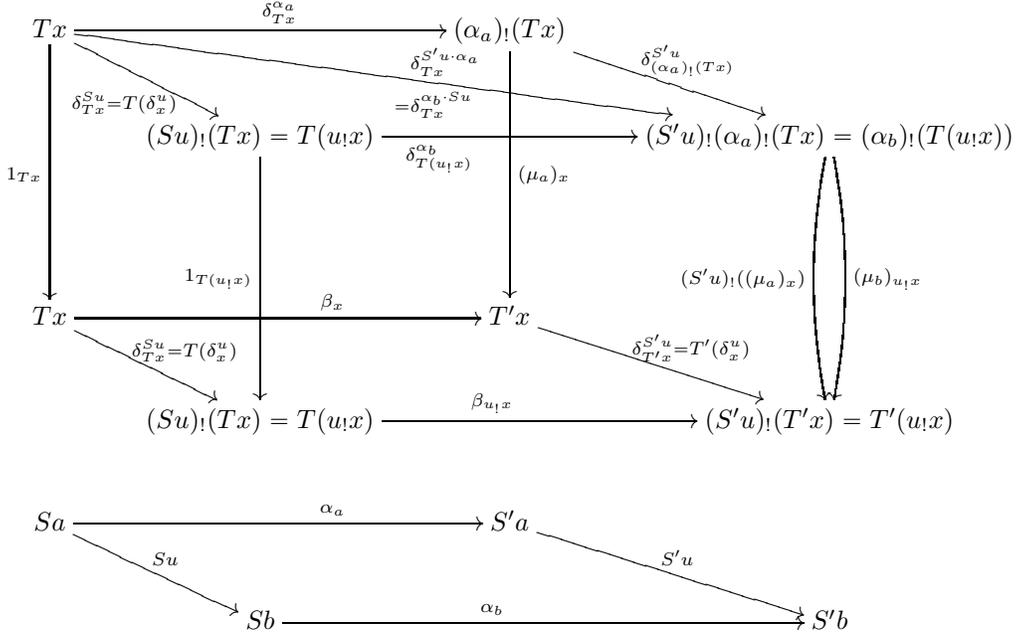
\begin{center}
$\xymatrix{Tx\ar[ddd]_{1_{Tx}}\ar[rr]^{\delta_{Tx}^{\alpha_a}}\ar[dr]_{\delta_{Tx}^{Su}=T(\delta_x^u)}\ar[drrr]^{\delta_{Tx}^{S'u\cdot\alpha_a}}_{=\delta_{Tx}^{\alpha_b\cdot Su}} & & (\alpha_a)_!(Tx)\ar[ddd]^{(\mu_a)_x}\ar[rd]^{\delta^{S'u}_{(\alpha_a)_!(Tx)}} & \\
& (Su)_!(Tx)=T(u_!x)\ar[ddd]_{1_{T(u_!x)}}\ar[rr]_{\delta_{T(u_!x)}^{\alpha_b}\qquad\qquad\qquad\qquad}& & (S'u)_!(\alpha_a)_!(Tx)=(\alpha_b)_!(T(u_!x))\ar@/_0.5pc/[ddd]_{(S'u)_!((\mu_a)_x)}\ar@/^0.5pc/[ddd]^{(\mu_b)_{u_!x}}\\
& & &  \\
Tx\ar[rr]^{\qquad\qquad\beta_x}\ar[rd]^{\delta_{Tx}^{Su}=T(\delta_x^u)} & & T'x\ar[rd]^{\delta^{S'u}_{T'x}={T'(\delta^u_x)}} & \\
& (Su)_!(Tx)=T(u_!x)\ar[rr]^{\beta_{u_!x}\qquad\qquad} & & (S'u)_!(T'x)=T'(u_!x)\\ 
Sa\ar[rr]^{\qquad\qquad\alpha_a}\ar[rd]^{Su} & & S'a\ar[rd]^{S'u} &\\
& Sb\ar[rr]^{\alpha_b\qquad\qquad} & & S'b\\
}$	
\end{center}

\begin{align*}
(S'u)_!((\mu_a)_x)\cdot\delta_{Tx}^{S'u\cdot\alpha_a} & =(S'u)_!((\mu_a)_x)\cdot \delta^{S'u}_{(\alpha_a)_!(Tx)}\cdot\delta^{\alpha_a}_{Tx} & (\text{composition of cocleavages})\\
& =\delta^{S'u}_{T'x}\cdot(\mu_a)_x\cdot\delta_{Tx}^{\alpha_a} & (\text{naturality of }\delta^{S'u})\\
& = T'(\delta_x^u)\cdot\beta_x & (T' \text{ preserves cocleavages; def. of }(\mu_a)_x)\\
& =\beta_{u_!x}\cdot T(\delta^u_x) & (\text{naturality of }\beta) \\
& =(\mu_b)_{u_!x}\cdot\delta_{T(u_!x)}^{\alpha_b}\cdot\delta_{Tx}^{Su}& (\text{def. of }(\mu_b)_{u_!x}; T\text{ preserves cocleavages})\\
& = (\mu_b)_{u_!x}\cdot \delta_{Tx}^{\alpha_b\cdot Su}& (\text{composition of cocleavages})\\
& =(\mu_b)_{u_!x}\cdot\delta_{Tx}^{S'u\cdot\alpha_a} & (\text{naturality of }\alpha).\\	
\end{align*}
The remaining lengthy verifications for the 2-functoriality of ${\sf Fib}_{\circ}$ may be left to the reader.
\end{proof}

\subsection{The Grothendieck Equivalence Theorem for split cofibrations}
We are now ready to formulate the following ``folklore'' theorem, a sufficiently elaborate proof of which does not seem to be easily available in the literature, at least not for a variable base category:

\begin{thm}\label{GrothendieckEquivalence}
The 2-functors $\int_{\circ}\dashv {\sf Fib}_{\circ}:{\EuRoman{SCoFIB}}_{\rm sf}\lra\LocSmallCats//\SmallCats$ are adjoint 2-equivalences.
\end{thm}

\begin{proof}
We first establish an invertible 2-natural transformation $\kappa:\int_{\circ}\circ\;{\sf Fib}_{\circ}\to 1_{{\rm Id}_{{\sf SOFIB}_{\rm sf}}}$. For a split cofibration $P:\CE\to\CB$ with small fibres, the $\LocSmallCats^{\EuRoman 2}$-morphism  $\kappa_P=(K_P,{\rm Id}_{\CB})$ as depicted by 
\begin{center}
$\xymatrix{\int_{\circ}\Phi_P\ar[r]^{\, K_P}\ar[d]_{\Pi_{\Phi_P}} & \CE\ar[d]^{P} \\
\CB\ar[r]^{{\rm Id}_{\CB}} & \CB \\
}$	
\end{center}
is given by the ``composition functor''
$$K_P:((u,f):(a,x)\to(b,y))\mapsto (f\cdot \delta_x^u: x\to y)$$
which, in the dual situation, is displyed in Theorem \ref{equivalent}. By design, $\kappa_P$ is a morphism of split cofibrations and, quite trivially, invertible. To confirm its 2-naturality, we consider, in the notation of Section \ref{2categories}, a 2-cell $(\alpha,\beta):(S,T)\Lra(S',T'):P\lra Q$ in ${\EuRoman{SCoFIB}}_{\rm sf}$ and the following diagram:
\begin{center}
$\xymatrix{\Pi_{\Phi_P}\ar@/^0.7pc/[rr]^{(S,\tilde{T})}\ar@/_0.7pc/[rr]_{(S',\tilde{T'})}\ar[d]_{\kappa_P=(K_P,{\rm Id}_{\CB})}  & {\scriptstyle{(\alpha,\tilde{\beta})}}\Downarrow & \Pi_{\Phi_Q}\ar[d]^{(K_Q,{\rm Id}_{\CC})=\kappa_Q}\\
P\ar@/^0.7pc/[rr]^{(S,T)}\ar@/_0.7pc/[rr]_{(S',T')}  & {\scriptstyle{(\alpha,\beta)}}\Downarrow & Q \\
}$		
\end{center}
Here, the functors $\tilde{T}, \tilde{T'}$ and the natural transformation $\tilde{\beta}:\tilde{T}\to\tilde{T'}$ are obtained by applying to the 2-cell $(\alpha,\beta)$ first ${\sf Fib}_{\circ}$ and then $\int_{\circ}$, with both 2-functors leaving the ``base'' transformation $\alpha$ unchanged. According to the definitions given in the proofs of Propositions \ref{2functorsDiagGr} and \ref{2functorFib}, one has
$$\tilde{T}:\textstyle{\int}_{\circ}\Phi_P\to\textstyle{\int}_{\circ}\Phi_Q,\quad[(u,f):(a,x)\to(b,y)]\longmapsto[(Su,Tf):(Sa,Tx)\to(Sb,Ty)],$$
$$\tilde{\beta}_{(a,x)}=(\alpha_a,(\mu_a)_x):(Sa,Tx)\to(S'a,T'x),\text{ with }(\mu_a)_x\cdot\delta_{Tx}^{\alpha_a}=\beta_x.$$
Now, from $K_Q(\tilde{\beta}_{(a,x)})=(\mu_a)_x\cdot\delta^{\alpha_a}_{Tx}=\beta_x=\beta_{K_P(a,x)}$ for all $(a,x)\in \int_{\circ}\Phi_P$ one has $K_Q\tilde{\beta}=\beta K_P$, which is the crucial ingredient to concluding the equality
$$\kappa_Q\cdot \textstyle{\int}_{\circ}({\sf Fib}_{\circ}(\alpha,\beta))=(\alpha,\beta)\cdot\kappa_P,$$
{\em i.e.,} the 2-naturality of $\kappa$.

Next, we establish an invertible 2-natural transformation $\Lambda:1_{{\rm Id}_{\LocSmallCats//\SmallCats}}\to {\sf Fib}_{\circ}\circ\int_{\circ}$ which, at the $\LocSmallCats//\SmallCats$-object $\Phi:\CB\to\SmallCats$,
 is the morphism $\Lambda_{\Phi}=({\rm Id}_{\CB},\lambda^{\Phi}):\Phi\to\Phi_{\Pi_{\Phi}}$,
  where $\lambda^{\Phi}_a:\Phi a\to (\int_{\circ}\Phi)_a$
   is the trivial bijective functor $x\mapsto(a,x)$ (see Theorem \ref{equivalent} in the dual situation). We check the 2-naturality of $\Lambda$ and, in the notation of \ref{2categories}, consider a 2-cell $(\sigma,\mu):(\Sigma,\tau)\Lra(\Sigma',\tau'):\Phi\lra\Psi$ in $\LocSmallCats//\SmallCats$. An examination of the definitions given in the proofs of Propositions \ref{2functorsDiagGr} and \ref{2functorFib}, show that, up to the identifications $\lambda^{\Phi},\lambda^{\Psi}$, the composite functor ${\sf Fib}_{\circ}\circ\int_{\circ}$ maps $(\sigma,\mu)$ to itself. As a consequence one obtains the needed equality
   $$\Lambda_{\Psi}\cdot(\sigma,\mu)={\sf Fib}_{\circ}(\textstyle{\int}_{\circ}(\sigma,\mu))\cdot\Lambda_{\Phi}.$$
   This completes the proof.
\end{proof}

The proof of Theorem \ref{GrothendieckEquivalence} remains intact if we drop the condition of small-fibredness and consider the 2-category ${\EuRoman{SCoFIB}}$ with objects all split cofibrations $P:\CE\to\CB$. 
These then correspond to $\LocSmallCats$-valued functors, rather than to $\SmallCats$-valued 
functors. Hence, one has to define the 2-category\footnote{The notation $\LocSmallCats//\LocSmallCats$
 is to be understood as analogous to the standard lax-comma category notation $\LocSmallCats//\SmallCats$. While the latter is legitimate (as $\SmallCats$ is an object of $\LocSmallCats$), 
 the former is not; rather, $\LocSmallCats//\LocSmallCats$ has to be considered as a full subcategory of ${\mathbb{CAT}}//\LocSmallCats $, with some higher-universe $\mathbb{CAT}$ that contains $\LocSmallCats$ as an object.} 
 $\LocSmallCats//\LocSmallCats$ just as $\LocSmallCats//\SmallCats$ 
 has been defined, to obtain the {\em{Grothendieck Equivalence Theorem}} \cite{Grothendieck1961} for split cofibrations:

\begin{cor}\label{BigGrothendieckEquivalence}
The 2-functors $\int_{\circ}\dashv {\sf Fib}_{\circ}:{\EuRoman{SCoFIB}}\lra\LocSmallCats//\LocSmallCats$ are adjoint 2-equivalences.	
\end{cor}

We can finally compose the 2-equivalence of Theorem \ref{GrothendieckEquivalence} with the 2-adjunction of Proposition \ref{Reducing dimension}, to obtain an alternative proof for the left-adjointness of the restricted functor $\int_{\circ}$ of \ref{laxintostrict}, without recourse to the fundamental adjunction of Theorem \ref{fundamental}. However, the advantage of having established the fundamental adjunction of Theorem \ref{fundamental} first is that we may {\em conclude} that the right adjoints ${\sf Diag}^{\circ}$ and $\mathbb S{\rm trict}\circ{\sf Fib}_{\circ}$ (with $\mathbb S{\rm trict}$ as in Proposition \ref{Reducing dimension}) coincide, up to 2-natural isomorphism---a fact that is a bit cumbersome to confirm when pursued directly. Either way, we have established the following important fact:

\begin{cor}\label{Diag factors through Fib}
The diagram 
\begin{center}
$\xymatrix{ & \LocSmallCats//\SmallCats\ar@/^0.3pc/[ld]^{\!\int_{\circ}}\ar@/^0.3pc/[dr]^{\mathbb S{\rm trict}} & \\
\EuRoman{SCoFIB}_{\rm sf}\ar@/^0.3pc/[ru]^{{\sf Fib}_{\circ}}\ar@/^0.5pc/[rr]^{{\sf Diag}^{\circ}} & \top & \LocSmallCats/\SmallCats\ar@/^0.5pc/[ll]^{\int_{\circ}}\ar@/^0.3pc/[ul]^{\rm Incl}
}$	
\end{center}
of adjunctions of 2-functors commutes. In particular, the 2-functor ${\sf Diag}^{\circ}$ factors through the Grothendieck equivalence ${\sf Fib}_{\circ}$.
\end{cor}

\subsection{The Grothendieck Equivalence Theorem for split fibrations}

It's time for us to dualize Theorem \ref{GrothendieckEquivalence} and Corollary \ref{BigGrothendieckEquivalence} and to consider the sub-2-category
$${\EuRoman{SFIB}}$$
of $\LocSmallCats^{\EuRoman 2}$, which is defined just like ${\EuRoman{SCoFIB}}$ in \ref{laxintostrict}, except that its objects 
$P:\CE\to\CB$ are now split fibrations (rather than split cofibrations), and that its morphisms $(S,T):P\to Q$ preserve cleavages, so that
$$T_au^*=(Su)^*T_b\quad\text{and}\quad T\vartheta^u=\vartheta^{Su}T_b,$$
for all $u:a\to b$ in $\CB$, where $T_a:\CE_a\to\CF_{Sa}$ is a restriction of $T$; 2-cells are as in $\LocSmallCats^{\EuRoman 2}$ (see \ref{2categories}). 

An application of the bijective 2-functor $\Box^{\rm op}:\LocSmallCats^{\rm co}\to \LocSmallCats$ to objects, morphisms and 2-cells of $\LocSmallCats^{\EuRoman 2}$  gives rise to the bijective 2-functor $\Box^{\rm op}:(\LocSmallCats^{\EuRoman 2})^{\rm co}\to \LocSmallCats^{\EuRoman 2}$ with
$$
[(\alpha,\beta):(S,T) \Rightarrow(S',T') :P\to Q ]\longmapsto [(\alpha^{\rm op},\beta^{\rm op}):(S^{\rm op},T^{\rm op})\Leftarrow ((S')^{\rm op},(T')^{\rm op}):P^{\rm op} \to Q^{\rm op} ].
$$
It maps morphisms covariantly but 2-cells contravariantly, and it restricts to a bijective 2-functor
$$\Box^{\rm op}:{\EuRoman{SFIB}}^{\rm co}\lra{\EuRoman{SCoFIB}}.$$
The bijective 2-functor $\Box^{\rm op}:\LocSmallCats^{\rm co}\to \LocSmallCats$ gives also rise to the bijective 2-functor
$$\Box^{\rm op}(-):\LocSmallCats//\LocSmallCats\lra\LocSmallCats//\LocSmallCats,$$
which post-composes every object, morphism and 2-cell with the functor $\Box^{\rm op}$:
$$[(\si,\mu)\!:\!(\Sigma,\tau)\Rightarrow(\Sigma',\tau')\!:\!\Phi\to\Psi]\longmapsto[( \Box^{\rm op}\si,\Box^{\rm op}\mu)\!:\!(\Box^{\rm op}\Sigma,\Box^{\rm op}\tau)\Rightarrow(\Box^{\rm op}\Sigma',\Box^{\rm op}\tau')\!:\!\Box^{\rm op}\Phi\to\Box^{\rm op}\Psi].$$
Now we may define the 2-functor ${\sf Fib}^{\circ}$ as the dualization of the 2-functor ${\sf Fib}_{\circ}$ of Corollary \ref{BigGrothendieckEquivalence}, that is: as the composite 2-functor given by the commutative diagram
\begin{center}
$\xymatrix{{\EuRoman{SFIB}}^{\rm co}\ar[r]^{{\sf Fib}^{\circ}}\ar[d]_{\Box^{\rm op}} & \LocSmallCats//\LocSmallCats\\
{\EuRoman{SCoFIB}}\ar[r]^{{\sf Fib}_{\circ}} & \LocSmallCats//\LocSmallCats.\ar[u]_{\Box^{\rm op}(-)}
}$	
\end{center}
Chasing a split fibration $P:\CE\to\CB$ around the lower path of the diagram shows that, as expected, ${\sf Fib}^{\circ}$ maps $P$ to its fibre representation $\Phi^{P}:\CB^{\rm op}\to\LocSmallCats$ (as in \ref{subsec:GrothendieckFibrations}), and morphisms and 2-cells of ${\EuRoman{SFIB}}$ get mapped as indicated by
\begin{center}
$\xymatrix{\CE\ar[d]_{P}\ar@/^0.7pc/[rr]^{T}\ar@/_0.7pc/[rr]_{T'} & {\scriptstyle{\beta}}\Downarrow\quad &\CF\ar[d]^{Q}\\
 	\CB\ar@/^0.7pc/[rr]^{S}\ar@/_0.7pc/[rr]_{S'} & {\scriptstyle{\alpha}}\Downarrow\quad&\CE
 	}\hfil\longmapsto\hfil \xymatrix{\CB^{\rm op}\ar@/^0.7pc/[rr]^{(S^{\rm op},\tau^{\rm op})}\ar@/_0.7pc/[rr]_{((S')^{\rm op},(\tau')^{\rm op})}\ar[dr]_{{\Phi}^P} & {\scriptstyle (\al^{\rm op},\mu^{\rm op})}\!\Uparrow & \CC^{\rm op}\ar[dl]^{{\Phi}^Q}\\
& \SmallCats}$
 	\end{center}
Here the natural transformations $\tau:\Phi^P\to\Phi^QS^{\op},\;\tau':\Phi^P\to\Phi^Q(S')^{\op }$, analogously to the respective definitions for ${\sf Fib}_{\circ}$ in the proof of Proposition \ref{2functorFib}, are determined as the restrictions $\tau_a=T_a:\CE_a\to\CF_{Sa},\;\tau'_a=T'_a:\CE_a\to\CF_{S'a}$ of $T$ for all $a\in\CB$, while the transformations $\mu_a:\tau_a\to (\alpha_a)^*\tau'_a$ comprising the modification $\mu:\tau\to\tau'$ 	are defined as the $Q$-vertical factors in the factorization $\beta_x=\theta^{\al_a}_{T'x}\cdot(\mu_a)_x$, for all $x\in\CE_a$.
 	
Now, since ${\sf Fib}_{\circ}$ is an equivalence of 2-categories, its dualization, $\sf Fib^{\circ}$ is one as well. Moreover, one obtains its quasi-inverse, $\int^{\circ}$, from the 2-functor $\int_{\circ}$, by the same dualization procedure that has produced ${\sf Fib}^{\circ}$ from ${\sf Fib}_{\circ}$. Indeed, the dualization diagram
 	
 \begin{center}
$\xymatrix{{\EuRoman{SFIB}}^{\rm co}\ar@/^0.2pc/[r]^{{\sf Fib}^{\circ}}\ar[d]_{\Box^{\rm op}} & \LocSmallCats//\LocSmallCats\ar@/^0.2pc/[l]^{{\int}^{\circ}}\\
{\EuRoman{SCoFIB}}\ar@/^0.2pc/[r]^{{\sf Fib}_{\circ}} & \LocSmallCats//\LocSmallCats.\ar@/^0.2pc/[l]^{{\int}_{\circ}}\ar[u]_{\Box^{\rm op}(-)}
}$	
\end{center}	
commutes at both, the ${\sf Fib}$- and the ${\int}$-level.	 Hence, with Theorem \ref{GrothendieckEquivalence} and Corollary \ref{BigGrothendieckEquivalence} we conclude:
\begin{cor}\label{GrothendieckEquivalenceDual}
The 2-functors ${\int}^{\circ}\dashv{\sf Fib}^{\circ}: {\EuRoman{SFIB}}^{\sf co}\lra\LocSmallCats//\LocSmallCats$ are adjoint 2-equivalences. By restriction to the small-fibred split fibrations they give the 2-equivalences ${\int}^{\circ}\dashv{\sf Fib}^{\circ}: ({\EuRoman{SFIB}}_{\rm sf})^{\rm co}\lra\LocSmallCats//\SmallCats$.	
\end{cor}

\section{A left adjoint to the Grothendieck construction: free split cofibrations}

In this section we give a novel proof of the essentially known fact\footnote{See ncatlab.org, ``Grothendieck construction'', for a proof of a corresponding statement with fixed base category.} that the composite 2-functor
\begin{center}
$\xymatrix{\LocSmallCats//\LocSmallCats\ar[r]^{{\int}_{\circ}} & {\EuRoman{SCoFIB}}\ar[r]^{\rm Incl} & \LocSmallCats^{\EuRoman 2}}$	
\end{center}
has a left adjoint. Indeed, since ${\int}_{\circ}$ is a 2-equivalence (by Corollary \ref{BigGrothendieckEquivalence}), it suffices to show that the inclusion 2-functor has a left adjoint, {\em i.e.}, that $\EuRoman{SCoFIB}$ is 2-reflective in $\LocSmallCats^{\EuRoman 2}$. This means that we must show {\em how an arbitrary functor $P$ can be freely ``made'' into a split cofibration}, compatibly so with the relevant 2-categorical structures. 
\subsection{Free split cofibrations}\label{Free Split Cofib}

We define the 2-functor
$${\sf  Free}:\LocSmallCats^{\EuRoman 2}\to{\EuRoman{SCoFIB}},$$
as follows. For every functor $P:\CE\to\CB$, the dual Grothendieck construction applied to the trivial slice functor $P/\Box :\CB\to\LocSmallCats$ gives us (in generalization of Example \ref{Standard}(3)) the split cofibration 
$${\sf Free}(P):=\Pi_{P/\Box}={\rm cod}_P:\textstyle{\int}_{\circ}P/\Box=P\downarrow\CB\longrightarrow\CB $$
of the comma category $P\downarrow\CB$. Here we therefore write an object in $P\downarrow\CB$ as a pair $(h,x)$ with $x\in\CE$ and $h:Px\to a$ in $\CB$; a morphism $(u,f):(h,x)\to(k,y)$ is given by the commutative square on the left of the diagram 
\begin{center}
$\xymatrix{Px\ar[r]^{Pf}\ar[d]_h & Py\ar[d]^k & = & Px\ar[r]^{P1_x}\ar[d]_h & Px\ar[r]^{Pf}\ar[d]_{u\cdot h=}^{k\cdot Pf} &Py\ar[d]^{k} \\
a\ar[r]^u &b & & a\ar[r]^u & b\ar[r]^{1_b} & b \;.\\
}$	
\end{center}
The right part of the diagram describes the designated $({\rm cod}_P\text{-cocartesian},{\rm cod}_P\text{-vertical})$-factorization of $(u,f)$, so that one has:

\medskip
$u_!(h,x)=(u\cdot h,x), \quad\delta^u_{(h,x)}=(u,1_x):(h,x)\to (u\cdot h,x),\quad \nu_{(u,f)}=(1_b,f):(u\cdot h
,x)\to (k,y).$
\medskip

The definition of ${\sf  Free}$ on morphisms and 2-cells is also straightforward. In the notation of \ref{2categories}
, the action of ${\sf Free}$ is described by
$$[(\alpha,\beta):(S,T)\Lra(S',T'):P\lra Q]\quad\longmapsto\quad[(\alpha,\bar{\beta}):(S,\bar{T})\Lra(S',\bar{T'}):{\rm cod}_P\lra {\rm cod}_Q],$$
where 
$$\bar{T}:P\downarrow\CB\lra Q\downarrow\CC,\quad ((u,f):(h,x)\to(k,y))\longmapsto((Su,Tf):(Sh,Tx)\to(Sk,Ty)),$$
$$\bar{\beta}_{(h,x)}=(\alpha_a,\beta_x):\bar{T}(h,x)=(Sh,Tx)\lra\bar{T'}(h,x)=(S'h,T'x),$$
for all objects $(h\!:\!Px\to a,\, x)$ in $P\downarrow\CB$. Trivially, $\bar{T}$ transforms ${\rm cod}_P$-cocleavages into ${\rm cod_Q}$-cocleavages.
Generalizing a corresponding result by \cite{Gray1966} (in a fixedp-based and one-dimensional context), we now prove:

\begin{thm}\label{comma}
The 2-functor ${\sf Free}$ is left adjoint to the inclusion $\EuRoman{SCoFIB}\lra\LocSmallCats^{\EuRoman 2}$.
\end{thm}

\begin{proof}
Given a $\LocSmallCats^{\EuRoman 2}$-object $P:\CE\to\CB$, we consider the functor 
$$H_P:\CE\to P\downarrow\CB,\quad(f:x\to y)\longmapsto(Pf,f):(1_{Px},x)\to(1_{Py},y)$$	
and claim that $({\rm Id}_{\CB},H_P):P\to{\rm cod}_P$ serves as the unit at $P$ of the 2-adjunction ${\sf  Free}\dashv {\rm Incl}$. Hence, we show that, for every split cofibration $Q:\CF\to\CC$, the precomposition with $({\rm Id}_{\CB},H_P)$ provides, naturally in $P$ and $Q$, a bijective functor
$$(-)\cdot({\rm Id}_{\CB},H_P):{\EuRoman{SCoFIB}}({\rm cod}_P,Q)\lra \LocSmallCats^{\EuRoman 2}(P,Q).$$
To establish its bijectivity on objects, we consider any $\LocSmallCats^{\EuRoman 2}$-morphism $(S,T):P\to Q$ and show that there is only one cocleavage-preserving functor $\tilde{T}:P\downarrow\CB\to\CF$ with $Q\,\tilde{T}=S\,{\rm cod}_P$ and $\tilde{T}\,H_P=T$. First, we observe that, for every morphism $(u,f):(h,x)\to(k,y)$ in $P\downarrow\CB$ (as in the diagram above), one has the following:
\begin{itemize}
\item[1.] $(u,f)=(1_b,f)\cdot(u,1_x)=\nu_{(u,f)}\cdot\delta_{(h,x)}^u$.
\item[2.] With the functor $h_!:(P\downarrow\CB)_{Px}=P/Px\lra(P\downarrow\CB)_a=P/a$,  the object $(h,x)\in P/a$ may be written as $h_!(1_{Px},x)$, where $(1_{Px},x)\in P/Px$.
\item[3.] Likewise, with the functor $k_!:P/Py\lra P/b$, the object $(k,y)\in P/b$ may be written as $k_!(1_{Py},y)$, where $(1_{Py},y)\in P/Py$, and the morphism $\nu_{(u,f)}$ in $P/b$ may be written as $k_!(1_{Py},y)$, with the morphism $(1_{Py},f):(Pf,x)\to(1_{Py},y)$ in $P/Py$.
\item[4.] The morphism $(1_{Py},f)$ in $P/Py$ as in 3. may be written as $(1_{Py},f)=\nu_{H_Pf}$.
\end{itemize}
Consequently, for any functor $\tilde{T}:P\downarrow\CB\to\CF$ satisfying the above properties, one necessarily has
\begin{align*}
\tilde{T}(u,f) & =\tilde{T}(k_!(\nu_{H_Pf}))\cdot	\tilde{T}(\delta^u_{h_!(1_{Px},x)})\\
& = (Sk)_!(\tilde{T}(\nu_{H_Pf}))\cdot\delta_{\tilde{T}(h_!(H_Px))}^{Su}\\
& = (Sk)_!(\nu_{\tilde{T}(H_Pf)})\cdot\delta_{{(Sh)}_!(\tilde{T}(H_Px))}^{Su}\\
& = (Sk)_!(\nu_{Tf})\cdot\delta^{Su}_{(Sh)_!(Tx)}\,,\\
\end{align*}
as in
\begin{center}
	$\xymatrix{ (Sh)_!(Tx)\ar[r]^{\delta^{Su}_{(Sh)_!(Tx)}\qquad\qquad\qquad} & (Su)_!(Sh)_!(Tx)=(Sk)_!(Q(Tf))_!(Tx)\ar[r]^{\qquad\qquad\qquad(Sk)_!(\nu_{Tf})} & (Sk)_!(Ty) & \CF\ar[d]^Q\\
	Sa\ar[r]^{Su} & Sb\ar[r]^{1_{Sb}} & Sb &\CC}$
\end{center}
Therefore, $\tilde{T}$ is unique. Conversely, setting
$$
\tilde{T}(u,f)= (Sk)_!(\nu_{Tf})\cdot\delta^{Su}_{(Sh)_!(Tx)},
$$
one has  to verify the needed properties for $\tilde{T}$. Showing that $\tilde{T}$ preserves the composition requires a careful application of the formulae for the $(Q\text{-cocartesian},\,Q\text{-vertical})$-factorization of composite arrows (see Section \ref{cofib and bifib}). Indeed, using the definition of $\tilde{T}$ for $(u,f):(h,x)\to(k,y),\;(v,g):(k,y)\to(\ell,z)$ in $P\downarrow\CB$ and the naturality of the transformation $\delta^{Sv}$, we obtain
\begin{align*}
\tilde{T}(v\cdot u,g\cdot f) & = (S\ell)_!(\nu_{Tg\cdot Tf})\cdot\delta^{Sv\cdot Su}	_{(Sh)_!(Tx)}\\
& = (S\ell)_![\nu_{Tg}\cdot (S(Pg))_!(\nu_{Tf})]\cdot \delta^{Sv}_{(Su)_!(Sh)_!(Tx)}\cdot \delta^{Su}_{(Sh)_!(Tx)}\\
& = (S\ell)_!(\nu_{Tg})\cdot (Sv)_!(Sk)_!(\nu_{Tf})\cdot\delta^{Sv}_{(Sk)_!(S(Pf))_!(Tx)}\cdot\delta^{Su}_{(Sh)_!(Tx) }   \\
& = (S\ell)_!(\nu_{Tg})\cdot\delta^{Sv}_{(Sk)_!(Ty)}\cdot(Sk)_!(\nu_{Tf})\cdot\delta^{Su}_{(Sh)_!(Tx) }   \\
& = \tilde{T}(v,g)\cdot\tilde{T}(u,f).\\
\end{align*}
The verification of the other needed properties of $\tilde{T}$ is straightforward.

It remains to be shown that $(-)\cdot({\rm Id}_{\CB},H_P)$ is fully faithful. Given a 2-cell
$$(\alpha,\beta):(S,T)\Lra(S',T'):P\lra Q$$ in $\LocSmallCats^{\EuRoman 2}$ as above, we should find a natural transformation $\tilde{\beta}:\tilde{T}\to\tilde{T'}$, unique with $Q\tilde{\beta}=\alpha\,{\rm cod}^P$ and $\tilde{\beta}H_P=\beta$.
 Since any object $(h,x)$ in $P\downarrow \CB$ gives rise to the morphism $$\delta^h_x=(h,1_x):H_Px=(1_{Px},x)\lra (h,x)$$ in $P\downarrow \CB$, the naturality of any such $\tilde{\beta}$ and the preservation of cocleavages by $T$ and $T'$ force
 $$
 \tilde{\beta}_{(h,x)}\cdot\delta^{Sh}_{Tx}  = \tilde{\beta}(h,x)\cdot\tilde{T}(\delta^h_x)
 = \tilde{T'}(\delta^h_x)\cdot \tilde{\beta}_{H_Px}
 = \delta^{Sh}_{T'x}\cdot\beta_x\, .
 $$
Since, with the naturality of $\alpha$, one has $$Q(\tilde{\beta}_{(h,x)}\cdot\delta^{Sh}_{Tx})= \alpha_x\cdot Sh=S'h\cdot\alpha_{Px}=Q(\delta^{S'h}_{T'x}\cdot\beta_x)\, ,$$
we see that, by the $Q$-cocartesianess of $\delta^{Sh}_{Tx}$, the morphism 
$\tilde{\beta}_{(h,x)}$ is necessarily the only $\CF$-morphism with $\tilde{\beta}_{(h,x)}\cdot\delta^{Sh}_{Tx}=\delta^{S'h}_{T'x}\cdot\beta_x$ 
and $Q(\tilde{\beta}_{(h,x)})=\alpha_{Px}$. Conversely, taking this as the definition of $\tilde{\beta}_{(h,x)}$, one routinely shows that $\tilde{\beta}$ has the required properties.
\end{proof}
We may now compose the 2-adjunction of Theorem \ref{comma} with the Grothendieck equivalence of Corollary \ref{BigGrothendieckEquivalence}, as in
\begin{center}
$\xymatrix{ & \EuRoman{SCoFIB}\ar@/^0.3pc/[ld]^{\rm Incl}_{\bot}\ar@/^0.5pc/[dr]^{{\sf Fib}_{\circ}} &\\
\LocSmallCats^{\EuRoman 2}\ar@/^0.5pc/[ru]^{\sf Free}\ar@/^0.5pc/[rr]^{} & \bot & \LocSmallCats//\LocSmallCats\ar@/^0.5pc/[ll]^{{\int}_{\circ}\quad}\ar@/^0.3pc/[ul]^{{\int}_{\circ}}_{\simeq}\,.
}$	
\end{center}
Since the 2-functor ${\sf Fib}_{\circ}\circ {\sf Free}$ assigns to the $\LocSmallCats^{\EuRoman 2}$-object $P:\CE\to\CB$ the fibre representation functor of the functor ${\rm cod}^P:P\downarrow\CB\to\CB$, the fibres of which are the slice categories $P/b\;(b\in\CB)$, we conclude:
\begin{cor}\label{localcomma}
The 2-functor ${\int}_{\circ}:\LocSmallCats//\LocSmallCats\lra
\LocSmallCats^{\EuRoman 2}$ has a left adjoint which maps the $\LocSmallCats$-object $P:\CE\to\CB$ to the $\LocSmallCats//\LocSmallCats$-object $P/\Box:\CB\to\LocSmallCats,\,b\mapsto P/b$ (considered in \ref{Free Split Cofib}).
\end{cor}
    
\subsection{A network of global 2-adjunctions}

With the help of the following diagram we summarize the 2-adjunctions established in this and the previous sections:
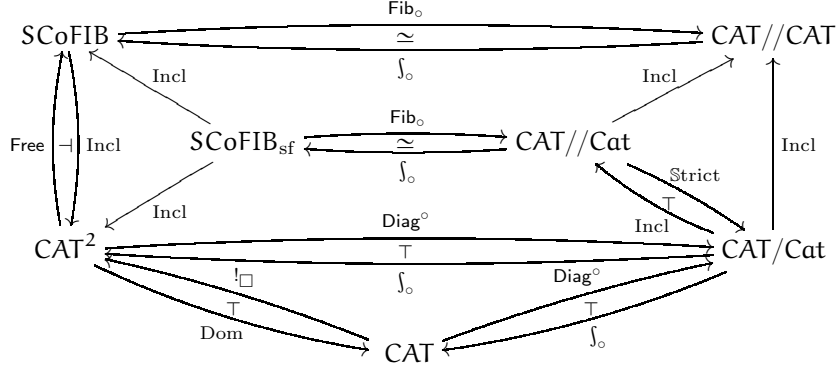
\begin{center}
$\xymatrix{
{\EuRoman{SCoFIB}}\ar@/^0.4pc/[rrrr]^{{\sf Fib}_{\circ}\quad}\ar@/^0.4pc/[dd]^{\rm Incl} & & \simeq & & \LocSmallCats//\LocSmallCats\ar@/^0.4pc/[llll]^{{\int}_{\circ}\quad } \\
\scriptstyle{\dashv} & {\EuRoman{SCoFIB}}_{\rm sf}\ar[ld]^{\rm Incl}\ar[lu]_{\rm Incl}\ar@/^0.3pc/[rr]^{{\sf Fib}_{\circ}} & \simeq & \LocSmallCats//\SmallCats\ar@/^0.3pc/[rd]^{\mathbb S{\rm trict}}_{\top}\ar[ru]^{\rm Incl}\ar@/^0.3pc/[ll]^{{\int}_{\circ}} & \\
\LocSmallCats^{\EuRoman 2}\ar@/_0.5pc/[drr]_{\rm Dom}\ar@/^0.4pc/[rrrr]^{{\sf Diag}^{\circ}\quad}\ar@/^0.4pc/[uu]^{\sf Free} & & \scriptstyle{\top} & & \LocSmallCats/\SmallCats\ar@/^0.5pc/[lu]^{\rm Incl}\ar@/^0.4pc/[llll]^{{\int}_{\circ}\quad}\ar[uu]_{\rm Incl}\ar@/^0.5pc/[dll]^{{\int}_{\circ}} \\
& & \;\LocSmallCats\ar@/_0.3pc/[llu]_{!_{\Box}}^{\top}\ar@/^0.3pc/[urr]^{{\sf Diag}^{\circ}}_{\top} & &\\
}$
	\end{center}
	\begin{itemize}
	\item The top horizontal adjunction displays the Grothendieck 2-equivalence between split cofibrations and $\LocSmallCats$-valued functors (Corollary \ref{BigGrothendieckEquivalence}). It restricts to a 2-equivalence between split cofibrations with small fibres and $\SmallCats$-valued functors (Theorem \ref{GrothendieckEquivalence}), as shown by the middle horizontal adjunction. The 2-equivalence ${\sf Fib}_{\circ}$ decomposes a split cofibration into the ``family'' of its fibres, indexed by its base category, while the Grothendieck construction ${\int}_{\circ}$ reassembles such gadgets.
	\item The ``vertical'' 2-functor ${\sf Free}$ modifies a given functor by ``freely adding cocartesian liftings'' to it, showing that the totality of split cofibrations is 2-reflective amongst all functors (Theorem \ref{comma}). The composition of this 2-adjunction with the top horizontal adjunction is described in Corollary \ref{localcomma}.\item The bottom horizontal 2-adjunction relates arbitrary functors (rather than split cofibrations) to $\SmallCats$-valued functors (Theorem \ref{fundamental}). Its left adjoint, ${\int}_{\circ}$, trivially factors through the namesakes above it. Not being able to functorially relate the fibres of an arbitrary functor with each other, the right adjoint, ${\sf Diag}^{\circ}$, relates the totality of all small diagrams over the fibres with each other, rather than the fibres themselves. Regarding categories $\CX$ as functors $!_{\CX}:\CX\to{\EuRoman 1}$, the lower horizontal 2-adjunction reduces to the lower right diagonal 2-adjunction, as first considered in its ordinary form by Guitart (Theorem \ref{Guitart}).
	\item The fibre-representation 2-functor, ${\sf Fib}_{\circ}$, of the middle horizontal equivalence maps morphisms of its domain to lax-commutative diagrams over $\SmallCats$, while the right adjoint of the lower horizontal adjunction, ${\sf Diag}^{\circ}$, maps morphisms to strictly commutative diagrams over $\SmallCats$. In fact, the restriction of the latter 2-functor factors through the former (up to isomorphism), by the strictification 2-functor, ${ \mathbb S}{\rm trict}$, which is right adjoint to a (non-full) inclusion functor (Proposition \ref{Reducing dimension}, Corollary \ref{Diag factors through Fib})	
	\end{itemize}

\section{Diagram categories as 2-(co)fibred categories over $\SmallCats$}
\subsection{Hermida-Buckley 2-fibrations}

In this supplementary section we pay tribute to the fact that $D^{\CX}:{\sf Diag}^{\circ}(\CX)\to{\SmallCats    }$ and $D_{\CX}:{\sf Diag}_{\circ}(\CX)\to {\SmallCats    }^{\op}$ are 2-functors (see Remarks \ref{Ignoring2cells}) and investigate under which conditions on $\CX$ (if any), $D^{\CX}$ or $D_{\CX}$ may be a (co)fibration as such. For that we employ Buckley's \cite{Buckley2014} improved version of Hermida's \cite{Hermida1999} notion of 2-fibration. We recall the relevant definitions:

\begin{defn}
Let $P:\CE\to\CB$ be a 2-functor.
\begin{itemize}
\item[(1)] A 1-cell $f:x\to y$ in the 2-category $\CE$ is {\em $P$-2-cartesian} if, for all objects $C$ in $\CE$, the diagram
\begin{center}
$\xymatrix{\CE(z,x)\ar[r]^{\CE(z,f)}\ar[d]_{P_{z,x}} & \CE(z,y)\ar[d]^{P_{z,y}}\\
            \CB(Pz,Px)\ar[r]_{\CB(\!P\!z\!,P\!f)} & \CB(Pz,Py)}$
            \end{center}
is a pullback in $\Cat$.
\item[(2)] A 2-cell $\al:f\to f':x\to y$ in $\CE$ is $P$-2-cartesian if it is {\em $P_{x,y}$-cartesian}, with respect to the ordinary functor $P_{x,y}:\CE(x,y)\to \CB(Px,Py)$.
\item[(3)] $P$ is a {\em (cloven) 2-fibration}  if
\begin{itemize}
\item[(a)] for all 1-cells $u:a\to b$ in $\CB$ and $y$ objects in $\CE_b$, there is a (chosen) $P$-2-cartesian lifting $f:x\to y$ in $\CE$, so that $Px=a$ and $Pf=u$;
\item[(b)] for all objects $x,y\in\CE$, the ordinary functor $P_{x,y}:\CE(x,y)\to\CB(Px,Py)$ is a (cloven) fibration;
\item[(c)] $P$-2-cartesianness of 2-cells in $\CE$ is preserved by horizontal composition.
\end{itemize}
\item[(4)] $P$ is a {\em 2-cofibration} if $P^{\rm coop}:\CE^{\rm coop}\to \CB^{\rm coop}$ is a 2-fibration.
\end{itemize}
\end{defn}

\begin{rem}\label{2fib explicit}
(1) By definition, the 1-cell $f:x\to y$ in $\CE$ is $P$-2-cartesian if, and only if, for all objects $z\in\CE$, the functor
$$\CE(z,x)\to\CB(Pz,Px)\times_{\CB(Pz,Py)}\CE(z,y),\quad (\tau:t\to t')\mapsto(P\tau:Pt\to Pt', f\tau:ft\to ft'),$$
is an isomorphism of categories. Its bijectivity on objects is equivalent to $f$ being $P$-cartesian in the ordinary sense, while its full faithfulness adds the following condition to the 1-categorical notion:
for all 2-cells $\zeta:w\Ra w':Pz\to Px$ in $\CB$ and $\rho:h\Ra h':z\to y$ in $\CE$ with $(Pf)\zeta=P\rho$, there is a unique 2-cell $\tau:t\Ra t':z\to x$ in $\CE$ with $P\tau=\zeta$ and $f\tau=\rho$.

(2) By definition, $P$-2-cartesianess of a 2-cell $\al:f\Ra f':x\to y$ in $\CE$ means that, for all 1-cells $k:x\to y$, the map
$$\CE(x,y)(k,f)\longrightarrow\CB(Px,Py)(Pk,Pf)\times_{\CB(Px,Py)(Pk,Pf')}\CE(x,y)(k,f'),\quad\mu\mapsto (P\mu, \si\cdot\mu),$$
is bijective, that is: for all 2-cells $\ga:Pk\Ra Pf$ and $\lambda:k\Ra f'$ in, respectively, $\CB$ and $\CE$, with $P\al\cdot \ga=P\lam$, one has $P\mu=\lam$ and $\al\cdot\mu=\lam$, for a unique 2-cell $\mu:k\to f$ in $\CE$.

(3) $P$ is a 2-fibration if, and only if, 
\begin{itemize}
\item[(a)] for every 1-cell  $u:a\to Py$ in $\CB$ with $y$ in $\CE$, there is a $P$-2-cartesian lifting $f:x\to y$ in $\CE$ with $Pf=u$; 
\item[(b)] for every 2-cell $\xi:u\Ra Pf':Px\to Py$ in $\CB$ with a ($P$-2-cartesian) 1-cell $f':x\to y$ in $\CE$, there is a $P$-2-cartesian lifting $\al:f\Ra f':x\to y$ with $P\al=\xi$; 
\item[(c)] for all 1-cells $t:z\to x,\,s:y\to w$ and 2-cells $\al:f\Ra f':x\to y$ in $\CE$, if $\al$ is 2-cartesian, so are $\al t:ft\to f't$ and $s\al:sf\to sf'$. (Of course, since $P$-cartesianess of 2-cells is closed under vertical composition, closure under (horizontal) pre- and post-composition with 1-cells suffices to make the property closed also under horizontal composition.)
\end{itemize}
\end{rem}

\begin{rem}
The definition of (in a quite obvious sense)  split 2-fibration as given above is motivated by the fact that a 2-fibration is, via a 2-categorical Grothendieck construction, 3-equivalently represented by a 2-functor $\CB^{\rm coop}\to \EuRoman{2Cat}$; see \cite{Buckley2014}. In fact, Buckley \cite{Buckley2014} proved a more general result at the bicategorical (rather than the 2-categorical) level.
\end{rem}

\subsection{How to consider diagram categories as 2-fibred or 2-cofibred over $\SmallCats$}

\begin{thm}\label{2fibration} 
{\em (1)} The 2-functor $D_{\CX}^{\op}:({\sf Diag}_{\circ}(\CX))^{\op}\to\SmallCats  $ is a 2-fibration, for every category $\CX$.

{\em (2)} If the category $\CX$ is cocomplete, then $D^{\CX}:{\sf Diag}^{\circ}(\CX)\to \SmallCats   $ is a 2-cofibration.
\end{thm}

\begin{proof}
(1) Recall that a morphism $(F,\varphi):(\CI,X)\to (\CJ,Y)$ in $({\sf Diag}_{\circ}(\CX))^{\op}$ is given by small categories $\CI,\CJ$, functors $F,X,Y$, and a natural transformation $\varphi$, as in the triangle below on the left, and a 2-cell $\al:(F,\varphi)\Lra(F',\varphi')$ is given by a natural transformation $\al:F\to F'$ with $\varphi=\varphi'\cdot Y\al$\,:

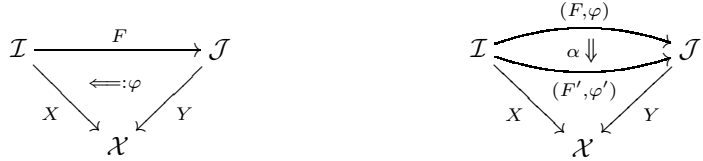
\begin{center}
$\xymatrix{\CI\ar[rr]^F\ar[dr]_X^{\;\;\Lla:\varphi} & & \CJ\ar[dl]^{Y}\\
& \CX}$ \hfil 
$\xymatrix{\CI\ar@/^0.7pc/[rr]^{(F,\varphi)}\ar@/_0.7pc/[rr]_{(F',\varphi')}\ar[dr]_X & {\scriptstyle \al}\!\Downarrow & \CJ\ar[dl]^{Y}\\
& \CX}$
\end{center}
Now, given $F:\CI\to\CJ$ in $\SmallCats$ and  $(\CJ,Y)\in ({\sf Diag}_{\circ}(\CX))^{\op}$, we have the trivial $D_{\CX}^{\op}$-cartesian lifting $(F,1_{YF}):(\CI,YF)\to(\CJ,Y)$ at the 1-category level (Proposition \ref{DiagasGr}). To show that $(F,1_{FY})$ is $D_{\CX}^{\op}$-2-cartesian, it suffices to consider a natural transformation $\zeta: G\Lra G':\CK\to\CI$ and a 2-cell $\rho:(H,\gamma)\Lra(H',\gamma'):(\CK,Z)\to(\CJ,Y)$
with $FG=H,\, FG'=H',\, F\zeta = \rho$, and show that $\zeta:(G,\gamma)\Lra(G',\gamma'):(\CK,Z)\to(\CI,YF)$ is actually a 2-cell in $({\sf Diag}_{\circ}(\CX))^{\op}$.
But this is trivial: the given identity $\gamma'\cdot Y\rho=\gamma$ may just be restated as the needed identity $\gamma'\cdot(YF)\zeta=\gamma$.

Next, in oder to verify property (b) of Remark \ref{2fib explicit}, we consider a 1-cell $(F',\varphi'):(\CI,X)\to(\CJ,Y)$ in $({\sf Diag}_{\circ}(\CX))^{\op}$ 
and a 2-cell $\al:F\Lra F':\CI\to\CJ$ in $\SmallCats    $ and show that the emerging 2-cell $\al:(F,\varphi:=\varphi'\cdot Y\al)\Lra(F',\varphi')$ is $D_{\CX}^{\op}$-2-cartesian. Indeed, given 2-cells $\lam:(K,\kappa)\to(F',\varphi')$ and $\ga:K\to F$ in, respectively, $({\sf Diag}_{\circ}(\CX))^{\op}$ and $\SmallCats    $ with $\al\cdot\ga=\lam$, the given identity $\varphi'\cdot Y\lam=\kappa$ translates to $\varphi\cdot Y\ga=\kappa$, thus making $\ga:(K,\kappa)\to (F,\varphi)$ a 2-cell in ${\sf Diag}^{\star}(\CX))$, as desired.

Finally, to verify property (c), for 1-cells $(T,\eta):(\CK,Z)\to(\CI,X),\;(S,\varepsilon):(\CJ,Y)\to(\CL,W)$ and the $D_{\CX}^{\op}$-2-cartesian 2-cell $\al:(F,\varphi)\Lra(F',\varphi'):(\CI,X)\to(\CJ,Y)$ as above, we must show that the horizontal composites
\begin{center}
$\xymatrix{(FT,\eta\cdot\varphi T)\ar@{=>}[r]^{\al T} & (F'T, \eta\cdot\varphi' T), & (SF,\varphi\cdot\varepsilon F)\ar@{=>}[r]^{S\al} & (SF',\varphi'\cdot\varepsilon F')}$
\end{center}
are $D_{\CX}^{\op}$-2-cartesian as well. Indeed, from $\varphi'\cdot Y\al=\varphi$ one obtains immediately 
$$\eta\cdot\varphi' T\cdot Y(\al T) = \eta\cdot\varphi T,\quad \varphi'\cdot\varepsilon F'\cdot W(S\al)=\varphi'\cdot Y\al\cdot\varepsilon F=\varphi\cdot\varepsilon F,$$
as desired.

(2)  We now consider $(\CI,X)\in{\sf Diag}^{\circ}(\CX)$ and $F:\CI\to\CJ$ in $\SmallCats   $ and form the $D^{\CX}$-cocartesian lifting $(F,\varphi):(\CI,X)\to(\CJ,Y)$ at the 1-categorical level, so that $\varphi:X\to YF$ presents $Y$ as a left Kan extension of $X$ along $F$ (Proposition \ref{Diagbifibred}). 
To show that $(F,\varphi)$ is $D^{\CX}$-2-cocartesian, given any 2-cells $\tau: G\Lra G':\CJ\to\CK$ and $\rho:(H,\gamma)\Lra(H',\gamma'):(\CI,X)\to(\CK,Z)$
with $GF=H, G'F=H', \tau F= \rho$, we let $\be:Y\to ZH,\,\be':Y\to ZH'$ be determined by the identities $\be F\cdot\varphi=\gamma,\,\be'F\cdot \varphi=\gamma'$ and must then confirm that $\tau:(G,\be)\Lra(G',\be'):(\CJ,Y)\to(\CK,Z)$ is a 2-cell in ${\sf Diag}^{\circ}(\CX)$.
But this is straightforward, since from 
$$(Z\tau\cdot\beta)F\cdot \varphi=Z\tau F\cdot\beta F\cdot\varphi=Z\tau F\cdot\gamma=Z\rho\cdot\gamma=\gamma'=\beta' F\cdot\varphi$$
one deduces the desired identity $Z\tau\cdot\beta=\beta'$.

Finding a $D^{\CX}$-2-cocartesian lifting for a 2-cell $\al:F\Lra F'$ in $\SmallCats   $ that comes with a 1-cell $(F,\varphi):(\CI,X)\to(\CJ,Y)$ in ${\sf Diag}^{\circ}(\CX)$ proceeds as in (1): one just puts $\varphi':=Y\al\cdot\varphi$ and easily shows that, given 2-cells $\lam:(F,\varphi)\Lra(K,\kappa)$ in ${\sf Diag}^{\circ}(\CX)$ and $\chi:F'\to K$ in $\SmallCats  $ with $\chi\cdot\al=\lam$, then $\chi:(F',\varphi')\to(K,\kappa)$ actually lives in ${\sf Diag}^{\circ}(\CX)$. Likewise, also the easy proof that pre- and post-composition with 1-cells in ${\sf Diag}^{\circ}(\CX)$ preserves the $D^{\CX}$-2-cocartesianess of $\al:(F,\varphi)\Lra(F',\varphi')$ proceeds as in (1).
\end{proof}

\begin{rem}\label{split2fib}
We note that the 2-fibration $D_{\CX}^{\op}$ is split, in the obvious sense that the induced functor
$$\Pi_{D_{\CX}^{\op}}:{\SmallCats}^{\op}\to{\LocSmallCats}, \;(F:\CI\to\CJ\text{ in }\SmallCats)\longmapsto(F^*:[\CJ,\CX]^{\op}\to[\CI,\CX]^{\op},\,Y\mapsto YF),$$
is actually a 2-functor. It assigns to a small category $\CI$ its fibre in $({\sf Diag}_{\circ}(\CX))^{\op}$, which is precisely the category $[\CI,\CX]^{\op}$. Furthermore, for all objects $(\CI,X),(\CJ,Y)$ in $({\sf Diag}_{\circ}(\CX))^{\op}$, the fibration
$$({\sf Diag}_{\circ}(\CX))^{\op}(X,Y)\lra[\CI,\CJ],\; (F,\varphi)\mapsto F$$
is actually discrete.
\end{rem}

\section{Appendix 1: Grothendieck fibrations and the  Grothendieck construction}\label{appendix}

\subsection{Cartesian morphisms}
\label{subsec:GrothendieckFibrations}

Given a functor $P\from \CatSymbA{E}\to \CatSymbA{B}$, a morphism $f\from x\to y$ in $\CatSymbA{E}$ is a \Defn{lifting} (along $P$) of a morphism $u:a\to b$ in $\CatSymbA{\CB}$   if $Pf=u$. The lifting $f$ is
\Defn{$P$-cartesian}\footnote{Following \cite{Grothendieck1961}, the older literature, such as \cite{Gray1966}, uses {\em strong cartesian} instead.}  if every diagram of solid arrows below can be filled uniquely, as shown:
 \begin{equation*}
 \xymatrix@R=5ex@C=4em{
 {z}\ar@{-->}[r]_-{t}\ar@/^3ex/[rr]^{h} &
 	{x}\ar[r]_-{f} & {y} &
 	\CatSymbA{E} \ar[d]^{P} \\
 {Pz}\ar[r]_-{w}\ar@/^3ex/[rr]^{Ph} &
 	{Px}\ar[r]_-{Pf} & {Py} & \CatSymbA{B}
 }
 \end{equation*}
Thus, if $h\from z\to y$ in $\CatSymbA{E}$ and $w\from Pz\to Px$ in $\CatSymbA{B}$ satisfy $Pf\cdot w=Pg$, then there is exactly one morphism $t\from z\to x$ in $\CatSymbA{E}$ with $f\cdot t=g$ and $Pt=v$; equivalently, for every object $z$ in $\CatSymbA{E}$, the square
\begin{equation*}
\xymatrix@R=5ex@C=5em{
\CatSymbA{E}(z,x)\ar[r]^{\CatSymbA{E}(z,f)}\ar[d]_{P_{z,x}} &
	\CatSymbA{E}(z,y)\ar[d]^{P_{z,y}}\\
	\CatSymbA{B}(Pz,Px)\ar[r]_{\CatSymbA{B}(P\!z,P\!f)}
		& \CatSymbA{B}(Pz,Py)}
\end{equation*}
is a pullback diagram in $\Sets$. 

We note that, for any functor $P$, a $P$-cartesian lifting $f$ of $u$ is an isomorphism if, and only if, $u$ is an isomorphism. The class ${\rm Cart}(P)$ of $P$-cartesian morphisms in $\CE$ contains all isomorphisms of $\CE$, is closed under composition, and satisfies the cancellation condition $(g\cdot f\in {\rm Cart}(P)\Longrightarrow f\in {\rm Cart}(P))$ whenever $g$ is monic or $P$-cartesian. Moreover, the class ${\rm Cart}(P)$ is stable under those pullbacks in $\CE$ which $P$ transforms into monic pairs; in particular, ${\rm Cart}(P)$ is stable under the pullbacks that are preserved by $P$. 

\subsection{Grothendieck fibrations}\label{Grothendieck fibrations}\label{fibrationbasics}

For an object $b$ in $\CB$ we denote by $\CE_b$ the \Defn{fibre} of $P:\CE\to\CB$ at $b$; this is the (non-full) subcategory of $\CE$ of all morphisms in $\CE$ that are liftings of $1_b$. Hence, for the inclusion functor $J_b:\CE_b\to \CE$ the functor $PJ_b=\Delta b$ is constant. The morphisms in the fibres of $P$ are also called $P$-{\em vertical}.  The functor $P$ is a \Defn{(Grothendieck) fibration} if, for every morphism $u\from a\to b$ in $\CatSymbA{B}$ and every object $y$ in $\CatSymbA{E}_b$, there is a $P$-cartesian lifting $f\from x\to y$ in $\CatSymbA{E}$. Since such a lifting is unique up to isomorphism when considered as an object in the slice category $\CE/y$, one may call $f$ {\em the} $P$-cartesian lifting of $u$ {\em at} $y$. In fact, we will assume throughout that our fibrations are \Defn{cloven}; this means, that a choice of $P$-cartesian liftings, also called a \Defn{cleavage}, has been made for all $u$ and $y$. We denote the chosen $P$-cartesian lifting of $u:a\to b$ in $\CB$ at $y\in\CE_b$ by $\theta^u_y:u^*(y)\to  y$. With this notation one sees immediately that a functor $P$ is a fibration if, and only if, for every object $y$ in $\CE$, the induced functor
$$P_y:\CE/y\longrightarrow\CB/Py$$
of the slice categories has a right adjoint right inverse ({\em rari}), namely $\theta_y$ (see \cite{Gray1966}).

For a fibration $P:\CE\to \CB$, one also calls $\CE$ {\em fibred over} $\CB$.  Every morphism $f:x\to y$ in $\CE$ then has a ($P$-vertical, $P$-cartesian)-factorization, as in

\begin{center}
$\xymatrix{x\ar@{-->}[d]_{\epsilon_f}\ar[rd]^{f} &\\
u^*(y)\ar[r]_{\theta^u_y} & y\;,\\
}$	
\end{center}
where $u=Pf$, and where the $P$-vertical morphism $\ep_f$ is uniquely determined by $f$. In fact, there is, for all morphisms $u:a\to b$ and objects $x\in\CE_a,\, y\in\CE_b$ a natural bijective correspondence
$$\CE_u(x,y)\cong\CE_a(x,u^*(y)),$$
where $\CE_u(x,y)=\CE(x,y)\cap P^{-1}(u)$. This correspondence means precisely that, for all $a\in\CB$, the full embedding
$$ I_a:\CE_a\hookrightarrow a\!\downarrow\! P,\quad x\mapsto (1_a: a\to Px), $$
has a right adjoint; it sends $u:a\to Py$, seen as an object of $a\!\downarrow\!P$, to $u^*y$. (An $I_a$-universal arrow at $u:a\to Py$ is also called a $P$-{\em precartesian} lifting of $u$ and $y$; obviously, when $P$ is a fibration, $P$-precartesian liftings are $P$-cartesian; see \cite{Borceux1994}.) We call $I_a$ the {\em comma insertion} of the fibre $\CE_a$.

The $P$-vertical morphisms and, more generally, the $\CE$-morphisms which are mapped by $P$ to isomorphisms, are orthogonal to $P$-cartesian morphisms. As a consequence one obtains that a functor $P:\CE\to\CB$ is a fibration if, and only if, $P$ is an {\em iso-fibration} (that is:, if every isomorphism $u:a\to b$ in $\CB$ admits a $P$-cartesian lifting at every $y\in\CE_b$), and if
$(P^{-1}({\rm Iso}\,\CB),{\rm Cart}(P))$ is an orthogonal factorization system of $\CE$; the second property means equivalently that $P$ is a {\em Street fibration} \cite{Street1980}. 

For every morphism $u:a\to b$ in $\CB$, the domains of the $P$-cartesian liftings of $u$ at the objects of $\CE_b$ give the object assignment of a functor $u^*:\CE_b\to\CE_a$ that makes
$\theta^u:J_au^*\to J_b$ a natural transformation; for a morphism $j:y\to y'$ in $\CE_b$ one has $u^*(j)=\ep_{j\cdot\theta_y^u}$.
\begin{center}
$\xymatrix{\mathcal E_a\ar[rd]_{J_a}^{\; \theta^u:\Longrightarrow} & & \mathcal E_b\ar[ld]^{J_b}\ar[ll]_{u^*}\\
	& \mathcal E &\\
	}$
	\hfil
	$\xymatrix{u^*(y)\ar[r]^{\theta^u_y}\ar[d]_{u^*(j)} & y\ar[d]^j\\
	u^*(y')\ar[r]^{\theta^u_{y'}} & y'\\
	}$
\end{center}
The commutative diagram below shows that the object assignment $b\mapsto \CE_b$ leads to a pseudofunctor
$$\Phi^P:\CB^{\rm op}\longrightarrow\LocSmallCats,\qquad (u:a\to b\text{ in }\CB)\longmapsto (u^*:\CE_b\to\CE_a)$$
and, thus, presents the fibration $P$ as an {\em indexed category} \cite{JohnstonePare1978}.
\begin{equation*}
 \xymatrix{x\ar[d]_{\ep_f}\ar[rd]^f & &  & & &  \\
 u^{*}(y)\ar[r]_{\theta^u_y}\ar[d]_{u^*(\ep_g)} & y\ar[rrd]^g\ar[d]^{\ep_g} & & & &  \\
 (v\cdot u)^*(z)\cong u^*(v^*(z))\ar[r]^{\qquad\qquad\theta^u_{v^*(z)}}\ar@/_1.5pc/[rrr]_{\theta^{v\cdot u}_z} & v^*(z)\ar[rr]^{\theta^v_z} & & z & y\cong 1_b^*(y)\ar[r]^{\quad\theta^{1_b}_y} & y\\
 a\ar[r]^u & b\ar[rr]^{v\quad} & & c=Pz & b\ar[r]^{1_b\quad} & b\\
 }	
 \end{equation*} 
If $\Phi^P$ is actually a functor, with the above canonical isomorphisms becoming identities, so that
$$(v\cdot u)^*=u^*v^*,\quad(1_b)^*={\rm Id}_{\mathcal E_b}, \quad\text{ and }\quad
\theta^{v\cdot u}=\theta^v\cdot\theta^uv^*,\quad\theta^{1_b}=1_{{\rm Id}_{\mathcal E_b}}$$
for all composable morphisms $u, v$ and objects $b$ in $\CB$, then $P$ is called a {\em split fibration}.

A functor $P:\CE\to\CB$ is {\em small-fibred} if all of its fibres are small; in case of a fibration $P$, this means that $\Phi^P$ takes its values in $\SmallCats$. A fibration $P$ is {\em discrete} if all of its fibres are discrete, that is: if $\Phi^P$ takes its values in $\EuRoman{SET}$. Clearly a functor $P$ is a discrete fibration if, and only if there is, for every $u:a\to b$ in $\CB$ and $y\in \CE_b$, exactly one lifting with codomain $y$; the fibration is necessarily split.

Here is how some elementary properties manifest themselves for a fibration $P:\CE\to\CB$: $P$ is faithful (full; essentially surjective on objects) if, and only if, for every $b\in \CB$, the fibre $\CE_b$ is a preordered class (has all of its homs non-empty; is non-empty, respectively). When $\CE$ has a terminal object, a fibration $P$ is an equivalence of categories if, and only if, $P$ preserves the terminal object and reflects isomorphisms. (In the last statement, the preservation of the terminal object is essential: for a monic arrow $f:x\to y$ in a a category $\CC$, the discrete fibration $f\cdot(-):\CC/x\to\CC/y$ is fully faithful, but does not preserve the terminal object $1_x$ of $\CC/x$, unless $f$ is an isomorphism in $\CC$.)

\subsection{Grothendieck cofibrations and bifibrations}\label{cofib and bifib}
For a functor $P:\CE\to \CB$, a morphism $f:x\to y$ in $\CE$ is {\em $P$-cocartesian} if $f$ is $P^{\rm op}$-cartesian in $\CE^{\rm op}$, with $P^{\op}:\CE^{\op}\to\CB^{\op}$. This means that every solid-arrow diagram below on the left can be filled uniquely as shown.
\begin{center}
$\xymatrix@R=5ex@C=4em{
 {x}\ar[r]_-{f}\ar@/^3ex/[rr]^{h} &
 	{y}\ar@{-->}[r]_-{s} & {z}  \\
 {Px}\ar[r]_-{Pf}\ar@/^3ex/[rr]^{Ph} &
 	{Py}\ar[r]_-{v} & {Pz} 
 }
 \hfil 
 	\xymatrix{x\ar[r]^{\delta^u_x\quad} & u_!(x)=y  & \mathcal E\ar[d]^P\\
 	a=Px\ar[r]^u & b & \mathcal B}
 $
 \end{center}
$P$ is a {\em (cloven Grothendieck) cofibration} if $P^{\rm op}:\CE^{\op}\to\CB^{op}$ is a fibration\footnote{Grothendieck cofibrations are now commonly referred to as {\em opfibrations}: see Footnote 1 of the Introduction.}. 
This means that for every morphism $u:a\to b$ in $\CB$ and every object $x$ in $\CE_a$ one has a (chosen) $P$-cocartesian lifting, which we denote by $\delta_x^u:x\to u_!(x)$; this fixes the {\em cocleavage} $\delta^u:J_a\to J_bu_!$. Every morphism $f:x\to y$ in $\CE$ now admits the ($P$-cocartesian,\,$P$-vertical)-factorization $f=\nu_f\cdot\delta_x^u$, with $u=Pf$. 
\begin{center}
	$\xymatrix{\mathcal E_a\ar[rd]_{J_a}^{\; \delta^u:\Longrightarrow}\ar[rr]^{u_!} & & \mathcal E_b\ar[ld]^{J_b}\\
	& \mathcal E &\\
	}$
	\hfil
	$\xymatrix{& & y\\
	x\ar[rru]^f\ar[rr]^{\delta^u_x} && u_!(x)\ar[u]_{\nu_f}\\
	}$	
	\end{center}	
One obtains a pseudofunctor
$$\Phi_P:\CB\longrightarrow\LocSmallCats,\quad(u: a\to b)\longmapsto(u_!:\CE_a\to\CE_b),$$
and the cofibration $P$ is {\em split} if $\Phi_P$ is a functor; more precisely, if

$$(v\cdot u)_!=v_!\,u_!,\quad (1_b)_!={\rm Id}_{\mathcal E_b},\quad\text{ and }\quad \delta^{v\cdot u}=\delta^vu_!\cdot\delta^u,\quad\delta^{1_b}=1_{{\rm Id}_{\mathcal E_b}},$$
for all composable morphisms $u,v$ and objects $b$ in $\CB$. For all $b\in\CB$ we have the {\em dual comma insertion} of the fibre $\CE_b$:
$$ I^b:\CE_b\hookrightarrow P\!\downarrow\! b,\quad y\mapsto (1_b: Py\to b). $$
When $P$ is a cloven cofibration, it has the left adjoint $(u:Px\to b)\longmapsto u_!(x)$.

A functor $P$ is a {\em bifibration} if it is simultaneously a fibration and a cofibration. The following criterion is certainly known but is not easily found and clearly spelled out in the literature:
\begin{thm}\label{bifibrations}
The following assertions are equivalent for a functor $P:\CE\to\CB$ :
\begin{itemize}
\item[{\em(i)}] $P$ is a bifibration;
\item[{\em (ii)}] $P$ is a fibration, and the functor $u^{\ast}$ has a left adjoint $u_!$, for all $u:a\to b$ in $\CB$;
\item[{\em(iii)}] $P$ is a cofibration, and the functor $u_!$ has a right adjoint $u^{\ast}$, for all  $u:a\to b$ in $\CB$.
\end{itemize}
For a bifibration $P$, the units $\eta^u$ and counits $\varepsilon^u$ of the adjunctions $u_! \dashv u^{\ast}$ are determined by the commutative diagrams
\begin{center}
$\xymatrix{
J_a\ar[d]_{J_a\eta^u}\ar[dr]^{\delta^u} & \\
J_au^{\ast}u_!\ar[r]_{\theta^uu_!} & J_bu_!\\
  }$
 \hfil  
$\xymatrix{& J_b  \\
J_au^{\ast}\ar[r]_{\delta^uu^{\ast}}\ar[ur]^{\theta^u} & J_bu_!u^{\ast}\ar[u]_{J_b\varepsilon^u}\\
        }$
\end{center}                    
\end{thm}

\begin{cor}\label{bifibrationsagain}
For a bifibration $P:\CE\to\CB$ and every morphism $u:a\to b$ in $\CB$, the functor $u^{\ast}:\CE_b\to\CE_a$ preserves all limits and $u_!:\CE_a\to\CE_b$ preserves all colimits.
\end{cor}

\subsection{Limits and colimits in a fibred or cofibred category}

It is certainly known how to form ordinary (co)limits of a specified type in a bifibred category $\CE$ from given (co)limits of the same type in the base category $\CB$ and the fibres of the fibration.
(For notions of, and criteria for, fibrational completeness, see \cite[Section 8.5]{Borceux1994}.) We sketch here a detailed but compact proof of this fact in a more general form, as we trace its steps in our main application in Section 4.

\begin{thm}\label{LiftingLimits}
Let $P\from \CatSymbA{E}\to \CatSymbA{B}$ be a fibration. If limits of shape $\CD$ exists in $\CatSymbA{B}$ and in all fibers of $P$, such that their comma insertions preserve them, then also $\CatSymbA{E}$ has $\CD$-limits, and $P$  preserves them.
\end{thm}

\begin{proof} 
We choose a cleavage $\CleavageOf{}{}$ 
for $P$. For a diagram $F\from \CD\to \CatSymbA{E}$, let $b \DefEq \LimOf{PF}$ in $\CatSymbA{B}$, with limit cone $\beta\from \Delta b\to PF$. For every object $d$ in $\CD$ we have the $P$-cartesian lifting of $\beta_d$ at $Fd$,
\begin{equation*}
(\alpha_d\from Ld\to Fd) \DefEq ( \CleavageOf{\beta_d}{Fd} \from  \beta_{d}^{\ast}(Fd)\to Fd)\ ,
\end{equation*}
to obtain a functor $L\from \CD\to \CatSymbA{E}_b$, together with a natural transformation $\alpha\from  J_bL\to F$. By design, $PJ_bL=\Delta b$ and $P\alpha=\beta$. Now let $z\DefEq \LimOf{L}$, with limit cone $\lambda\from \Delta z\to L$. in $\CE_b$. We claim that the composite transformation
\begin{equation*}
\xymatrix@R=5ex@C=4em{
\Delta z\ar[r]^{J_b\lambda} & J_bL\ar[r]^{\alpha} & F
}
\end{equation*}
is a limit cone in $\CatSymbA{E}$.

Consider any cone $\mu\from \Delta x\to F$ in $\CatSymbA{E}$. Its $P$-image factors as $\beta\cdot\Delta u=P\mu$, for a unique $\CatSymbA{B}$-morphism $u\from Px\to b$. As $\alpha_d$ is $P$-cartesian, for every $d\in \CatSymbB{D}$ one has a morphism $\gamma_d\from x\to Ld$, unique with $\alpha_d\cdot\gamma_d=\mu_d$ and $P\gamma_d=u$.  This gives a cone $\gamma\from \Delta x\to J_bL$ in $\CatSymbA{E}$ with $\alpha\cdot\gamma=\mu$ and $P\gamma=\Delta u$. With the comma insertion $I^b:\CE_b\longrightarrow b\downarrow\! P$ preserving the limit cone $\lambda: \Delta z\to L$, we can view $\gamma$ as a cone $\Delta(u:Px\to b)\longrightarrow I^bL$ in $b\!\downarrow P$ and, hence, factor it uniquely through $I^b\lambda: \Delta I^bz\to I^b L$. This means that there is unique morphism $f:x\to z $ in $\CE$ with $Pf=u$ and $\lambda_d\cdot f = \gamma_d$ for all $d\in\CD$, and we obtain the factorization $(\alpha\cdot J_b\lambda)\cdot\Delta f =\alpha\cdot\gamma=\mu$.


If $g\from x\to z$ is any $\CatSymbA{E}$-morphism  with $(\alpha\cdot J_b\lambda)\cdot\Delta g=\mu$, we must confirm $g=f$. An application of $P$ to the given identity shows $\beta\cdot \Delta Pg=P\mu=\beta\cdot\Delta u$ and, hence, $Pg=u$. Now the $P$-cartesianness of  $\alpha_d$ shows $\lambda_d\cdot g=\mu_d$ for every $d\in\CD$, so that $I_b\lambda\cdot\Delta g=\gamma$. This forces $g=f$.

\end{proof}

An application of the theorem to $P^{\op}$ instead of $P$ produces the dual statement:

\begin{cor}\label{LiftingColimits}
\label{thm:BiFibration-LiftingColimits}
Let $P\from \CatSymbA{E}\to \CatSymbA{B}$ be a cofibration. If colimits of shape $\CD$ exist in $\CatSymbA{B}$ and in all fibers of $P$, such that the dual comma insertions preserve them, then $\CD$-colimits exist in $\CatSymbA{E}$, and $P$ preserves them.
\end{cor}
With Corollary \ref{bifibrationsagain} we conclude from Theorem \ref{LiftingLimits}  and Corollary \ref{LiftingColimits}:

\begin{cor}\label{thm:BiFibration-LiftingColimits}	Let $P:\CE\to\CB$ be a bifibration. If (co)limits of shape $\CD$ exist in $\CB$ and in all fibers of $P$, then also $\CE$ has all (co)limits of shape $\CD$, and $P$ preserves them.
\end{cor}

\subsection{The Grothendieck construction for indexed categories}\label{GrothendieckConstruction}
The indexed category $\Phi^P:\CB^{\rm op}\to\LocSmallCats$ of a fibration $P:\CE\to\CB$ preserves all information about the base category $\CB$ and the fibres $\CE_b=\Phi^P b\;(b\in \CB)$, including their pseudo-functorial interaction. The Grothendieck construction shows how one can rebuild the category $\CE$ from that information. Below (on the left) is the definition of the {\em Grothendieck category} (also called {\em total category}) of $\Phi$, usually denoted by $\int^{\circ}\Phi$, 
in the split (=strict) case, that is, for a genuine functor $\Phi:\CB^{\rm op}\to\LocSmallCats$. 

On the right we describe the dual construction, {\em i.e.}, give the definition of the {\em dual Grothendieck category}, 
$\int_{\circ}\Phi$, for a functor $\Phi:\CB\to\LocSmallCats$. In the case $\Phi=\Phi_P$ where $P:\CE\to\CB$ is a cofibration, it recovers the category $\CE$.
\\
\\
{\renewcommand{\arraystretch}{1.5}%
\begin{tabular}{|lc|cr|}
\hline
\begin{minipage}[t]{6.8cm}
The Grothendieck category $\int^{\circ}\Phi$ of a functor $\Phi\from \CatSymbA{B}^{\rm op}\to \LocSmallCats$ is the category with
\begin{enumerate}[$\bullet$]
\item {\em objects} \ pairs $(b,y)$, for $b\in \CatSymbA{B}$ and $y\in \Phi b$;
\item {\em morphisms} \ $(u,f)\from (a,x)\to (b,y)$, for $u\from a\to b$ in $\CatSymbA{B}$ and $f\from x\to (\Phi u)y$ in $\Phi a$;
\begin{equation*}
\xymatrix@R=2ex@C=5em{
(a,x) \ar[rdd]^-{(u,f)}\ar[dd]_{(\OneMapOn{a},f)} & \\ \\
(a,(\Phi u)y) \ar[r]_-{(u,1_{(\Phi u)y})}  &
	(b,y) \\
a \ar[r]_-{u} &
	b
}
\end{equation*}
\item {\em composition} \ $(v,g)\cdot(u,f)=(v\cdot u, (\Phi u)g\cdot f)$.
\end{enumerate}
$\int^{\circ}\Phi$ is fibred over $\CB$, with split fibration
\begin{center}
$\Pi^{\Phi}:\int^{\circ}\Phi\to\CB,\quad(u,f)\mapsto u,$\\ 
\medskip
$u^*(b,y)=(a,(\Phi u)y),$\\
\medskip
$\theta^u_{(b,y)}=(u,1_{(\Phi u)y}),\quad\varepsilon_{(u,f)}=(1_a,f).$
\end{center}
\end{minipage}
 & & &
\begin{minipage}[t]{6.7cm}
The dual Grothendieck category $\int_{\circ}\Phi$ of a functor $\Phi\from \CatSymbA{B}\to \LocSmallCats$ is the category with
\begin{enumerate}[$\bullet$]
\item {\em objects} \ pairs $(a,x)$, for $a\in \CatSymbA{B}$ and $x\in \Phi a$;
\item {\em morphisms} \ $(u,f)\from (a,x)\to (b,y)$, for $u\from a\to b$ in $\CatSymbA{B}$ and $f\from (\Phi u)x\to y$ in $\Phi b$;
\begin{equation*}
\xymatrix@R=2ex@C=5em{
& (b,y)  \\ \\
(a,x) \ar[r]_-{(u,{\OneMapOn{(\Phi u)x}})}  \ar[ruu]^-{(u,f)} &
	(b,(\Phi u)x)\ar[uu]_{(\OneMapOn{b},f)}	 \\
a \ar[r]_-{u} &
	b
}
\end{equation*}
\item {\em composition} \ $(v,g)\cdot (u,f)=\!(v\cdot u, g\cdot(\Phi v)f).$
\end{enumerate}
$\int_{\circ}\Psi$ is fibred over $\CB$, with split cofibration
\begin{center}
$\Pi_{\Phi}:\int_{\circ}\Phi\to\CB,\quad(u,f)\mapsto u,$\\ 
\medskip
$u_!(a,x)=(b,(\Phi u)x),$\\
\medskip
$\delta^u_{(a,x)}=(u,1_{(\Phi u)x}),\quad\nu_{(u,f)}=(1_b,f).$
\end{center}

\end{minipage}\\
\hline
\end{tabular} }
\\

One can make precise in which sense the construction on the right is dual to the construction on the left, as follows.
Given $\Phi:\CB\to{\sf CAT}$	, dualize the ``base" $\CB$  and every ``fibre'' $\Phi b\;(b\in \CB)$, that is: form the indexed category
\begin{center}
$\xymatrix{\Phi^{\circ}:=[\CB=(\CB^{\rm op})^{\rm op}\ar[r]^{\qquad\;\Phi}& {\sf CAT}\ar[r]^{(-)^{\rm op}} &{\sf CAT}].
}$	
\end{center}
Then there is a trivial bijective functor mapping objects and morphisms identically and making
\begin{center}
$\xymatrix{\int^{\circ}(\Phi^{\circ})\ar[rr]^{\cong\;\;}\ar[rd]_{\Pi^{\Phi^{\circ}}} && (\int_{\circ}\Phi)^{\rm op}\ar[ld]^{(\Pi_{\Phi})^{\rm op}}\\
&\CB^{\rm op}&\\
}$	
\end{center}
commute. (We note that 
there is also the Borceux-Kock dualization of a fibration which dualizes the fibres but not the base,
turning the (vertical, cartesian) factorization for a fibration into a (cartesian, vertical) factorization for a ``dual fibration"; see \cite{Kock2015} for details).

An elementary (and quite obvious) rendition of the equivalence of split fibrations and strictly functorial indexed categories reads as follows; as a 2-categorical equivalence it is formulated as Corollary \ref{GrothendieckEquivalenceDual}  .

\begin{thm}\label{FixedBaseGrothEquivalence}
\label{thm:1-FibS(B)<->Func(B^op,CAT)}
\label{equivalent}
\begin{enumerate}[(i)]
\item For every split fibration $P\from \CatSymbA{E}\to \CatSymbA{B}$ with cleavage $\CleavageOf{}{}$, there is a bijective functor $K^P$, satisfying $PK^P=\Pi^{\Phi^P}$  and preserving the cleavages, given by
\begin{equation*}
\xymatrix@R=0.2ex@C=4em{
\int^{\circ}\Phi^P \ar[r]^-{K^P} \ar[ddd]_{\GrothendieckCFibOf{\Phi^P}} &
	\CatSymbA{E} \ar[ddd]^{P} &
	(b,y) \ar@{|->}[r] & y \\
	&& [(u,f)\from (a,x)\to (b,y)] \ar@{|->}[r] &
			[\CleavageOf{u}{y}\cdot f \from x\to y] \\ \\
\CatSymbA{B} \ar@{=}[r] &
	\CatSymbA{B} 
}
\end{equation*}
\item For every functor $\Phi\from \CatSymbA{B}^{\op}\to \LocSmallCats$, there is a natural isomorphism $\Lambda^{\Phi}:\Phi\to\Phi^{\Pi^{\Phi}}$ whose component at $b\in\CB$ is the bijective functor
\begin{equation*}
\Lambda^{\Phi}_b\from \Phi b \longrightarrow {\textstyle (\int^{\circ}{\Phi})_b},\qquad (y \XRA{f} y') \mapsto [(b,y) \XRA{(1_b,f)} (b,y')]\ .
\end{equation*}
\end{enumerate}
\end{thm}
Under the above dualization principle one concludes from the theorem that split cofibrations correspond equivalently to functors $\Phi:\CB\to\LocSmallCats$. Furthermore,
Theorems \ref{bifibrations}, \ref{LiftingLimits} and Corollaries \ref{LiftingColimits}, \ref{bifibrationsagain} may now be formulated in indexed-category form, as follows.
\begin{cor}
\begin{enumerate}[(1)]
\item A functor $\Phi:\CB^{\op}\to\LocSmallCats$ has the property that every functor $\Phi u$ (with $u$ a morphism in $\CB$) has a left adjoint if, and only if, $\Pi^{\Phi}:\int^{\circ}\Phi\to\CB$ is a bifibration. In that case, if $\CB$ and all categories $\Phi b\;(b\in\CB)$ have (co)limits of a specified diagram type $\CD$, so does $\int^{\circ}\Phi$.
\item A functor $\Phi:\CB\to\LocSmallCats$ has the property that every functor $\Phi u$ (with $u$ a morphism in $\CB$) has a right adjoint if, and only if, $\Pi_{\Phi}:\int_{\circ}\Phi\to\CB$ is a bifibration. In that case, if $\CB$ and all categories $\Phi b\;(b\in\CB)$ have (co)limits of a specified diagram type $\CD$, so does $\int_{\circ}\Phi$.\end{enumerate}	
\end{cor}

\subsection{Standard examples}\label{Standard}
\begin{enumerate}[(1)]
 \item For any category $\CC$, the functors ${\rm Id}:\CC\to\CC$ and $!:\mathcal C\to \EuRoman 1$ (where $\EuRoman 1$ is terminal in $\LocSmallCats$) are split bifibrations. Every morphism in $\CC$ is ${\rm Id}$-(co)cartesian and $!$-vertical; the $!$-(co)cartesian morphisms are the isomorphisms in $\CC$. 
 The indexed categories induced by $\rm Id$ and $!$ have (up to isomorphism) constant value $\EuRoman 1$ and $\CC$, respectively.
 \item For a fixed object $A$ in a category $\CC$, consider its hom-functor $\CC(-,A):\CC^{\op}\to \Sets$ as having discrete-category values. Then $\int^{\circ}\CC(-,A)$ is the slice category $\CC/A$, presented as a discretely-fibred category over $\CC$.

 \item The slice categories of (2) define a functor $\CC/(-):\CC\to\LocSmallCats,\;A\mapsto\CC/A$ whose dual Grothendieck category $\int_{\circ}\CC/(-)$ is the arrow category $\CC^{\EuRoman 2}$ (where the only non-identical morphism in the category $\EuRoman 2$ is $0\to 1$), equipped with its codomain functor ${\rm cod}= \Pi_{\CC/(-)}:\CC^{\EuRoman 2}\to \LocSmallCats$. Hence, ${\rm cod}$ is a split cofibration, and it is a (cloven) fibration precisely when $\mathcal C$ has (chosen) pullbacks. A morphism $(f,u):x\to y$ in $\CC^{\EuRoman 2}$, represented by the commutative square
  
\begin{center}
$\xymatrix{\bullet\ar[r]^{f}\ar[d]_x & \bullet\ar[d]^{y}\\
a\ar[r]^u & b \\
}$	
\end{center}
in $\CC$, is ${\rm cod}$-cocartesian precisely when $f$ is is an isomorphism, and it is ${\rm cod}$-cartesian precisely when it is a pullback diagram in $\CC$.
 \item 	Bijective functors (isomorphisms in $\LocSmallCats$) are split bifibrations. The composite of two (split) fibrations is again a split fibration, and so is any pullback in  $\LocSmallCats$ of a (split) fibration; likewise for (split) cofibrations. 
 \item   A left action of a group $G$ on a group $N$ is described by a homomorphism $\phi:G\to {\rm Aut}(N)$ or, equivalently, by a functor $\phi:G\to\SmallCats$ which maps the only object of $G$ (seen as a category) to $N$ (seen as a category and, hence, as an object in $\SmallCats$). The dual Grothendieck category $\int_{\circ}\phi$ is (up to switching coordinates) precisely the semidirect product $N\rtimes G$. A right action of $G$ on $N$ is given by a functor $G^{\rm op}\to\SmallCats$ with value $N$ or, equivalently by its Grothendieck category $\int^{\circ}\phi$.
 \item There is a functor $\Phi:{\EuRoman{Rng}}^{\op}\to \LocSmallCats$ which assigns to a ring $R$ the category ${\EuRoman{Mod}_R}$ of (left) $R$-modules; every homomorphism $\varphi:R\to S$ gives the functor $\varphi^*:\EuRoman{Mod}_S\to \EuRoman{Mod}_R$ which considers every $S$-module $N$ as an $R$-module, via $ra=\varphi(r)a$ for all $r\in R,\, a\in N$. The category $\int^{\circ}\Phi$ is the category $\EuRoman{Mod}$ of all modules; its objects are pairs $(R,M)$ where $R$ is a ring and $M$ is an $R$-module, and its morphisms $(\varphi,f):(R,M)\to(S,N)$ are given by a morphism $\varphi:R\to S$ in $\EuRoman{Rng}$ and an $R$-linear map $f:M\to\varphi^*(N)$. The projection $\Pi^{\Phi}:\EuRoman{Mod}\to\EuRoman{Rng}$ is a split fibration.
 \item The functor $\CO:\Sets\to\SmallCats$ assigns to every set $X$ the set of topologies on $X$, ordered by $\supseteq$ and, as such, considered as a small category; for a map $f:X\to Y$ one has the monotone map $f^*:\CO(Y)\to\CO(X)$, defined by taking inverse images. The Grothendieck category $\int^{\circ}\CO $ is the category $\EuRoman{Top}$ of topological spaces, with underlying $\Sets$-functor $\Pi^{\CO}$, which is in fact a split bifibration. More generally, one may characterize topological functors with small fibres (see \cite{AHS1990}) as those fibrations $P$ for which the indexed category $\Phi^P$ takes values in the category of complete lattices and their ${\rm inf}$-preserving maps (see \cite{Wyler1971, Tholen1974}), making $P$ in fact a split bifibration.
 \item Considering the functor 
 ${\rm Id}_{\Sets}: {\EuRoman{ Set}}\to{\EuRoman{ Set}}$ as having discrete-category values, one obtains the category $\Sets_{\bullet}$ of pointed sets as its dual Grothendieck category $\int_{\circ}{\rm Id}_{\EuRoman{ Set}}$. Writing the domain of ${\rm Id}_{\Sets}$ in the form $({\Sets}^{\rm op})^{\rm op}$ we also have the Grothendieck category 
  $\int^{\circ}{\rm Id}_{\Sets}$, which is precisely the category ${\Sets}^{\op}_{\bullet}$.
 \end{enumerate}
 The ``categorification'' of the last (rather trivial) example leads to an important fact, which we describe next.
 
 \subsection{The classifying split (co)fibration}
 The dual Grothendieck category $\int_{\circ}{\rm Id}_{\SmallCats}$  of the functor
  ${\rm Id}_{\SmallCats}: \SmallCats\to{\SmallCats}$ (with its codomain to be embedded into $\LocSmallCats$) is the category $\SmallCats_{\bullet}$ of {\em small lax-pointed categories}. Its objects $(\CC,x)$ are given by a small category $\CC$ equipped with an object $x\in\CC$, and a morphism 
$(F,f):(\CC,x)\to(\CD,y)$ is given by a functor $F:\CC\to\CD$ and a morphism $f:Fx\to y$ in $\CD$. The forgetful functor $\Pi_{\bullet}:=\Pi_{{\rm Id}_{\SmallCats}}  :\SmallCats_{\bullet}\to\SmallCats$ is a small-fibred split cofibration, called {\em classifying}, since one has the following rather obvious fact:
\begin{thm}\label{classifying cofib}
Every small-fibrerd split cofibration is a pullback	(in $\LocSmallCats$) of the classifying split cofibration, as shown in the diagram
\begin{center}
$\xymatrix{\CE\ar[r]\ar[d]_P  & {\SmallCats}_{\bullet}\ar[d]^{\Pi_{\bullet}}\\
\CB\ar[r]_{\Phi_P\;\;} &{\SmallCats   } \\
}$
\hfil
$[f\!:\!x\to y]\longmapsto[((Pf)_!,\nu_f)\!:\!(\CE_{Px},\!x)\to(\CE_{Py},\!y)]\,.	$
\end{center}
\end{thm}

Writing the domain of ${\rm Id}_{\SmallCats}$ as $(\SmallCats^{\op})^{\op}$, we can also form the Grothendieck category $\int^{\circ}{\rm Id}_{\SmallCats}$, which is the category $\SmallCats^{\bullet}$ of {\em small oplax-pointed categories}. Its objects are the same as those of $\SmallCats_{\bullet}$, but its morphisms $(F,f):(\CC,x)\to(\CD,y)$ are now given by functors $F:\CD\to\CC$ equipped with a morphism $f:x\to Fy$ in $\CC$. The forgetful functor $\Pi^{\bullet}:=\Pi^{{\rm Id}_{\SmallCats}}:\SmallCats^{\bullet}\to\SmallCats^{\op}$ classifies the small-fibred split fibrations:
\begin{cor}\label{classifying fib}
	Every small-fibred split fibration is a pullback (in $\LocSmallCats$) of the classifying split fibration, as shown in the diagram
\begin{center}
$\xymatrix{\CE\ar[r]\ar[d]_P  & \SmallCats^{\bullet}\ar[d]^{\Pi^{\bullet}}\\
\CB\ar[r]_{(\Phi^P)^{\op}\;\;} &{\SmallCats}^{\rm op} \\
}$
\hfil
$[f\!:\!x\to y]\longmapsto[((Pf)^*,\nu_f)\!:\!(\CE_{Px},\!x)\to(\CE_{Py},\!y)]\,.	$
\end{center}
\end{cor}

Of course, nothing prevents us from dropping the restriction of $P$ being small-fibred.: Theorem \ref{classifying cofib} and Corollary \ref{classifying fib} remain true {\em verbatim} if we delete ``small-fibred'' and replace $\SmallCats,\SmallCats_{\bullet},\SmallCats^{\bullet}$ by $\LocSmallCats,
\LocSmallCats_{\bullet}, \LocSmallCats^{\bullet}$, respectively, and $\LocSmallCats  $ by the {\em colossal category} $\mathbb{CAT}$ which contains $\LocSmallCats$ as an object, with the last exchange only formally needed for the provision of a legitimate home of the amended pullback diagrams.

\section{Appendix 2: Regular epimorphisms in $\LocSmallCats$ and a confinality criterion}
\subsection{Characterizing regular epimorphisms in $\LocSmallCats$}

First we characterize regular epimorphisms in $\LocSmallCats$, {\em i.e.}, functors $F:\CC\to\CD$ which serve as coequalizers of parallel pairs of functors, in terms of suitable relations on the classes of objects and morphisms of their domains. Although parts of the characterization we give may be found, at least in some similar form, in the older literature (such as \cite{Ehresmann1965, Mersch1965, Tholen1974}), we give here a compact description of it, also since these sources may not be easily accessible to the reader. 

We begin by describing the quotient of a category $\CC$ with respect to a suitable pair $(\approx,\sim)$ of relations. The ``object branch'' of the pair is a mere equivalence relation $\approx$ on the class of objects of $\CC$; we denote the $\approx$-equivalence class of $c\in\CC$ by $[c]$. These equivalence classes are the vertices of the directed graph $\CC_{\approx}$; an edge $u:[c]\to[d]$ in $\CC_{\approx}$ is simply a $\CC$-morphism $u:a\to b$ with $c\approx a$ and $b\approx d$. We denote by
$$ \Omega(\CC_{\approx})$$
the category\footnote{It is important to note that the category $\Omega(\CC_{\approx})$ will generally \emph{fail} to have small hom-sets. Indeed, since the equivalence classes $[c]$ may be large, the graph $\CC_{\approx}$ may have a proper class of edges between two fixed vertices, causing its generated category to have large hom-sets. As Remark \ref{totalizer} shows, an unwelcome consequence of this fact is that, unlike $\SmallCats$,   the (huge) category of $\LocSmallCats $ of locally small categories fails to have coequalizers.} freely generated by the graph $\CC_{\approx}$; hence, its objects are the vertices of the graph, and its morphisms $(u_n, ..., u_1):[c]\to[d]$ are given by strings of $\CC$-morphisms
\begin{equation*}
c\approx a_1 \xrightarrow{\ u_1\ } b_1 \approx a_2 \xrightarrow{\ u_2\ }b_2\approx  \cdots  \xrightarrow{\ u_n\ } b_n \approx d,
\end{equation*}
including the empty strings $():[c]\to[c]$ (for $n=0$). We may write the string $(u)$ of length 1 simply as $u$; composition in $\Omega(\CC_{\approx})$ proceeds by juxtaposition. Hence, in $\Omega(\CC_{\approx})$ the morphism $(u_n, ..., u_1)$ is the composite $u_n...u_1$ of strings of length 1.  

The ``morphism branch'' of our pair is simply a relation $\sim$ on the class of morphisms of $\Omega(\CC_{\approx})$ which allows $\alpha\sim\beta$ only if $\alpha$ and $\beta$ belong to the same hom-set of $\Omega(\CC_{\approx})$. 
We then consider the least equivalence relation $\simeq$ on the class of morphisms of $\Omega(\CC_{\approx})$ containing $\sim$ and 
satisfying the following compatibility conditions for all morphisms $u:c\to d, v:d\to e$ in $\CC$ (whose composite in $\CC$ is denoted by $v\cdot u$), and all $\alpha:[a]\to[b],\;\beta,\beta':[b]\to[c],\;\gamma:[c]\to[k]$ in $\Omega(\CC_{\approx})$:
\begin{enumerate}
\item $1_c\;\simeq\;():[c]\to[c]\;;$
\item $v\cdot u\;\simeq\;(v,u):[c]\to[e]\;;$
\item $\beta\;\simeq\;\beta'\;\;\Longrightarrow\;\;\gamma\beta\alpha\;\simeq\;\gamma\beta'\alpha$\;.
\end{enumerate}
There is then only one way of making the $\simeq$-equivalence classes of morphisms in $\Omega(\CC_{\approx})$ the morphisms of a category $\Omega(\CC_{\approx})/\!\simeq $ with the same objects as $\Omega(\CC_{\approx})$ , so that the projection $\Omega(\CC_{\approx})\to  \Omega(\CC_{\approx})/\!\simeq $ that maps objects identically and assigns to morphisms their  $\simeq$-equivalence classes, becomes a functor. We note that, by condition 1, the relation $\approx$ is actually determined by $\sim$, since 
$$a\approx b\iff 1_a\simeq 1_b,$$
for all objects $a,b\in\CC$. Moreover, because of conditions 1 and 2, the composite graph morphism
$$P:\CC\to \Omega(\CC_{\approx})\to\Omega(\CC_{\approx})/\!\simeq,\quad (u:a\to b)\mapsto (u:[a]\to[b])\mapsto (Pu:[a]\to [b]),$$
becomes a functor, where $Pu$ is the $\simeq$-equivalence class of the morphism $u:[a]\to[b]$ in $\Omega(\CC_{\approx})$. 

\medskip

We call $P$ the {\em projection} of the {\em quotient structure} on $\CC$ as given by the relations $\approx$ and $\sim$. The quotient structure is \emph{locally small} if the category $\Omega(\CC_{\approx})/\simeq$ has small hom-sets.
We now prove that the projection of a locally small quotient structure is the prototypical regular epimorphism in $\LocSmallCats$.

\begin{prop}\label{quotient}  The following assertions for a functor $F:\CC\to\CD$ are equivalent:
\begin{enumerate}[(i)]
\item $F$ is a regular epimorphism in $\LocSmallCats$;
\item every morphism in $\CD$ may be written as a composite morphism $Fu_n\cdot...\cdot Fu_1$ with morphisms $u_1,..., u_n$ in $\CC,\; n\geq 1$; furthermore, if $Fu_n\cdot...\cdot Fu_1=Fv_m\cdot...\cdot Fv_1$ with $v_1,..., v_m$ in $\CC,\; m\geq 1$, then also $Gu_n\cdot...\cdot Gu_1=Gv_m\cdot...\cdot Gv_1$, for any functor $G:\CC\to\CB$ satisfying the condition $(Fu=Fv\Longrightarrow Gu=Gv)$ for all morphisms $u,v$ in $\CC$\;;
\item there is a locally small quotient structure $(\approx,\sim)$ on $\CC$ with projection $P$ and a bijective functor $J:\Omega(\CC_{\approx})/\!\simeq\;
\longrightarrow \CD$, such that $F=JP$.
\end{enumerate}
\end{prop}
\begin{proof}
(iii)$\Rightarrow$(ii): Every morphism in $\Omega(\CC_{\approx})$ has the form $Pu_n\cdot...\cdot Pu_1$, including $P(\;()\!:\![c]\to[c]\;)	=P1_c$\;), which confirms the first statement of (ii), by bijectivity and functoriality of $J$. For any functor $G:\CC\to\CB$ satisfying $(Fu=Fv\Longrightarrow Gu=Gv)$ for all morphisms $u,v$ in $\CC$, one has in particular $(a\approx b\Longrightarrow P1_a=P1_b\Longrightarrow Ga=Gb)$ for all objects $a,b\in\CC$. Consequently, $G$ induces a morphism $\CC_{\approx}\to\CB$ of graphs, which extends uniquely to a functor $\overline{G}:\Omega(\CC_{\approx})\to\CB$, via $\overline{G}(u_n...u_1)=Gu_n\cdot...\cdot Gu_1$. By functoriality of $G$ and $\overline{G}$, the equivalence relation $\sim_G$ on $\Omega(\CC_{\approx})$ with $(\alpha\sim_G \beta\iff \overline{G}\alpha=\overline{G}\beta)$ satisfies conditions 1-3, and it contains $\sim$ by hypothesis on $G$:
$$ u\sim v\Longrightarrow Pu=Pv\Longrightarrow Fu=Fv\Longrightarrow Gu=Gv\;.$$
Hence, $\simeq$ is contained in $\sim_G$, making $\overline{G}$ factor through $\Omega(\CC_{\approx})/\!\simeq$, by a unique functor $\widetilde{G}:\Omega(\CC_{\approx})/\!\simeq\to\CB$.
Consequently, assuming $Fu_n\cdot...\cdot Fu_1=Fv_m\cdot...\cdot Fv_1$, one first has $Pu_n\cdot...\cdot Pu_1=Pv_m\cdot...\cdot Pv_1$ and therefore $u_n...u_1\simeq v_m...v_1$. Now the application of $\widetilde{G}$ to the $\simeq$-equivalence classes of these morphisms gives the desired equality $Gu_n\cdot...\cdot Gu_1=Gv_m\cdot...\cdot Gv_1$.

(ii)$\Rightarrow$(i): The conditions given allow one to show routinely that any functor $G:\CC\to \CB$ that coincides on the projections $\CC\times_{\CB}\CC\rightrightarrows C$ of the kernel pair of $F$ factors uniquely through $F$, so that $F$ is in fact the coequalizer of its kernel pair.

(i)$\Rightarrow$(iii): As a regular epimorphism, $F$ is the coequalizer of its kernel pair. Let $\approx_F\,=\,\approx$ be the equivalence relation induced by $F$ on the class of objects of $\CC$. As in (iii)$\Rightarrow$(i), with $G$ traded for $F$, one obtains the functor $\overline{F}:\Omega(\CC_{\approx})\to\CB$. The equivalence relation $\sim_F\,=\,\simeq$ induced by $\overline{F}$ via $(\alpha\sim\beta\iff \overline{F}\alpha=\overline{F}\beta)$ actually satisfies the conditions 1-3, so that $\sim$ equals $\simeq$. Hence, in the notation of (iii)$\Rightarrow$(ii), $F$ factors through $P$, via $J:=\widetilde{F} :\Omega(\CC_{\approx})/\!\sim\;\longrightarrow\CB$. Since for every morphism $(u,v)$ in $\CC\times_{\CD}\CC$ one has $Fu=Fv$, the functor $P$ coincides on the projections of the kernel pair of $F$ and must therefore factor through the coequalizer $F$, by what then turns out to be the inverse of $J$, thanks to the uniqueness part of the universal property of a coequalizer. The bijection $J$ makes the category $\Omega(\CC_{\approx})/\!\sim$ inherit the local smallness from the given category $\CD$, so that the quotient structure $(\approx,\sim)$ is indeed locally small.
\end{proof}

\begin{rem}\label{totalizer}
(1) As a quotient structure $(\approx,\sim)$ on any given category $\CC$, one may make the extreme choice and let $\approx$ be the all-relation on the class of objects of $\CC$ and $\sim$ be empty, so that $\simeq$ will be the least relation on the class of morphisms of $\Omega(\CC_{\approx})$ satisfying the conditions 1-3 above. Then $T(\CC)=\Omega(\CC_{\approx})/\!\simeq$ has been called the {\em totalizer} of (the partially defined composition rule for the morphisms of) $\CC$; its projection $P:\CC\to T(\CC)$ serves as a reflection of $\CC$ into the (colossal) category $\EuRoman{MON}$ of large monoids and their homomorphisms. Even for finite $\CC$, $T(\CC)$ may be infinite: $T(\{0\to 1\})\cong (\mathbb N,+)$.  Remarkably, as shown in \cite{Borger1977}, while identifying all objects, $P$ preserves the morphism structure of $\CC$ to the largest extent possible: for any morphism $f,g$ in $\CC$, one has $Pf=Pg$ if, and only if, $f=g$ or both, $f$ and $g$, are identity morphisms.

(2) The classical (and most important) instance of the formation of a quotient category arrises when one forms the {\em category of fractions} $\CC[\CS^{-1}]$ for a class $\CS$ of morphisms in a category $\CC$: formally adjoin to $\CC$ a designated inverse $s'$ of each morphism $s$ in $\CS$ and then impose the relations that insure that $s'$ becomes an actual inverse of $s$ in $\CC[\CS^{-1}]$. Again, it is important to note that, local smallness of $\CC$ does not guarantee at all local smallness of  $\CC[\CS^{-1}]$.
\end{rem}

 \subsection{Confinality vs. conneted fibres for regular epimorphisms and fibrations} 
 
We begin with a sufficient condition for the recognition of confinal\footnote{In order to avoid ambiguity, we follow Gabriel and Ulmer \cite{GabrielUlmer1971} (Definition 2.12) and use the traditional term \emph{confinal} for functors $F:\CC\to\CD$ for which the comma categories $z\downarrow F\;(z\in\CD)$ are (non-empty and) connected. More recent authors use ``cofinal'', ``final'', or even ``initial'' for the same notion.}  regular epimorphisms in $\LocSmallCats$. We did not find this criterion in the literature.

\begin{thm}\label{quotientscofinal}
A regular epimorphism in $\LocSmallCats$ is confinal if each of its fibres is a connected category.
\end{thm}
\begin{proof}
	
	By Proposition \ref{quotient} we may assume that the given regular epimorphism is the projection $P$ of a quotient structure $(\approx,\sim)$ on $\CC$ and, when $P$ has connected fibres, must show that, for every object $c\in \CC$, the comma category $[c]\downarrow P$ is non-empty and connected. As we have the object  $(1_{[c]}=P1_c,\;c)$ in $[c]\downarrow P$, it suffices to show that every object $(\langle\alpha\rangle, d)$ in $[c]\downarrow P$ is connected to $(P1_c,c)$ by a zig-zag of morphisms in $[c]\downarrow P$;
	here $\langle\alpha\rangle$ denotes the $\simeq$-equivalence class of a morphism $\alpha=(u_n,...,u_1):[c]\to[d]$ in $\Omega(\CC_{\approx})$, given by 
	$\CC$-morphisms
\begin{equation*}
c\approx a_1 \xrightarrow{\ u_1\ } b_1 \approx a_2 \xrightarrow{\ u_2\ }b_2\approx  \cdots  \xrightarrow{\ u_n\ } b_n \approx d\;.
\end{equation*}
Then $u_1$ becomes a morphism $(1_{Pc},\;a_1)\longrightarrow (Pu_1,\;b_1)$ in $[c]\downarrow P$, and every $u_i$ with $i\geq 2$ gives a morphism $(Pu_{i-1}\cdot...\cdot  Pu_1,\;a_i)\longrightarrow (Pu_i\cdot ...\cdot Pu_1,\; b_i)$, as shown in
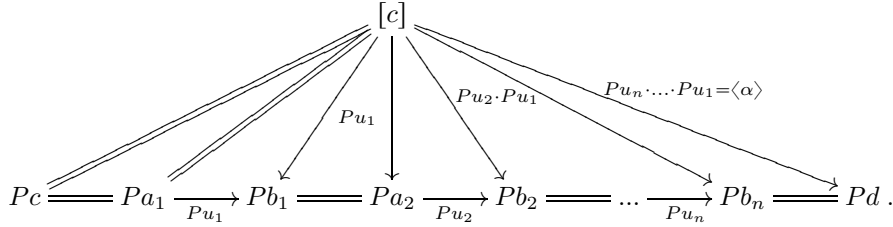
\begin{center}
$\xymatrix{&&& [c]\ar@{=}[llldd]\ar@{=}[lldd]\ar[ldd]^{Pu_1} \ar[dd] \ar[rdd]^{\!\!Pu_2 \cdot Pu_1\qquad}\ar[rrrdd]\ar[rrrrdd]^{\quad Pu_n\cdot...\cdot Pu_1=\langle\alpha\rangle} &&&&\\
&&&&&&&&\\
Pc\ar@{=}[r] & Pa_1\ar[r]_{Pu_1} & Pb_1\ar@{=}[r] &Pa_2\ar[r]_{Pu_2} &Pb_2\ar@{=}[r] & ...\ar[r]_{Pu_n}&  Pb_n\ar@{=}[r] & Pd\;. \\
}$
\end{center}
It therefore suffices to show that the horizontal equalities of the diagram may be replaced by fitting zig-zags of morphisms in $[c]\downarrow P$. In fact, it is easy to see that any two objects $(t, a), (t,b)$ in $[c]\downarrow P$ are connected by a zig-zag of morphisms in that category; it is provided by any zig-zag connecting the objects $a,b$ in the connected fibre of $P$ at $Pa=Pb$. This completes the proof.
\end{proof}

A confinal regular epimorphism in $\sf Cat$ (or $\sf CAT$) may have disconnected fibres. In deed, if we consider the category $\{a\to c\leftarrow b\}$ and identify its two non-identical arrows, we obtain the category $\sf 2=\{0\to 1\}$ as its quotient category; the projection is confinal, but its fibre at 0 is the discrete category $\{a,b\}$.
 
For fibrations, rather than regular epimorphisms, we obtain a stronger recognition criterion for confinality:

\begin{prop}
A fibration is confinal if, and only if, all of its fibres are non-empty and connected.
\end{prop}
\begin{proof}
Given a fibration $P:\CE\to\CB$ and an object $a\in \CB$, the embedding $\CE_a\hookrightarrow a\!\downarrow\!P$ is full and has a right adjoint, by \ref{fibrationbasics}. Since $P$ is confinal if and only if every category $a\!\downarrow\!P$ is non-empty and connected, the claim follows from the fact that a reflective or coreflective subcategory is connected if, and only if, its parent category is connected.
\end{proof}

\vfill


\end{document}